\documentclass[table,x11names]{amsart} 
\usepackage{va, styell}

\usepackage{bm}
\usepackage{subcaption}
\usepackage{textcomp}
\usepackage{array,multirow} 
\usepackage{float}
\usepackage{adjustbox}

\usepackage{tikz}
\usetikzlibrary{arrows} 
\usetikzlibrary{cd}

\usepackage[pdftex]{hyperref}
\hypersetup{colorlinks,citecolor=blue,linktocpage,hyperindex=true,backref=true}
\allowdisplaybreaks

\usepackage{textcomp}

\usepackage{etoolbox}

\makeatletter
\pretocmd{\section}{\addtocontents{toc}{\protect\addvspace{5\p@}}}{}{}
\makeatother
\newcommand\cme{\cM_e}
\newcommand\ume{\uM_e}

\usepackage{todonotes}

\date{July 5, 2021}
\title[Compactification of moduli of elliptic K3 surfaces]
{Compactifications of moduli of elliptic\\ K3 surfaces: stable pair
  and toroidal} 
\author{Valery Alexeev}
\email{valery@uga.edu}
\address{Department of Mathematics, University of Georgia, Athens GA
  30602, USA}
\author{Adrian Brunyate}
\email{brunyate@gmail.com}
\author{Philip Engel}
\email{philip.engel@uga.edu}
\address{Department of Mathematics, University of Georgia, Athens GA
  30602, USA}
\begin{document}
\numberwithin{equation}{section}

\begin{abstract}
  We describe two geometrically meaningful compactifications of the moduli space of elliptic K3 surfaces via stable slc pairs, for two different choices of a polarizing divisor, and show that their normalizations are two different toroidal compactifications of the moduli space, one for the ramification divisor and another for the rational curve divisor. 

In the course of the proof, we further develop the theory of integral affine spheres with 24 singularities. 
We also construct moduli of rational (generalized) elliptic stable slc surfaces of types ${\bf A_n}$ ($n\ge1$), ${\bf C_n}$ ($n\ge0$) and ${\bf E_n}$ ($n\ge0$).
\end{abstract}

\maketitle 
\tableofcontents
\listoffigures
\listoftables


\section{Introduction}
\label{sec:intro}

It is well known \cite{mumford1972an-analytic-construction,
  namikawa1976a-new-compactification-of-the-siegel1,
  alexeev1999on-mumfords-construction,
  alexeev2002complete-moduli}
that there exists a functorial, geometrically meaningful
compactification of the moduli space of principally polarized abelian
varieties $A_g$ via stable pairs whose normalization is a
distinguished toroidal compactification $\oA_g^{\rm vor}$ for the 2nd
Voronoi fan. Finding analogous compactifications for moduli spaces
of K3 surfaces is a major problem that guided and motivated a lot of
research in the last twenty years. Here, we solve this problem in the
case of elliptic K3 surfaces, and in two different ways.

The moduli space of stable pairs provides a geometrically meaningful
compactification $\oP_{2d, n}$ for the moduli space $P_{2d, n}$ of
pairs $(X,\epsilon R)$, where $X$ is a K3 surface with $ADE$
singularities, $L$ a primitive ample polarization of degree $L^2=2d$,
and $R\in |nL|$ an effective divisor. We recall this construction in
Section~\ref{sec:complete-moduli}.

Let $F$ be a moduli space of K3 surfaces with lattice
polarization $\bM\subset\Pic X$. The most common example is the
moduli space $F_{2d}$ of primitively polarized K3 surfaces $(X,L)$ of
degree $L^2=2d$; here $\bM = \bZ h$ with $h^2=2d$. The main subject of
this paper is $F=F_\el$, the moduli space of K3 surfaces polarized by
the standard rank 2 even unimodular lattice $H=\II_{1,1}$, with a
choice of vectors $s,f$ such that $s^2=-2$, $f^2=0$, $s\cdot f=1$. 
Choosing the marking appropriately,
these are elliptic surfaces $X\to\bP^1$ with a section $s$ and fiber $f$.

Pick a vector $h\in \bM$ with $h^2=2d>0$ representing an ample line
bundle $L$ on a generic surface in $F$.
Next, if possible, make a
\emph{canonical choice} of an effective divisor $R\in |nL|$ for all
the surfaces in $F$. This gives an embedding $F\into P_{2d,n}$.  Let
$\oF^\slc$ be the closure of $F$ in $\oP_{2d,n}$,
taken with the reduced scheme structure. This is a projective
variety. We are interested in whether this compactification
can be described explicitly, and which stable pairs $(X,\epsilon R)$ appear
over the boundary.

Since $F=\mathbb{D}/G$ is an arithmetic
quotient of a Hermitian symmetric domain of type IV, it is natural to
ask if $\oF^\slc$ is related to a toroidal compactification
$\overline{\mathbb{D}/G}^\tor$
of \cite{ash1975smooth-compactifications}
for some choices of admissible fans at
the 0-cusps of the Baily-Borel compactification. For $F=F_\el$ there is
only one 0-cusp. So the combinatorial data is a $\Gamma$-invariant
fan: a rational polyhedral decomposition of the rational
closure $\oC_\Q$ of the positive cone in $\IIell\otimes \R$ which
is invariant under the group
$\Gamma=O^+(\IIell)$ of isometries of the even unimodular
lattice of signature $(1,17)$. There is a very natural choice of fan because
$\Gamma$ contains an index $2$ subgroup generated by reflections
and we may take the fan to be the $\Gamma$-orbit of the Coxeter chamber.

There are many natural choices of a polarizing divisor for
$F$. One comes from the embedding of $F$ into $F_2$ as the
unigonal divisor. Every K3 surface of degree~2 comes with a canonical
involution. For a generic surface the quotient $X/\bZ_2$ is
$\bP^2$. The surfaces $X$ in the unigonal divisor have an $A_1$
singularity, which upon being resolved becomes the section $s$
of an elliptic fibration, and the double cover $X\to \bP(1,1,4)$
is the elliptic involution. Thus the ramification divisor $R$ is the trisection
of nontrivial $2$-torsion points on the fiber.
It is absolutely canonical and one checks that $R\in |3(s+2f)|$.
We denote the corresponding stable pair
compactification by $\slcram$. In Section~\ref{sec:ramification-divisor} we
derive the description of
$\slcram$ and the surfaces appearing on the boundary from
\cite{alexeev2019stable-pair}, where we solved the analogous problem
for the larger space $\oF_2^\slc$.

\begin{theorem}\label{intro1} The normalization of $\slcram$ is the toroidal compactification
associated to the $\Gamma$-orbit of one chamber, formed from the union of $4$ Coxeter chambers.
\end{theorem}

Another natural choice of polarizing divisor is
$R=s+ m \sum_{i=1}^{24}f_i$, where $s$ is the section and $f_i$ are
the 24 singular fibers of the elliptic fibration, counted with
multiplicities. Here, any
$m\ge1$ gives the same result. We denote the stable pair
compactification for this choice by $\slcrc$ where ``rc'' stands
for ``rational curves''.

The reason for this notation is the following. It was observed by Sean Keel
about 15 years ago that for a generic K3 surface $(X,L)$ with a
primitive polarization the sum $R=\sum C_i$ of the singular rational
curves $C_i\in |L_i|$, counted with appropriate multiplicities, is a
canonical polarizing divisor.  Their number $n_d$ is given by the
Yau-Zaslow formula.  Our space $F$ embeds into each $F_{2d}$ with the
class of $L$ equal to $s+(d+1)f$. On such an elliptic K3 surface, each
curve $C_i$ specializes to a sum of the section $s$ and $d+1$ singular
fibers~$f_i$, cf. \cite{bryan2000enumerative-geometry}. It follows
that
\begin{displaymath}
R \equiv n_d\  \big( s + \frac{d+1}{24} \sum_{i=1}^{24} f_i \big),
  \text{ which is proportional to }s+ m\sum_{i=1}^{24} f_i.
\end{displaymath}
Stable surfaces appearing on the boundary of $\slcrc$
were described in \cite{brunyate15modular-compactification}, its
normalization was conjectured to be toroidal, and the hypothetical fan
was described. We prove this conjecture:

\begin{theorem}\label{intro2} The normalization of $\slcrc$ is the toroidal compactification
associated to the $\Gamma$-orbit of a subdivision of the Coxeter chamber into $9$ sub-chambers.
\end{theorem}

Modular compactifications of elliptic surfaces have attracted a lot of
attention recently.  The papers of Ascher-Bejleri
\cite{ascher2017log-canonical, ascher2019moduli-fibered,
  ascher2017moduli-weighted}, using twisted stable maps, construct
compactifications for the moduli spaces of elliptic fibration pairs
$(X\to C, s + \sum a_i g_i)$, where $g_i$ are some fibers, both
singular and nonsingular, and $0 \le a_i \le 1$.  The paper
\cite{ascher2019compact-moduli} considers the case when $X$ is an
elliptic K3 and shows that the moduli space for
$(X, s + \sum_{i=1}^{24} \epsilon f_i)$, where $f_i$ are the singular
fibers, is isomorphic to the normalization of our $\oF^\rc$, although
the stable pairs are different, as
we consider the divisor $\epsilon s +m\epsilon \sum_{i=1}^{24} f_i$.
Inchiostro 
\cite{inchiostro20moduli-weierstrass} considers pairs with
arbitrary coefficients $(X, a_0 s + \sum a_ig_i)$, where $g_i$ are
some fibers, and it includes the case of small $a_0,a_i$.  
We not that when $a_0$ is not small, the underlying surface $X$ may be
only quasi-elliptic, with the contracted section.
The
connection to toroidal compactifications was not considered in the above
papers. 

We also note an interesting recent preprint \cite{odaka2020pl-density}
that appeared after our paper, where our classification of
degenerations of elliptic surfaces into unions of ${\bf ACE}$ surfaces is
explored from a differential geometric viewpoint.

The general approach of this paper continues the program developed in
\cite{engel2018looijenga, engel2021smoothings, alexeev2019stable-pair}
to understand degenerations of (log) Calabi-Yau surfaces via
integral-affine structures on the two-sphere. 
It complements the works of Kontsevich-Soibelman
\cite{kontsevich2006affine-structures} and Gross, Siebert, Hacking,
Keel \cite{gross2003affine-manifolds,
  gross2015mirror-symmetry-for-log, gross2016theta-functions} which
discovered the relevance of integral-affine structures to
understanding mirror symmetry for Calabi-Yau degenerations.

The main new technical tool is explained in Section~\ref{sec:method},
where we give a general criterion for when the normalization of a
stable pair compactification of K3 moduli is toroidal.


The fans of Theorems \ref{intro1} and \ref{intro2} are described in
Section~\ref{sec:fans}. Background on integral-affine structures and
degenerations of K3 surfaces is given in Section \ref{sec:ia-pairs}.
The main theorems are proved in
Sections~\ref{sec:ramification-divisor} and
\ref{sec:rat-curves-divisor}. 
Throughout, we work over $\bC$.

\begin{acknowledgements}
  The first author was partially supported by NSF under DMS-1902157
  and the second author under DMS-1503062.
\end{acknowledgements}

\section{Basic notions}
\label{sec:basics}

We use \cite{alexeev2019stable-pair} as a general reference for many
of the basic definitions and results, including the definition of semi
log canonical (slc) singularities, and define here the most important notions.

\subsection{Models for degenerations of K3 surfaces}
\label{sec:models}

We review several models for degenerations of K3 surfaces and name
them.  For a family $\pi\colon X\to S$ and two line
bundles $L_1,L_2$ on $X$, we write $L_1\simeq_S L_2$ if $L_1\otimes
L_2\inv = \pi^* F$ for some line bundle on $S$.
Below, $C$ is a smooth curve with a point $0$, and $C^*=C\setminus 0$.

\begin{definition}\label{def:kulikov}
  Let $X^*\to C^*$ be a flat family in which every fiber is a
  smooth K3 surface. A \emph{Kulikov model} is a proper analytic completion
  $X\to C$ such that $X$ is smooth, the central fiber
  $X_0$ is a reduced normal crossing divisor, and $K_{X}\sim_C 0$. 
  We say that the Kulikov model is Type I, II, or III depending on whether
  $X_0$ is smooth, has double curves but no triple points, or
  has triple points, respectively. 
\end{definition}

\begin{definition}
  In addition, assume that we have a relatively nef and big line
  bundle $L^*$ on $X^*$. A \emph{nef model} is a Kulikov model
  $X\to C$ with a relatively nef line bundle $L$ extending
  $L^*$.
\end{definition}

\begin{definition}
  Assume that we additionally have an effective divisor
  $R^*\in |L^*|$ not containing any fibers.  A \emph{divisor model} is a
  nef model with an effective divisor $R\in |L|$ extending
  $R^*$, such that $R$ does not contain any strata of $X_0$.
\end{definition}

Given $X^*\to C^*$, a Kulikov model exists by Kulikov
\cite{kulikov1977degenerations-of-k3-surfaces} and Persson-Pinkham
\cite{persson1981degeneration-of-surfaces}, possibly after a finite
ramified base change $(C',0)\to (C,0)$. Given $L^*$, a nef model
exists by Shepherd-Barron
\cite{shepherd-barron1981extending-polarizations}. Given $R^*$, a
divisor model exists by \cite[Thm.2.11,
Rem.2.12]{laza2016ksba-compactification} and \cite[Claim
3.13]{alexeev2019stable-pair}.

Shepherd-Barron also proved that for any $n\ge4$ the sheaf $L^n$
is globally generated.  Thus, the linear system $|L^n|$ for $n\gg0$
defines a contraction $f\colon X\to \oX$ to a normal variety over
$C$ such that $L = f^*(\oL)$ for a relatively ample line bundle
$\oL$ on $\oX$.  Denote $\oR = f(R)$. This is a
Cartier divisor, and $R = f^*(\oR)$. We call the pair $(\oX,
\oR)$ the \emph{stable model of the divisor model $(X,R)$.}
This gives the following:

\begin{theorem}\label{thm:1-param-limit}
  Let $(\oX^*, \oR^*)\to C^*$ be a family of K3 surfaces with ADE
  singularities together with an ample Cartier divisor. Then possibly
  after a finite ramified base change there exists a completion
  $f\colon (\oX,\oR)\to C$ such that
  \begin{enumerate}
  \item The morphism $f$ is Gorenstein and $\omega_\oX \simeq_C
    \cO_\oX$.
  \item $\oR$ is an effective relative Cartier divisor.
  \item For the central fiber $(\oX_0,\oR_0)$, the
    surface $\oX_0$ is a reduced Gorenstein surface with
    $\omega_{\oX_0}\simeq \cO_{\oX_0}$ which has slc singularities.
  \item The divisor $\oR_0$ does not contain the log centers of
    $\oX_0$, and the pair $(\oX_s,\epsilon\oR_s)$ is slc for any
    $0<\epsilon\ll1$ and all $s\in C$.
  \end{enumerate}
  This completion is unique. On each fiber one has
  $H^i(\oX_s,\oL_s)=0$ for $i>0$.
\end{theorem}
\begin{proof}
  After a finite base change $(C',0)\to (C,0)$, there is a
  simultaneous resolution of singularities $X^*\to \oX^*$, so that
  $X^*\to C^*$ is a family of smooth K3s (denoting the new curve
  $C'$ again by $C$ to simplify the notation). By
  \cite[3.13]{alexeev2019stable-pair} possibly after a further finite
  change there exists a divisor model. As above, we take $(\oX,\oR)$
  to be its stable model. It satisfies conditions (1-4) and outside
  the central fiber recovers the original family.

  Uniqueness is a general well known property of families of stable
  slc pairs since it is the relative log canonical model for any
  completion.  The proof of $H^i(\oX_s,\oL_s)=0$ for $i>0$ can be
  found in \cite[p.155]{shepherd-barron1981extending-polarizations} in
  the proof of Theorem 2W.
\end{proof}

We use the terms ``stable pair" or ``stable slc pair" interchangeably
to refer to a pair $(\oX_s,\epsilon\oR_s)$ with slc singularities 
and $K_{\oX_s}+\epsilon \oR_s$ ample.
Some literature uses the term ``KSBA pair."

We also note the following lemma for more general families of divisor
models: 

\begin{lemma}\label{lem:div-to-slc-families} 
  Let $\pi\colon (X,R)\to S$ be a flat family of divisor models over a
  locally Noetherian scheme, $L=\cO_X(R)$. Then $L^n$ for $n\ge4$ is
  relatively globally generated over $S$ and $L^n$ for $n\gg0$ defines
  a contraction $f\colon X\to \oX\to S$ to a flat family of stable
  models $(\oX, \epsilon \oR)$ over $S$, $L=f^*\oL$ and $R=f^*\oR$.
\end{lemma}
\begin{proof}
  By \cite[Lemma 2.17]{shepherd-barron1981extending-polarizations} for
  every fiber $X_s$ one has $H^i(X_s, L_s^n)=0$ for $n\ge0$ and
  $i>0$. 
  Thus by Cohomology
  and Base Change \cite[III.12.11]{hartshorne1977algebraic-geometry}
  for any $s\in S$ the morphism $\pi_*L^n \otimes k(s) \to
  H^0(X_s,L^n)$ is an isomorphism. Hence, for $n\gg0$ the sheaf $L^n$
  defines a contraction whose restriction to each fiber $X_s$ is the
  contraction given by $|L_s^n|$, to the stable model.
\end{proof}

\subsection{Complete moduli via stable slc pairs}
\label{sec:complete-moduli}

\cite{alexeev2019stable-pair} constructed the stable pair
compactification of the moduli space of K3 surfaces $(\oX, \epsilon
\oR)$ with ADE singularities together 
with an effective ample divisor.  For reader's convenience, we provide
more details of this construction in Theorem~\ref{thm:slc-moduli}. They
are well known to experts but scattered throughout the literature.
Also, our case is significantly easier than the case of general stable
pairs, see Remark~\ref{rem:general-slc-moduli}.

\begin{definition}
  For a positive integer $e$, 
  a \emph{stable $K$-trivial pair} of degree $e$ over an algebraically
  closed field of characteristic $0$ is 
  a pair $(Y,\epsilon B)$ such that
  \begin{enumerate}
  \item $Y$ is a reduced connected projective Gorenstein surface with
    $\omega_Y\simeq\cO_Y$,
  \item $B$ is an effective ample Cartier divisor on $Y$ with $B^2=e$.
  \item Denoting $L=\cO_Y(B)$, 
the Hilbert polynomial $h(n)$ is $\chi(L^{\otimes n})= \frac12 en^2 + 2$.
  \item $Y$ has slc singularities and $B$ does
    not contain any log centers of $Y$. Equivalently, 
    the pair $(Y,\epsilon B)$ is slc for any $0<\epsilon\ll1$.
  \end{enumerate}
\end{definition}

\begin{definition}\label{def:family-stable-k3}
  Let $S$ be a locally Noetherian scheme over $\bC$.  A family of
  stable $K$-trivial pairs of degree $e$ over $S$ is a proper flat
  Gorenstein morphism $f\colon (Y,B)\to S$ such that
  $\omega_{Y/S}\simeq \cO_Y$ locally on $S$, the divisor $B$ is an
  effective relative Cartier divisor and such that every geometric
  fiber is a stable $K$-trivial pair of degree $e$.

  The moduli functor $\ume$ is the contravariant functor from
  the category of locally Noetherian schemes over $\bC$ to the
  category of sets associating to a scheme $S$ the set $\ume(S)$
  of such families modulo isomorphisms over
  $S$.
  
  The moduli stack $\cme$ associates to a scheme $S$ the
  groupoid of sets $\cme(S)$ of such families, in which arrows are
  isomorphisms of families over $S$.
\end{definition}

\begin{theorem}\label{thm:slc-moduli}
  The stack $\cme$ is a Deligne-Mumford stack with finite
  stabilizer
   which has a coarse moduli space $M_e$, an algebraic
  space of finite type over $\bC$. Each proper subspace of $M_e$
  is projective.
\end{theorem}
\begin{proof}
  Following a standard procedure, one has to check that the functor
  $\ume$ is bounded, locally closed or at least constructible,
  separated, and has finite automorphisms. Then the first half of the
  theorem is proved by showing that the stack $\cme$ is the quotient
  stack of an appropriate subscheme of a Hilbert scheme by a group
  action and applying \cite{keel1997quotients-by-groupoids}. The
  projectivity of proper subspaces is the result of
  \cite{kovacs2017projectivity, fujino2018semipositivity-theorems} 
  following the earlier work \cite{kollar1990projectivity-of-complete}.

  (1) \emph{Boundedness}. By
  \cite[Thm. 2.1.2]{kollar1985toward-moduli} the family of polarized
  surfaces with a fixed Hilbert polynomial is bounded. Thus, there
  exists an $m$ such that for any polarized surface $(Y,L)$ with the
  Hilbert polynomial $h(n)=\frac12 en^2+2$ and any $k\ge m$ one has
  that $L^k$ is very ample, $H^i(Y,L^k)=0$ for $i>0$, and $H^0(Y,L^k)$
  generates the graded algebra
  $R(Y,L^k)=\oplus_{d\ge0} H^0(X, L^{dk})$.

  (2) \emph{Local closedness.} Let $f\colon (Y, L)\to S$ be a
  proper flat morphism with a relatively ample line bundle and a
  closed subscheme $B$ given by a compatible collection of sections
  $s_i$ of $L$ on $Y\times_S U_i$ for an open cover $S=\cup U_i$.
  We claim that there exists a locally closed subscheme $T\into S$ such
  that for any $S'\to S$ the base changed family
  $f'\colon (Y,B)\times_S S' \to S'$ is a family of stable
  $K$-trivial pairs of degree $e$ iff the morphism $S'\to S$ factors
  through $T$.

  First of all, the locus in $S$ where the geometric fibers are
  reduced, equidimensional, and Cohen-Macaulay is open in $S$ by
  \cite[IV$_3$, 12.2]{grothendieck1965elements-de-geometrie43}.
  Since the function $h^0(\cO_X)$ is upper semi continuous, the subset
  of $S$ where fibers are connected is also open. We
  shrink $S$ to this open subset. Since the fibers are reduced and
  Cohen-Macaulay, the condition that $B$ is a relative Cartier
  divisor is equivalent to the condition that the fibers of $B\to S$
  are equidimensional. Again, this is an open condition by
  \emph{ibid.}

  Because formation of the relative dualizing sheaf commutes with base
  changes, the Gorenstein property is also open on $S$. Further, the
  property that the two invertible sheaves $\omega_{Y/S}$ and $\cO_Y$
  differ by a line bundle from the base
  is represented by a locally closed subscheme by 
  \cite[Lem.\ 1.19]{viehweg1995quasi-projective-moduli}.  The property
  of having at worst nodal singularities in codimension~$1$ is open as
  well.

  For families satisfying the above conditions, the property of fibers
  $Y$ to have slc singularities is open, cf.
  \cite[2.6]{karu2000minimal-models} and
  \cite[5.5]{kollar1988threefolds-and-deformations}.  One checks it on
  $1$-parameter deformations, i.e. on base changes $C\to S$ with
  $(C,0)$ a regular pointed curve. First, assume that the general
  fiber of $Z=Y\times_S C\to C$ is normal. By Serre's criterion of
  normality, $Z$ is normal in an open neighborhood of $Z_0$. By
  shrinking $C$ we can assume that $Z$ is normal. Assume that
  $Z_0$ is slc. By Inversion of Adjunction
  \cite{kawakita2007inversion-of-adjunction} the pair $(Z,Z_0)$
  is log canonical. Let $\pi\colon\wZ\to Z$ be a log resolution of
  singularities with exceptional divisors $E_i$. One has
  \begin{math}
    K_{\wZ} = \pi^* K_{Z} + \sum a_i E_i
  \end{math}
  with $a_i\ge -1$.
  By shrinking $C$ we can assume that the the images of each $E_i$
  are either $C$ or $0$ and that for $t\ne0$ the map
  $\wZ_s\to Z_s$ is a log resolution of singularities. Then for
  $t\ne0$ one has
  \begin{math}
    K_{\wZ_s} = \pi^* K_{Z_s} + \sum a_i E_i|_{\wZ_s},
  \end{math}
  so $Z_s$ has log canonical singularities. When the general fiber
  of $Z$ is not normal, one considers the normalization
  $({Z}^\nu,D)$ together with the preimage $D$ of the double
  locus. Repeating the same argument, the fibers $({Z}^\nu_s,D_s)$
  are log canonical. One concludes that $Z_s$ are slc by gluing back
  $({Z}^\nu_s,D_s)$ and applying
  \cite[5.38]{kollar2013singularities-of-the-minimal}.

  The same argument shows that the union of the log centers of the
  fibers is a closed subset of $Y$. Then the property that the
  divisor $B_s$ does not contain a log center of $Y$ is open on the
  base. This concludes the proof of local closedness.
  
  (3) \emph{Separatedness.} Each family of $K$-trivial stable pairs
  over a punctured curve $C\setminus 0$ has at most one completion to a
  family over $C$. In a very standard way, this follows from the
  uniqueness of the relative canonical model over $C$.

  (4) \emph{Finite automorphisms.} Again, it is very well known that
  stable slc pairs have finite automorphisms.

  \medskip

  We now give the actual construction. Let $m$ be as in (1). Let
  $H$ be the Hilbert scheme and 
  $$Y_H \subset H \times \bP^{h(m)-1} \times \bP^{h(m+1)-1}$$ be the
  universal family parameterizing closed
  subschemes of $\bP^{h(m)-1} \times \bP^{h(m+1)-1}$ embedded by Segre
  into $\bP^{h(m)h(m+1)-1}$ using $\cO(1,1)$, with the Hilbert
  polynomial our surfaces would have under such embedding. There is an
  open subset $U\subset H$ parameterizing subschemes that map
  isomorphically under both projections $p_1,p_2$ to $\bP^{h(m)-1}$
  and $\bP^{h(m+1)-1}$ and such that the projections have Hilbert
  polynomials $h(mn)$, resp. $h((m+1)n)$.  Over $U$, we have two line
  bundles $L_m=p_1^*\cO(1)$ and $L_{m+1}=p_2^*\cO(1)$. Let
  $U'\into U$ be the locally closed subscheme representing the
  property 
  $L_m^{m+1} \simeq L_{m+1}^m$ locally over the base, it exists by
  \cite[Lem.\ 1.19]{viehweg1995quasi-projective-moduli}.
  Let $L = L_{m+1}\otimes L_m^{-1}$. Then $L_m \simeq L^m$
  and $L_{m+1} \simeq L^{m+1}$. Thus, $L$ is a relatively ample
  line bundle with Hilbert polynomial $h(n)$. 
  
  Let $V\subset U'$ be the open subset over which the fibers satisfy
  $H^i(Y,L)=0$ for $i>0$, using the upper semi continuity of
  $H^i(Y,L)$ in flat families.  Let $\pi\colon Y_V \to V$ be the
  restricted family. By Cohomology and Base Change $\pi_*L$ is a
  locally free sheaf on $V$ of rank $h(n)$. Over $W=\bP_V(\pi_*L)$
  we have a family $(Y_W,B_W)\to W$ of pairs as in (2). By local
  closedness there exists a locally closed subscheme $T\into W$ whose
  fibers are $K$-trivial stable pairs of degree $e$ and all such pairs
  occur.

  The family $(Y_T, B_T) \to T$ is the fine moduli space for the
  families $f\colon (Y,B)\to S$, $L=\cO_Y(B)$ of $K$-trivial
  stable pairs of degree $e$ with two additional pieces of data:
  nondegenerate embeddings $i_m\colon Y \subset S\times \bP^{h(m)-1}$
  and $i_{m+1}\colon Y \subset S\times \bP^{h(m+1)-1}$ with
  $i_m^*\cO(1)\simeq_T L^m$, resp.
  $i_{m+1}^*\cO(1)\simeq_T L^{m+1}$.  Vice versa, any family of
  $K$-trivial stable pairs admits such extra data isomorphisms locally
  in Zariski topology over~$S$. It follows that the stack $\cM_e$ is
  the quotient stack $[T: (\PGL_{h(m)}\times\PGL_{m+1})]$.  We
  complete the proof by applying \cite[1.1,
  1.3]{keel1997quotients-by-groupoids}.
\end{proof}

\begin{corollary}
  Fix $e>0$. Then there exists $\epsilon_0>0$ such that for any
  $\epsilon\in\bQ_{>0}$ with $\epsilon\le\epsilon_0$ and any family
  $f\colon (Y,B)\to S$ of $K$-trivial stable pairs of degree~$e$
  the geometric fibers $(Y,\epsilon B)$ have slc singularities and
  ample $\bQ$-divisor $K_Y+\epsilon B$. Thus the family
  $f\colon (Y,\epsilon B)\to S$ is a family of stable slc pairs.
\end{corollary}
\begin{proof}
  This follows from boundedness of the moduli functor by Noetherian
  induction. Indeed, the scheme $T$ above is of finite type over
  $\bC$.
\end{proof}

\begin{remark} \label{rem:general-slc-moduli} Ours is a fortunate
  situation where the morphisms $Y\to S$ are Gorenstein and the
  divisors $B$ are relative Cartier divisors. In the case of more
  general stable pairs, where $K_Y+B$ is only $\bQ$-Cartier, there are
  significant complications that we are able to avoid completely:

  (1) Boundedness is a highly nontrivial result. For surfaces, it was
  done in \cite{alexeev1994boundedness-and-ksp-2} and for higher
  dimensional pairs in \cite{hacon2018boundedness-moduli}.

  (2) Even in the case of varieties $Y$ with $B=0$, for general
  families formation of the sheaves
  $\omega_Y^{[n]}=(\omega_Y^{\otimes n})^{**}$ does not commute with
  base change.  As a consequence, the definition of the moduli functor
  becomes highly nontrivial, and there are several choices for it. To
  prove that a chosen moduli functor is constructible, one applies the
  theory of \cite{kollar2008hulls-and-husks}.

  (3) For a completed $1$-parameter degeneration $(Y,B)\to (C,0)$
  the Minimal Model Program only guarantees that the divisor
  $K_Y+B$ is $\bQ$-Cartier. If $B$ is not $\bQ$-Cartier then the
  closed subscheme $B_0\subset Y_0$ may have embedded
  components. One then needs to have an appropriate theory in order to
  be able to work with families $(Y,B)$ with divisors $B$ rather than
  closed subschemes $B$.

  Discussing this more general case is beyond the scope of this paper.
\end{remark}

\begin{remark}
  Since below we are only interested in the closure, with reduced
  scheme structure, of the locus of ADE K3 surfaces, an alternative
  way is to work over reduced bases $S$ only and to use the moduli
  functor of pairs defined in \cite{kovacs2017projectivity}.

  We chose
  to work with families over not necessarily reduced bases but the
  resulting coarse moduli space $M_e$ is perhaps not proper.  If one
  proved an analogue of Theorem~\ref{thm:1-param-limit} for
  log Calabi-Yau pairs $(\oX,\Delta + \epsilon \oR)$,
  crucially with a Cartier divisor $\oR$, 
  that would imply that the entire
  connected component containing a point corresponding to a normal K3
  surface is proper.
\end{remark}

\medskip

Now let $F$ be the moduli space of ADE elliptic K3 surfaces
$\pi\colon \oX\to \bP^1$ such that every fiber of $\pi$ is irreducible,
with a section~$s$ and a fiber class $f$.
Such fibrations have a unique Weierstrass model.
$F$ is an 18-dimensional quasiprojective variety.
Suppose that for each such K3 surface we have chosen in some
canonical way an ample divisor $\oR\in |\oL|$ for $\oL$ a polarization in
$\Z s\oplus \Z f$. We will call $\oR$ the
\emph{polarizing divisor}.
Then the pairs $(\oX,\epsilon \oR)$ are automatically $K$-trivial
stable slc pairs. There exists $\epsilon_0$ such that for any
$0<\epsilon\ll1$ the pairs $(\oX,\epsilon \oR)$ are stable slc pairs.

Suppose that $\oL = n(s+(d+1)f)$ for some positive
integers $n,d$, as is always the case in this paper. Let $P_{2d,n}\subset M_e$
be the projective bundle over $F_{2d}$ of sections of $n$ times the
primitive polarization; here $e=2dn^2$. We claim that the morphism
\begin{displaymath} F\to P_{2d,n} \qquad (\oX,\pi,s,f)\mapsto
                                                (\oX,\epsilon\oR) 
\end{displaymath} 
is a closed immersion: First, note that the morphism $F\to F_{2d}$ is set-theoretically injective
because $s$ can be reconstructed as the base locus of $|s+(d+1)f|$, and thus
so can $f$ and $\pi=|f|$. Since $F\to F_{2d}$ is a Heegner divisor, locally cut out
in period coordinates by a hyperplane, the set-theoretic injectivity
implies that $F\to F_{2d}$ is an immersion. Then, the choice of $\oR$
is a section of the projective bundle $P_{2d,n}\big{|}_F\to F$
and hence defines an immersion $F\hookrightarrow P_{2d,n}\big{|}_F\hookrightarrow P_{2d,n}$. 

\begin{definition}\label{def:oFslc}
  For a choice of polarizing divisor $\oR$, denote by $\oF^\slc$
  the closure of $F$ in the moduli $M_e$ of stable slc pairs, taken with the
  reduced scheme structure. \end{definition}
  
$\oF^\slc$ is projective because $F$ embeds in $P_{2d,n}$ and
$\overline{P}_{2d,n}\subset M_e$ is projective by Theorems~\ref{thm:1-param-limit}
and \ref{thm:slc-moduli}.

\begin{definition}
  The compactification for the polarizing
  divisor $\oR=s+m\sum_{i=1}^{24} f_i$ for a fixed $m\ge1$,
  where $s$ is the section and $f_i$ are the singular fibers, which
  may coincide, is denoted by $\slcrc$.  Any $m\ge1$ gives the same
  result.
\end{definition}

Another natural choice is given by the ramification divisor of the
elliptic involution. If $\wX\to\bP^1$ is a Weierstrass fibration
with section $s$, the ramification divisor of the elliptic involution
is a disjoint union of $s$ and the trisection $\wR$ of $2$-torsion
points. One has $s^2=-2$, so the ramification divisor is not nef.
But after contracting the section, one
obtains a nodal surface $\oX$ that is a double cover of
$Y=\bP(1,1,4)$, and the image $\oR$ of $\wR$ is ample. On the
resolutions the class
of $R$
 is $3(s+2f)$ and the morphism to $Y$ is given by the linear
system $|s+2f|$.

Since $(s+2f)^2=2$ these contracted, pseudoelliptic surfaces
are K3 surfaces with degree 2 polarization and $ADE$ singularities.
They are distinguished among generic degree 2 K3 surfaces
because $s$ is contracted. Their moduli
$F$ forms the unigonal divisor in the 
moduli space $F_2$.
The K3 surfaces outside of this divisor maintain an
involution, but are instead double covers
$X\to \bP^2$ ramified in a sextic. The description of the 
compactification for the pairs $(\oX,\epsilon \oR)$ in this case follows
from that of the compactification $\oF_2^\slc$ considered in
\cite{alexeev2019stable-pair}.

\begin{definition}
  Let $\slcram$ denote the compactification of the moduli space of
  pseudoelliptic pairs $(\oX,\epsilon \oR)$ for the choice of polarizing
  divisor $R$ equal to the ramification divisor of the double cover
  $\oX\to\bP(1,1,4)$.
\end{definition}

\subsection{Toroidal compactifications of $F$}
\label{sec:toroidal-comp}

Let $\II_{2,18} = H^2\oplus (-E_8)^2$ be the unique even unimodular
lattice of signature $(2,18)$. Let $O(\II_{2,18})$ be its
isometry group. Define the period domain
\begin{displaymath}
 \mathbb{D} = \{x\in \bP(\II_{2,18}\otimes \bC) \mid x^2=0,\,
  x\cdot \overline{x} > 0 \}.
\end{displaymath}
It consists of two isomorphic connected components, each a bounded Hermitian
symmetric domain of Type IV, naturally interchanged by complex conjugation.
By the Torelli theorem \cite{piateski-shapiro1971torelli}, 
the quotient $\bD/O(\II_{2,18})$ is $F$. It is connected and so we may
as well replace $\mathbb{D}$ with one of its connected components, and instead quotient
by the subgroup $O^+(\II_{2,18})$ preserving this component.

The space $F$ has the Baily-Borel \cite{baily1966compactification-of-arithmetic} compactification $\oF^\bb$
in which the boundary consists of a unique $0$-cusp, a point, and two
$1$-cusps, which are curves. The $0$- and $1$-cusps are in bijection
with $O^+(\II_{2,18})$-orbits of primitive isotropic lattices of ranks $1$ and $2$ respectively.
Let $\delta\in \II_{2,18}$ be a primitive vector with $\delta^2=0$.  Then
$\delta^\perp/\delta \simeq \IIell = H\oplus E_8^2$ is the
unique even unimodular lattice of signature $(1,17)$.

Let $C$ denote a connected
component of the positive norm vectors of $\delta^\perp/\delta\otimes \bR$ and let
$\oC_{\bQ}$ be its rational closure, obtained by adding the rational isotropic rays
on the boundary of $C$. Let $\Gamma={\rm Stab}_\delta/U_\delta\cong O^+(\IIell)$
be the quotient of the stabilizer
${\rm Stab}_\delta\subset O^+(\II_{2,18})$
by its unipotent subgroup $U_\delta$.
It follows from the general theory \cite{ash1975smooth-compactifications}
that a toroidal compactification
$\oF^\cF$ is defined by a $\Gamma$-invariant fan $\mathcal{F}$ with support equal
to $\oC_\bQ$ and finitely many orbits of cones.

The toroidal compactification is described in a neighborhood of the $0$-cusp by
the quotient $X(\mathcal{F})/\Gamma$. By the nilpotent orbit theorem
\cite{schmid1973variation, friedman1986type-III},
one-parameter arcs
approaching the $0$-cusp are approximated by translates of co-characters
of the algebraic torus $\delta^\perp/\delta\otimes \C^*\cong \Hom(\delta^\perp/\delta,\,\C^*)$
modulo $\Gamma$. These co-characters are of the form $\lambda\otimes \C^*$
for some $\lambda\in C \cap \delta^\perp/\delta\textrm{ mod }\Gamma$,
with $\lambda^2>0$.
Similarly, one-parameter arcs
approaching a $1$-cusp are approximated by a co-character associated
to a vector $\lambda \in \oC_\Q\cap \delta^\perp/\delta$
satisfying $\lambda^2=0$.

\begin{definition} We say $\lambda$ is the {\it monodromy invariant}
of an elliptic K3 degeneration $X^*\to C^*$ if a translate of the co-character
$\lambda\otimes \C^*$ approximates the degeneration
of the period map $C^*\to \delta^\perp/\delta\otimes \C^*$. \end{definition}

\section{Proof method for Theorem \ref{intro2}}   
\label{sec:method}

We describe a general method for proving the existence of a morphism 
$$\oF_\bM^\mathcal{F}\to \oF_\bM^\slc$$
from a toroidal compactification to an slc compactification
of the moduli space of $\bM$-lattice polarized K3 surfaces
for
some choice of fan $\mathcal{F}$ and polarizing divisor $R$. Under suitable
circumstances this map is the normalization.
The method was developed in
\cite{alexeev2019stable-pair} in the case of moduli of degree 2 K3 surfaces $F_2$,
but was not phrased as a general theorem.

Consider a moduli space of $\bM$-lattice polarized K3
surfaces. See \cite[Def.~2.33]{alexeev2021compact} for a precise definition.
There is an isomorphism of coarse spaces $F_{\bM}=\bD_\bM/G_\bM$
\cite[Thm.~2.34]{alexeev2021compact} with a Type IV arithmetic quotient.
Suppose that on a generic K3 surface in this moduli we have
chosen, in some canonical way, an effective divisor $R$ in some ample
class $h\in\bM$. The space $\oF_\bM^\slc$ is defined the same way as in
Def.~\ref{def:oFslc}, by taking a closure of $F_\bM\subset M_e$ for $e=R^2$.

For example, for ordinary primitively polarized K3
surfaces $(X,L)$, $L^2=2d$, this means a choice $R\in |nL|$ in some
fixed multiple~$h=nL$ of the generator.
   
\begin{theorem}\label{threeinputs} Let $F_\bM=\mathbb{D}_\bM/G_\bM$
be a moduli space of $\mathbb{M}$-lattice polarized K3 surfaces, and let $R$ be a canonical choice of
polarizing divisor. Suppose we are given the following inputs:
\begin{enumerate} 
\item[\bf (div)] Some divisor model $(X(\lambda),R)$ with
  possibly imprimitive
  monodromy invariant $\lambda$,
for all projective classes
$[\lambda]$ of rational lines in $\oC_\Q\cap \delta^\perp/\delta$, and all $G$-orbits of
primitive isotropic vectors $\delta$.
\item[\bf (d-ss)] A theorem proving that all $d$-semistable (cf. Definition \ref{def:dss}) deformations of
$X_0(\lambda)$ which keep the classes in $\mathbb{M}$ Cartier also admit a deformation
of the divisor $R$, so that the deformed pair is also a divisor model.
\item[\bf (fan)] A fan $\mathcal{F}$ such that the combinatorial type of the stable model
$(\oX_0(\lambda),\epsilon\oR)$ is constant for all $\lambda$ in the interiors of the cones of $\mathcal{F}$.
\item[{\bf (qaff)}] A proof that the Type III strata of $\oF_\bM^\slc$ are quasiaffine.
\end{enumerate}
Then there is a morphism
$\oF_\bM^{\cF}\rightarrow\oF_\bM^{\slc}$
from the toroidal compactification to the stable pair compactification for the divisor $R$,
mapping strata to strata.
\end{theorem}

\begin{proof} Since the interiors are isomorphic, we have a birational
  map
  $\varphi\colon \oF_\bM^{\cF} \dashrightarrow
  \oF_\bM^{\slc}$ between the two moduli
  spaces. Eliminate indeterminacy by
$$\oF_\bM^{\cF}\leftarrow Z\to \oF_\bM^{\slc}.$$
Let $Z_p$ be the fiber of the left-hand map over
$p\in \oF_\bM^{\cF}$. Since $\oF_\bM^\cF$ is normal, if
$\varphi$ is not regular then there exists a $p$ such that the map
$Z_p\to \oF_\bM^\slc$ is non-constant.

Let $(C,0)\to Z$ be an
arbitrary one-parameter family such that $0\mapsto Z_p$.
The curve $(C,0)$ defines some monodromy invariant $\lambda\in \oC_\Q(\delta)/\Gamma$
depending on how it approaches the boundary. Here $\Gamma = {\rm Stab_\delta}/U_\delta$ where
${\rm Stab}_\delta\subset G$ is the stabilizer
of $\delta$. Either $\lambda^2>0$ and $\Z\delta$ corresponds
to the $0$-cusp that $(C,0)$ approaches or $\lambda^2=0$ and $\Z\lambda\oplus \Z\delta$ 
corresponds to the $1$-cusp that $(C,0)$ approaches. Such arcs are respectively
given by Type III or Type II
degenerations. 

Let $\oF_\bM^{\lambda}$ be the toroidal extension of the moduli space
whose only cones are rays in the directions of $\Gamma\lambda$. Then
$\oF_\bM^{\lambda}$ is the union $\mathcal{M}$ with a single divisor $\Delta$
on the boundary. When $\lambda^2>0$, the boundary divisor $\Delta$ is
isomorphic to the ${\rm Stab}_\lambda$-quotient
of a torus of dimension $19-\rk\mathbb{M}$. When $\lambda^2=0$
it is a finite quotient of a family of abelian varieties
isogenous to $\mathcal{E}^{18-\rk\bM}$, the self-fiber product of the universal
family over some modular curve. Let $V_\lambda$ be an analytic neighborhood
of the boundary divisor  $\Delta\subset \oF_\bM^{\lambda}$ and
let $U_\lambda\to V_\lambda$
be a cover branched along $\Delta$ of order the imprimitivity of $\lambda$.

Input {\bf (div)} implies that there is some possibly imprimitive 
$\lambda$ representing $[\lambda]$ which is the monodromy invariant of some
divisor model $(X(\lambda),R)$. When $\rk \mathbb{M}=1$, an important result of
Friedman-Scattone \cite[5.5, 5.6]{friedman1986type-III}
implies that there is a family $\mathfrak{X}_\lambda\to \widetilde{U}_\lambda$
extending the universal family over the $d$-semistable deformation space of
$X_0(\lambda)$ which keep the classes in $\mathbb{M}$
Cartier---here $\widetilde{U}_\lambda$ is a some etale cover 
of $U_\lambda$. The same proof
applies to higher rank polarization.
 
Input {\bf (d-ss)} implies that not just the line bundles in $\mathbb{M}$,
but also the divisor models, extend to produce a family
$(\mathfrak{X}_\lambda,\mathfrak{R}) \to  \widetilde{U}_\lambda$.

Since $C^*\to F_\bM$ is approximated by the cocharacter
$\lambda$, it follows that the period map extends to a morphism
$(C,0)\to \oF_\bM^{\lambda}$. Lifting this arc to the
cover $\widetilde{U}_\lambda$ and restricting
$(\mathfrak{X}_\lambda,\mathfrak{R})$ we get a divisor model
 $(X,R)\to (C,0)$. By
Lemma~\ref{lem:div-to-slc-families}
the stable model of $(X,\epsilon R)$ is $(\oX,\epsilon\oR)$.
Note the choice of lift of the arc doesn't ultimately affect
the resulting stable model.

Following \cite[Thm.~10.5]{alexeev2019stable-pair},
consider an arc in $F_\bM$ limiting a point in $Z_p$. While $p$
does not determine the monodromy invariant $\lambda$ of this arc,
we necessarily have that $\lambda$ lies in the interior of the cone corresponding to the 
boundary stratum of $\oF_\bM^\mathcal{F}$ containing $p$.

 Input {\bf (fan)} allows us to
conclude: For all arcs $(C,0)$ approaching
a point in $Z_p$ the stable model $(\oX,\epsilon \oR)\to (C,0)$ has a fixed combinatorial type.

Thus, the image of the morphism $Z_p\to \oF_\bM^{\slc}$ lies in a fixed
boundary stratum of the stable pair compactification. By {\bf (qaff)}, for
Type III degenerations, these strata are quasiaffine. Since $Z_p$ is proper, we conclude that this
morphism is constant if $p$ lies in the Type III locus.
This is a contradiction, so $\varphi$ is regular at $p$.

Finally, it remains to show that there is no indeterminacy in the Type II locus. Any
fan $\mathcal{F}$ contains the Type II isotropic rays as one-dimensional cones,
and $\oF_\bM^\lambda\subset \oF_\bM^\cF$ is an open subset.
Consider again the family $(\mathfrak{X}_\lambda,\mathfrak{R}) \to  \widetilde{U}_\lambda$.
Taking the relative proj of $n\mathfrak{R}$ gives a family of stable models
$(\overline{\mathfrak{X}}_\lambda,\epsilon\overline{\mathfrak{R}}) \to  \widetilde{U}_\lambda$ and the classifying
morphism $\widetilde{U}_\lambda\to \oF_\bM^\slc$ must factor through $V_\lambda$ because the fibers of
$ \widetilde{U}_\lambda\to V_\lambda$ lying the smooth locus give the smooth K3 surface with divisor.
The theorem follows. \end{proof} 

\begin{corollary}\label{normalization} Suppose that in addition,
\begin{enumerate}
\item[{\bf (dim)}] Any stratum in $\oF_\bM^{\cF}$ and its image in $\oF_\bM^{\slc}$ have the same dimension.
\end{enumerate}
Then
$\oF_\bM^{\cF}$ is the normalization of $\oF_\bM^{\slc}$. \end{corollary}

\begin{proof} The condition implies that the morphism from Theorem \ref{threeinputs} is finite. Since $\oF_\bM^{\cF}$ is normal,
we conclude by Zariski's main theorem that the morphism is the normalization. \end{proof}

\section{Three toroidal compactifications}
\label{sec:fans}

We now define three fans $\cF_\ram$,
$\cF_\cox$, $\cF_\rc$. Each
successively refines the previous. They are named the
{\it ramification fan}, {\it Coxeter fan}, and {\it rational curve fan}
respectively. These fans give three toroidal
compactifications of $F$ and our main theorem is that
the outer two are the normalizations of the 
compactifications $\oF^{\ram}$ and $\oF^{\rc}$
via stable slc pairs for the ramification divisor and
the rational curve (i.e. $s+m\sum_{i=1}^{24} f_i$)
divisor, respectively. The Coxeter fan is auxiliary.

\subsection{The Coxeter fan}
\label{sec:coxeter-fan}

The group $\Gamma=O^+(\IIell)$ contains the Weyl group
$W$ generated by reflections in the
roots, the $(-2)$-vectors $\alpha\in\IIell$. 
The Coxeter diagram $G_\cox$
of $W$ is well known and given in Fig.~\ref{fig:coxeter}. The nodes
correspond to a choice of simple roots $\alpha_1, \dotsc, \alpha_{19}$, so
that a fundamental domain for $W$-action is the positive chamber
\begin{math}
  P=\{\lambda\in \oC_\bQ\mid \lambda\cdot \alpha_i \ge0\}
\end{math}
with 19 facets.

\begin{figure}[htp]
  \centering
  \includegraphics[width=.8\textwidth]{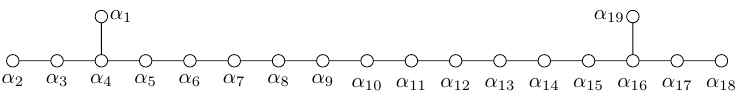}
  \caption{Coxeter diagram $G_\cox$ of $\IIell$}
  \label{fig:coxeter}
\end{figure}

One has $\alpha_i^2=-2$, $\alpha_i\cdot\alpha_j=1$ if the corresponding
nodes of the Coxeter diagram are
connected by an edge and~$0$ otherwise. Since $\IIell$
has rank $18$ there is a unique linear relation amongst the $19$ roots
$\alpha_i$:

\begin{equation}\label{eq:coxeter}
  3\alpha_1 + 2\alpha_2 + 4\alpha_3 +
  \sum_{k=4}^{16} (10-k) \alpha_k
  -4\alpha_{17} - 2\alpha_{18} - 3\alpha_{19} = 0
\end{equation}

\begin{definition}
  The Coxeter fan $\cF_\cox$ is defined by cutting the cone $\oC_\bQ$
  by the mirrors $\alpha^\perp$ to the roots. 
\end{definition}

Since $W$ is a reflection group,
the (orbits of) cones $\cF_\cox/W$ are in a bijection with faces of~$P$.
The group $\Gamma$ is an extension of $W$ by $\Aut G_\cox=\bZ_2$.
Thus, the cones in $\cF_\cox/\Gamma$ are in a bijection with faces of $P$
modulo the left-right symmetry.

By \cite[Thm.3.3]{vinberg1973some-arithmetic}, the nonzero faces of
$P$ are of two types. Type $\II$ rays corresponding to maximal
parabolic subdiagrams of $G_\cox$: maximal disjoint unions of
the affine
Dynkin diagrams. Type $\III$ cones of dimension
$18-r$ correspond to elliptic subdiagrams of $G_\cox$: disjoint
unions of Dynkin diagrams with
$0\le r\le17$ vertices. A subset $I\subset G_\cox$ of vertices
corresponds to the face $\cap_{i\in I} \alpha_i^\perp\cap P$.

The two type $\II$ rays correspond to the maximal parabolic subdiagrams
$\wE_8\wE_8$ and $\wD_{16}$. Similarly, one can count the 80 type
$\III$ rays and count the higher-dimensional faces. In our special case,
however, there is an easier way.

\begin{lemma}\label{lem:faces}
  Suppose that an $18$-dimensional cone $P$ is defined by 19
  inequalities $a_i\ge0$ and that the linear forms $a_i$ satisfy a
  unique linear relation
  $\sum_{i=1}^{9} n_ia_i=\sum_{i=11}^{19}m_ia_i$, with $n_i>0$,
  $m_i>0$.  Then the faces of $P$ are in a bijection with arbitrary
  subsets $I\subset\{1,\dotsc,19\}$ satisfying a single condition:
  $\{1,\dotsc,9\}\subset I$ $\iff$ $\{11,\dotsc,19\}\subset I$. 
  A subset $I$ corresponds to the face $\cap_{i\in I} \{a_i=0\}\cap P$.
  For $I$ not containing $\{1,\dotsc,9\}$ $\codim F = |I|$, for
  those that do $\codim F = |I|-1$.
\end{lemma}

 \begin{proof}
   A face of $P$ is obtained by intersecting $P$ with some hyperplanes
   $a_i=0$.  Each point of $P$ gives a decomposition
   $I\sqcup I^c = \{1,\dotsc,19\}$ with $a_i=0$ for $i\in I$ and
   $a_j>0$ for $j\in I^c$. Obviously, $I$ must satisfy the above
   condition and, vice versa, for any such $I$ there exists a solution
   $(a_1,\dotsc,a_{19})$. 
 \end{proof}

\begin{corollary}
  In $\cF_\cox/W$ there are $2\cdot 9+1=19$ facets and $9^2+1=82$
  rays. In $\cF_\cox/\Gamma$ there are $9+1=10$ facets and
  $\frac{9\cdot 10}{2}+1=46$ rays. The total number of cones in
  $\cF_\cox/W$ is $2N^2+2$ and in $\cF^\cox/\Gamma$ it is $N^2+N+2$, where
  $N=2^9-1$. 
\end{corollary}

 \begin{proof}
For $\cF_\cox/W$, this follows from counting subsets $I$ satisfying the condition
of Lemma \ref{lem:faces}.
The cones in $\cF_\cox/\Gamma$ biject with involution orbits of such subsets.
 \end{proof}

\subsection{The ramification fan}
\label{sec:coxeter-semifan}

\begin{definition}
  The ramification fan $\cF_\ram$ is defined as a \emph{coarsening} of
  $\cF_\cox$. The unique $18$-dimensional cone is a union of four
  chambers $P_\ram = \cup_{g\in W_J} g(P)$ of $\cF_\cox$, where
  $W_J = \bZ_2\oplus\bZ_2$ is the subgroup of $W$ generated by 
  reflections in the roots $\alpha_1,\alpha_{19}$. The other maximal
  cones of $\cF_\ram$ are the images $g(P_\ram)$ for $g\in W$.

  The corresponding toroidal compactification of $F$ is denoted
  $\torram$.
\end{definition}

This is a special case of a generalized Coxeter semifan defined in
\cite[Sec. 10C]{alexeev2019stable-pair}, where its main properties are
described. The data for a generalized Coxeter semifan is a subdivision
$I\sqcup J$ of the nodes of $G_\cox$ into \emph{relevant} and
\emph{irrelevant} roots.  The maximal cones are the unions of the
chambers $g(P)$ with $g\in W_J$, the subgroup generated by the
reflections in the irrelevant roots, in this case
$\alpha_1,\alpha_{19}$. In general, the subgroup $W_J$ may be infinite
and the resulting cones may not be finitely generated.  In the present
case the group $W_J$ is finite, and so $\cF_\ram$ is an ordinary fan.

The cones of $\cF_\ram/W$ are in a bijection with the subdiagrams of
$G_\cox$ which do not have connected components consisting of the
irrelevant nodes $\alpha_1$ and $\alpha_{19}$.  The cones in
$\cF_\ram/\Gamma$ are in a bijection with orbits of these under
$\Aut G_\cox=\bZ_2$.  In $\cF_\ram/W$ there are 17 facets and 63 rays, and
in $\cF_\ram/\Gamma$ 9 facets and 35 rays.

 \subsection{The rational curve fan}
\label{sec:rc-fan}

Define the vectors
\begin{displaymath}
  \beta_L = \alpha_3+2\alpha_2-\alpha_1,\ 
  \gamma_L = \alpha_3-\alpha_1, 
  \qquad
  \beta_R = \alpha_{17}+2\alpha_{18}-\alpha_{19},\ 
  \gamma_R = \alpha_{17}-\alpha_{19}.
\end{displaymath}

The fan $\cF_\rc$ is a \emph{refinement} of the Coxeter fan, obtained
by subdividing the chamber $P$ by the hyperplanes $\beta_L^\perp$,
$\gamma_L^\perp$, $\beta_R^\perp$, $\gamma_R^\perp$ 
into $3\cdot 3 = 9$
maximal-dimensional subcones $\sigma_{LR}$ with  left and right ends
$L,R\in \{1,2,3\}$. The other maximal-dimensional cones of
$\cF_\rc$ are the $W$-reflections of these cones. The involution in
$\Aut G_\cox$ acts by exchanging $L$ and $R$. Thus, modulo $\Gamma$ there
are 6 maximal cones $\sigma_{11}$, $\sigma_{12}$, $\sigma_{13}$, $\sigma_{22}$,
$\sigma_{23}$, $\sigma_{33}$.


The subdivisions on the left
and right sides work the same way and independently of each
other. So we only explain the left side,
writing simply $\beta,\gamma$ for $\beta_L,\gamma_L$.
Since $\gamma = \beta - 2\alpha_2$ and $\alpha_2\ge0$ on $P$, $\beta\le 0$ implies
$\gamma\le0$, and $\gamma\ge0$ implies $\beta\ge0$. Thus, the hyperplanes
$\beta^\perp$ and $\gamma^\perp$ divide $P$ into three maximal
cones.
Fig.~\ref{fig:cones} gives a pictorial description of the
subdivision and the vectors involved. One has $\beta^2=-8$ and
$\gamma^2=-4$. The number of edges indicate the intersection numbers,
and negative numbers are shown by dashed lines.
In addition, not shown is $\beta\cdot\alpha_1=2$.

\begin{figure}[htp]
  \centering
  \includegraphics[width=.75\linewidth]{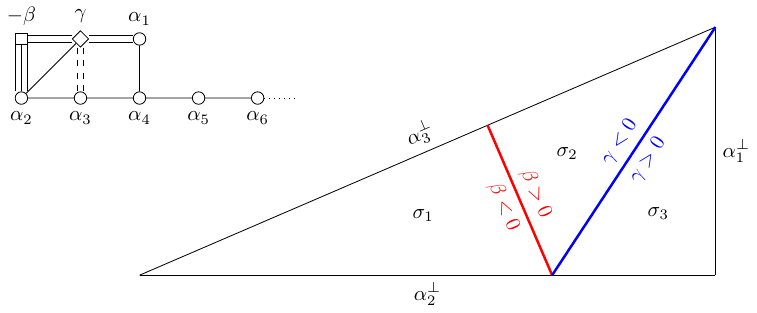}
  \caption{Subdivision of Coxeter chamber for the fan $\cF_\rc$}
  \label{fig:cones}
\end{figure}

These 
three maximal cones have 19 facets
and the vectors defining the facets satisfy a unique linear
relation: \vspace{5pt}

\begin{equation}
  \label{eq:lin-relations}
\begin{tabular}[h]{llcccccccc}
  $L=1:$ &$-\beta\ge0$  &\quad &$3(-\beta)$  &$+$ &$8\alpha_2$ &$+$ &$7\alpha_3$ &$+$ &$\dotsc =0$ \\
  $L=2:$ &$\beta\ge0$, $-\gamma\ge0$ &\quad &$\beta$ &$+$ &$4(-\gamma)$ &$+$ &$7\alpha_3$ &$+$ &$\dotsc = 0$ \\
  $L=3:$ &$\gamma\ge0$ &\quad &$2\alpha_2$ &$+$ &$4\gamma$  &$+$ &$7\alpha_1$ &$+$ &$\dots = 0$
\end{tabular}
\end{equation}

Here, the rest of each relation is $6\alpha_4 + 5\alpha_5+\dotsb$, the same as in
equation~\eqref{eq:coxeter} for the Coxeter chamber. Similarly, we
have a subdivision into 3 cones using the hyperplanes $\beta_R^\perp$ and
$\gamma^\perp_R$. Each of the resulting 9 cones $\sigma_{LR}$ has $19$ facets,
with the supporting linear functions satisfying a unique linear
relation. For every cone the relation has the same pattern of
signs. One concludes that each of the 9 cones is $\bQ$-linearly
equivalent to the Coxeter chamber, and Lemma~\ref{lem:faces} gives a
description of its faces.

For convenience
define $\sigma_L = \cup_{R\in\{1,2,3\}} \sigma_{LR}$,
which specifies only the left-end behavior.
The cones $\sigma_2$ and $\sigma_3$ are related by a reflection
$w$ in the $(-4)$-vector $\gamma$. Indeed, $w(\beta)=2\alpha_2$,
$w(\alpha_3)=\alpha_1$, and $w(\alpha_i)=\alpha_i$ for
$i\ge4$. However, this reflection does not preserve the lattice
$\IIell$. For example, $\beta$ is primitive and $2\alpha_2$ is
2-divisible.

There are $1+5+7+3=16$ cones of dimension $0\le d\le 3$ in
Fig.~\ref{fig:cones}. Therefore, in $\cF_\rc/W$ there are $3^2=9$
maximal cones, $2(7+6)+1 = 27$ facets, $(5+6)^2+1=122$ rays, and
a total of $2N^2+2$ cones, $N=16\cdot 2^6-1$. In $\cF_\rc/\Gamma$ there
are $\frac{3\cdot 4}{2}=6$ maximal cones, $7+6+1 = 14$ facets,
$\frac{11\cdot 12}{2}+1=67$ rays, and $N^2+N+2$ cones.

\begin{definition}
  The toroidal compactification corresponding to the fan $\cF_\rc$
  is denoted $\torrc$.
\end{definition}

Since the fan $\cF_\rc$ is very important for this paper, we describe
it in more detail and give each cone a unique $ADE$ label.  First, to
each maximal cone $\sigma_L$ we associate a Coxeter diagram whose
vertices correspond to the facets $v^\perp$ with $v\ge0$ on
$\sigma_L$.  Then a face $F$ of $\sigma_L$ is described by a
subdiagram of black vertices for the vectors $v$ such that
$F\subset v^\perp$.  In Table~\ref{tab:rc-cones3} we list several
cones of codimension $0,1,2,3$.  For a cone lying in more than one of
the maximal cones $\sigma_1,\sigma_2,\sigma_3$, we can choose either
of them to describe $F$, and we indicate our choice in bold in the
first column.

For each cone, we also indicate which other 
linear functions $\alpha_i,\beta,\gamma$ vanish on it. Namely, on the
cone $\sigma_2\cap \sigma_3 = \gamma^\perp$ one has
$\alpha_1=\alpha_3$ and $\beta/2=\alpha_2$, so once one of them
vanishes then so does the other.

\begin{table}[htp]
  \centering

\renewcommand{\arraystretch}{2}
\begin{tabular}[htp]{|l|c|c|l|}
  \hline
  Cone & Symbol & Diagram \\
  \hline
${\bm \sigma_1}$ &  $E_0$& \adjustbox{valign=t}{
                \includegraphics{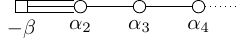}
                }\\ 

${\bm \sigma_2}$ &  $D'_0$ & \adjustbox{valign=m}{
                \includegraphics{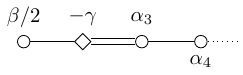}
                                               }\\
${\bm \sigma_3}$ & $D_0$ &\adjustbox{valign=t}{
                \includegraphics{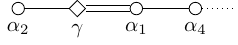}
                }\\
  
\hline

  ${\bm \sigma_2} \cap \sigma_1$ & $E_1$ &\adjustbox{valign=m}{
                \includegraphics{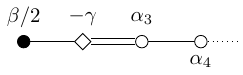}
                                                         
                                                         }\\

  ${\bm \sigma_3} \cap \alpha_2^\perp$ & $E'_1$ & 
                                                      \adjustbox{valign=t}{
                \includegraphics{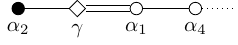}
                                                           }\\

  ${\bm\sigma_2 \bm\cap \bm\sigma_3}$ & $D_1$ &
\adjustbox{valign=m}{
                \includegraphics{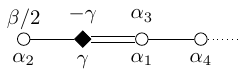}

                                                                    } \\
\hline

${\bm\sigma_2 \bm\cap \bm\sigma_3}\cap \sigma_1$ & $E_2$ &
                                                                   \adjustbox{valign=m}{
                \includegraphics{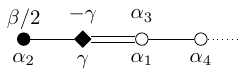}

                                                                   } \\
  ${\bm\sigma_2 \bm\cap \bm\sigma_3}\cap \alpha_1^\perp$ & $D_2$ &
\adjustbox{valign=m}{
                \includegraphics{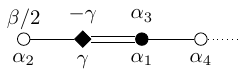}

                                                                    } \\
\hline
  ${\bm\sigma_2 \bm\cap \bm\sigma_3}\cap \sigma_1\cap \alpha_1^\perp$ & $E_3$ &
                                                                   \adjustbox{valign=m}{
                \includegraphics{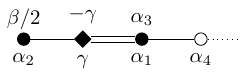}

                                                                   } \\
\hline                                                      
\end{tabular}

\medskip
  \caption{Basic type III cones in $\cF_\rc$}
  \label{tab:rc-cones3}
\end{table}

The lower-dimensional cones are obtained from these cones by
intersecting with some $\alpha_i^\perp$ for $i\ge4$. The diagram is
then obtained by marking these nodes black. Adding to the $D_2$ and
$E_3$ diagrams adjacent vertices makes it into larger $D_n$, $E_n$
diagrams. Marking some of the vertices that are not adjacent to the
end $D$ and $E$ diagrams adds some $A_n$ inside the chain
$\alpha_4,\dotsc,\alpha_{16}$.

\begin{remark}
  The reason for the $ADE$ notation is as follows: Starting with $D_2$
  and $E_3$, the cone is already a cone of the Coxeter fan $\cF_\cox$,
  so we use a subdiagram of the Coxeter diagram of
  Fig.~\ref{fig:coxeter} to label it.
  Note that for an $E_n$ diagram one gets a nonzero cone only if
  $n\le9$. For $n\le 8$ this is an elliptic subdiagram of
  Fig.~\ref{fig:coxeter}, i.e. a type III cone; for $n=9$ the cone
  $\cap\alpha_i^\perp \cap \oC_\bQ$ is the Type II $\wE_8\wE_8$ ray.

  We chose the labels $E_0$, $E_1$, $E'_1$, $E_2$, $D_0$, $D'_0$,
  $D_1$ 
  by analogy with the larger
  $E$ and $D$ diagrams. This will be further explained in
  Section~\ref{sec:rat-curves-divisor}.
\end{remark}

\begin{notation}
  To make the resulting $ADE$ label unique, we add the symbol
  $A_0$ to denote adjacent unmarked vertices.
  By a convention, explained further
  Section~\ref{sec:rat-curves-divisor}, we assign each
  label a \emph{charge}: $Q(A_n)=n+1$,
  $Q(D_n)=n+4$, $Q(E_n)=n+3$, and we require the sum of charges to be
  $24$. With these notations, a string of four white vertices is
  denoted by $A_0^3$ and adding black vertices to the interior two vertices
  produces diagrams $A_1A_0$, $A_0A_1$, $A_2$.
\end{notation}

We summarize this discussion as follows:

\begin{lemma}\label{lem:rc-cones}
  In the fan $\cF_\rc$ there are 9 maximal cones $\sigma_{ij}$, $1\le i,j\le
  3$ modulo $W(\II_{1,17})$ with the Dynkin labels,
  where $(D_0|D_0')$ denotes either $D_0$ or $D_0'$:
  \begin{displaymath}
  E_0A_0^{18}E_0,\ 
  E_0A_0^{17}(D_0|D_0'),\  (D_0|D_0')A_0^{17}E_0, \ 
  (D_0|D_0') A_0^{16} (D_0|D_0'),
  \end{displaymath}
  All type III cones are in a bijection with the labels
  \begin{displaymath}
    (E_{n_0}|E'_1|D_{n_0}|D_0') A_{n_1}\dotsc A_{n_k}
    (E_{n_{k+1}}|E'_1|D_{n_{k+1}}|D_0')
  \end{displaymath}
  with some $n_i\ge0$, and with $n_i\le 8$ for the $E_n$ diagrams, of
  total charge $24$.
\end{lemma}

Next, we list the type II rays of $\cF_\rc$. They are the rays of the
rational closure $\oC_{\bQ}$ of the cone $\{v^2>0\}$, so they are the
same as for the Coxeter fan. In the fan $\cF_\rc$, the $\wE_8\wE_8$
ray is contained in each of the 9 cones $\sigma_{ij}$, and the $\wD_{16}$
ray is contained in $\sigma_{ij}$ for $i=2,3$, $j=2,3$.

\begin{figure}[htp]
  \centering
  \includegraphics[width=.8\textwidth]{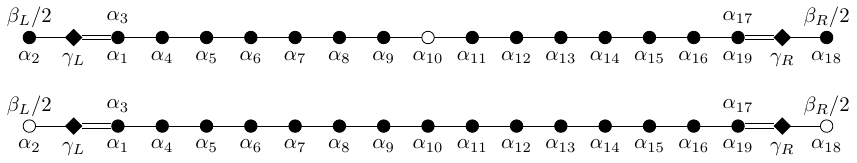}
\caption{Type II cones in $\cF_\rc$: $\wE_8\wE_8$ and $\wD_{16}$,
  shown in $\sigma_{33}$}
  \label{fig:rc-cones2}
\end{figure}

We conclude this section with a result which goes a long way towards
explaining some peculiar features of the fan $\cF_\rc$ which otherwise
may seem quite mysterious.

Recall: Let $\cF$ be a fan in a lattice $N$ defining a toric variety
$X(\cF)$. A cone $\tau\in\cF$ defines a torus orbit $O(\tau)$ whose
closure is $X(\tau)\subset X(\cF)$. Denote $N_\tau = N\cap
\bR\tau$. Then $X(\tau)$ is a toric variety for the fan $\Star(\tau)$
in the lattice $N/N_\tau$. 

Also recall that the root lattices $C_n$ and $D_n$ are the same but
their Weyl groups are different: $W(D_n)\subset W(C_n)$ is a subgroup
of index 2.

\begin{lemma}\label{lem:C-fan}
  Let $\Delta=\{\alpha_1,\alpha_3,\alpha_4,\alpha_5\dots\}$ be a $D_n$
  subdiagram in the Coxeter graph of Fig.~\ref{fig:coxeter} and
  $\tau\in\cF_\cox$ be the corresponding cone. Then $\Star(\tau)$ in
  $\cF_\cox$ is the Coxeter fan for $W(D_n)$ but $\Star(\tau)$ in
  $\cF_\rc$ is the Coxeter fan for $W(C_n)$.
\end{lemma}
\begin{proof}
  Note first that replacing either $\alpha_1$ or $\alpha_3$ by the
  $(-4)$-vector $\gamma=\alpha_3-\alpha_1$ transforms $\Delta$ into a
  $C_n$ Dynkin diagram. Also note that by \eqref{lem:saturation} the
  $\Delta$ root sublattice of $\II_{1,17}$ is saturated.

  The statement for $\Star(\tau)$ in $\cF_\cox$ is standard. The
  hyperplane $\gamma^\perp$ divides the fundamental chamber for
  $W(D_n)$ into two halves, each a fundamental chamber for
  $W(C_n)$. The reflection in $\gamma$ is not defined on
  $N=\II_{1,17}$ but it is well defined on $N/N_\tau$ which is the
  dual of the root lattice $D_n$, the same as for $C_n$. 
\end{proof}

\begin{remark}
  We will see in Section~\ref{sec:rat-curves-divisor} that the moduli
  of the corresponding stable surfaces are described by
  $T(C_n)/W(C_n)$, where $T(C_n)$ is the torus $\Hom(C_n, \bC^*)$. The
  map $T(D_n)/W(D_n) \to T(C_n)/W(C_n)$ is $2:1$.  This leads to an
  involution on a part of the fan $\cF_\rc$ and to the two cones
  $D_0$, $D'_0$ mapping to a unique stable surface of type $C_0$. This
  $D/C$ dichotomy appears to be the main reason for the refinement
  $\cF_\rc$ of $\cF_\cox$.
\end{remark}

\section{Degenerations of K3 surfaces and integral-affine spheres}
\label{sec:ia-pairs}

To prove that $\oF^\slc$ coincides with a toroidal
compactification, we extend the method developed in
\cite{alexeev2019stable-pair}. Central to this method is the notion of
an integral affine pair $(\ias,R_\ia)$ consisting of a singular integral-affine
sphere and an effective integral affine divisor on it.
From a nef model of a type $\III$ one-parameter degeneration, we
construct a pair $(\ias,R_\ia)$.  Vice versa, given a pair
$(\ias,R_\ia)$ we construct a combinatorial type of nef model.

\begin{definition} An {\it integral-affine structure} on an oriented real
  surface $B$ is a collection of charts to $\R^2$ whose transition
  functions lie in $\SL_2(\Z)\ltimes \R^2$. \end{definition}

On the sphere, such structures must have singularities.
We review some unpublished material from
 \cite{engel2018moduli-space} on these singularities.
 Let $\widetilde{\SL}_2(\R)\to \SL_2(\R)$ be the universal cover. This
 restricts to an exact sequence $$0\to \Z\to \widetilde{\SL}_2(\Z)\to \SL_2(\Z)\to 0.$$
Since $\SL_2(\R)$ acts on $\R^2\setminus \{0\}$, its universal cover and
the subgroup $\widetilde{\SL}_2(\Z)$ act on
$\widetilde{\R^2\setminus 0}$,
which admits natural polar coordinates $(r,\theta)\in \R^+\times \R$. A generator of the kernel
$\Z$ acts by the deck transformation $(r,\theta)\mapsto (r,\theta+2\pi)$.

\begin{definition}\label{naiveias} A {\it naive singular integral-affine structure} on $B$
is an integral-affine structure on the complement $B\backslash \{p_1,\dots,p_n\}$ 
of a finite set such that each point $p_i$ has a punctured neighborhood $U_i\setminus \{p_i\}$
modeled by an {\it integral-affine cone singularity}: The result of gluing a circular sector $$
\{ \theta_1 \leq \theta \leq \theta_2\}\subset \widetilde{\R^2\setminus 0}$$
along its two edges $\theta=\theta_1,\theta_2$ by an element of $\widetilde{\SL}_2(\Z)$.
 \end{definition}
 
 \begin{definition}  Let $(B,p)$ be an integral-affine cone singularity.
 We may assume that $\theta_1, \theta_2$ have rational slopes.
 Decompose $\theta_1\leq \theta\leq \theta_2$ into
 standard affine cones, i.e. regions $\SL_2(\Z)$-equivalent
 to the positive quadrant. Let $\{\vec{e}_1,\dots,\vec{e}_n\}$
 denote the successive primitive integral vectors pointing
 along the one-dimensional rays of this
 decomposition. Define integers $d_i$ by the formula
 $$\vec{e}_{i-1}+\vec{e}_{i+1}=d_i\vec{e}_i$$ using the
 gluing to define $d_1$. Then the {\it charge} is $$Q(B,p):=12+\sum (d_i-3)$$ and
 does not depend on the choice of 
 decomposition into standard affine cones.
  \end{definition}

By \cite{engel2018moduli-space, kontsevich2006affine-structures}, 
a naive singular integral-affine
structure on a compact oriented surface $B$ of genus $g$
satisfies $\sum Q(B,p_i) = 12(2-2g)$. As we are interested in
the sphere, the sum of the charges of singularities is $24$. This formula was
first proven by \cite[Prop.~3.7]{friedman1983smoothing-cusp} in the context
of the dual complex of a Kulikov degeneration, see Thm.~\ref{thm:ias-to-kulikov}.
For application to degenerations of K3 surfaces, we
need a more refined notion of integral-affine singularity.

\begin{definition} An {\it anticanonical pair} $(Y,D)$ is a smooth rational surface
$Y$ and an anticanonical cycle $D=D_1+\dots+D_n \in |-K_Y|$ of rational curves. 
Define $d_i:=-D_i^2$. \end{definition}

\begin{definition} The {\it naive pseudo-fan} $\mathfrak{F}(Y,D)$ of
  an anticanonical pair is a integral-affine cone singularity
  constructed as follows: For each node $D_i\cap D_{i+1}$ take a
  standard affine cone $\R_{\geq 0}\{\vec{e}_i,\vec{e}_{i+1}\}$ and
  glue these cones by elements of $\SL_2(\Z)$ so that
  $\vec{e}_{i-1}+\vec{e}_{i+1} = d_i \vec{e}_i$. 
\end{definition}

\begin{remark}\label{cbec-invt} 
Note that the cone singularity itself does not keep track of the rays.
For instance, blowing up the node
$D_i\cap D_{i+1}$ produces a new anticanonical pair
$(Y',D')\to (Y,D)$ whose naive pseudo-fan $\mathfrak{F}(Y',D')$ is identified
with $\mathfrak{F}(Y,D)$. The standard affine cone
$\R_{\geq 0}\{\vec{e}_i,\vec{e}_{i+1}\}$ is subdivided in two.
The charge $Q(Y,D):=Q(\mathfrak{F}(Y,D))$
is invariant under such a {\it corner blow-up}. \end{remark} 

\begin{definition} The {\it c.b.e.c.} (corner blow-up equivalence class)
of $(Y,D)$ is the equivalence class of anticanonical pairs which can be
reached from $(Y,D)$ by corner blow-ups and blow-downs. \end{definition}

Remark \ref{cbec-invt} implies that $\mathfrak{F}(Y,D)$ depends only on
the c.b.e.c. of $(Y,D)$. 

\begin{definition} A {\it toric model} of a c.b.e.c. is a choice of
  representative $(Y,D)$ and an {\it exceptional collection}: A sequence
  of  $Q(Y,D)$ successively contractible $(-1)$-curves
  which are not components of $D$.  The blowdown
  $(\overline{Y},\overline{D})$ is a {\it toric pair}, i.e. a toric
  surface with its toric boundary. We call these {\it internal
    blow-ups}.
\end{definition} 

\begin{definition}\label{def:true-sing} An {\it integral-affine singularity}
is an integral-affine cone singularity isomorphic
to $\mathfrak{F}(Y,D)$ for some anticanonical pair $(Y,D)$,
with a multiset of rays $\{\vec{e}_i\}$
corresponding to the components $D_i\subset D$
meeting an exceptional collection. The {\it pseudo-fan}
$\mathfrak{F}(Y,D)$ is the naive pseudo-fan, equipped with this data.
 \end{definition}
 
 Note that the components $D_i\subset D$ meeting an exceptional collection uniquely determine
 the deformation type of the anticanonical pair $(Y,D)$. 
 
 \begin{definition}\label{iso} Let $\phi\colon \mathfrak{F}(Y,D) \rightarrow \mathfrak{F}(Y',D')$
 be an isomorphism of integral-affine cone singularities.
 We say that $\phi$ is an {\it isomorphism of
  integral-affine singularities} if the two multisets of rays
  $\{\phi(\vec{e}_i)\}$ and $\{\vec{e}_i'\}$ determine the same
 deformation type. \end{definition}
 
Equivalently, after making corner blow-ups on $(Y',D')$ until
 the rays $\phi(\vec{e}_i)$ all form edges of the decomposition of $\mathfrak{F}(Y',D')$
 into standard affine cones, the pair $(Y',D')$ admits
 an exceptional collection meeting the components corresponding to $\phi(\vec{e}_i)$. 
From the definitions, integral-affine singularities, up to isomorphism, are in bijection with
c.b.e.c.s of deformation types of anticanonical pairs $(Y,D)$.
We are now equipped to remove the word ``naive" in Definition \ref{naiveias}.

 \begin{definition} An {\it integral-affine sphere}, or $\ias$ for short,
 is an integral-affine structure
 on the sphere with integral-affine singularities as in Definition \ref{def:true-sing}. \end{definition}
 
In particular, there is a forgetful map from $\ias$ to naive $\ias$
which forgets the data of the multisets of outgoing rays from each singularity.

\begin{definition} Let $(\vec{v}_1,\dots,\vec{v}_k)$ be a counterclockwise-ordered
sequence of primitive integral vectors in $\R^2$
and let $n_i$ be positive integers. We define an integral-affine singularity
$(B,p) = I(n_1\vec{v}_1,\dots,n_k\vec{v}_k)$ by declaring
$(B,p)=\mathfrak{F}(Y,D)$ where $(Y,D)$ is a blow-up of
a smooth toric surface $(\overline{Y},\overline{D})$ whose fan contains the rays
$\R_{\geq 0} \vec{v}_i$ at $n_i$ points
on the component $\overline{D}_i$ corresponding to $\vec{v}_i$.
\end{definition} 

Every c.b.e.c. admits some toric model and hence can be
presented in the form $I(n_1\vec{v}_1,\dots,n_k\vec{v}_k)$.
Since $Q(I(n_1\vec{v}_1,\dots,n_k\vec{v}_k)) = \sum n_i \geq 0$,
 an integral-affine surface with singularities, as defined,
is either a non-singular $2$-torus, or the $2$-sphere.

\begin{definition} Define the {\it$I_k$ singularity} as $I(k\vec{e})$. It has charge $k$. \end{definition}

\begin{remark}\label{rem:i1s} If an $\ias$ has all $I_1$ singularities
there are $24$ such. There is only one integral-affine singularity which underlies the naive cone
singularity of $I(\vec{e})$, corresponding to either marking the ray $\vec{e}$ or $-\vec{e}$.
Hence in the case where all $24$ charges are distinct,
there is no difference between
a naive $\ias$ and an $\ias$. \end{remark}

\begin{definition}
  An $\ias$ is {\it generic} if it has $24$ distinct $I_1$
  singularities.
\end{definition}

The relevance of these definitions lies in the following:

\begin{theorem}\label{thm:ias-to-kulikov} Let $X\rightarrow C$ be a Type III Kulikov model.
The dual complex $\Gamma(X_0)$ has
the structure of an $\ias$, triangulated into lattice triangles
of lattice volume 1.
Conversely, such a triangulated $\ias$ with singularities
at vertices determines a 
Type III central fiber $X_0$ uniquely up to
topologically trivial deformations.
\end{theorem}

\begin{proof}  See \cite{engel2018looijenga} or
\cite[Rem1.11v1]{gross2015mirror-symmetry-for-log} for the
  forward direction. Roughly, one glues together unit volume lattice triangles
  by integral-affine maps, in such a way
  that the vertex $v_i$ corresponding to a component $V_i\subset X_0$
  has integral-affine singularity
  $\mathfrak{F}(V_i,D_i)$. Here $D_i=\sum_j D_{ij}$ and $D_{ij}:=V_i\cap V_j$
  are the double curves lying on $V_i$.  
  For the reverse direction, one glues together the
  anticanonical pairs $(V_i,D_i)$ whose pseudo-fans 
  model the vertices of the triangulated $\ias$. The gluings
  are ambiguous, but all such gluings give homeomorphic surfaces $X_0$
  which are related by topologically trivial deformations.
  \end{proof} 

\begin{definition}\label{IAdivisor}
  Let $B$ be an $\ias$. An {\it integral-affine divisor} $R_\ia$ on
  $B$ consists of two pieces of data:
  \begin{enumerate}
  \item A weighted graph $R_\ia\subset B$ with vertices $v_i$,
    rational slope line segments as edges $v_{ij}$, and integer labels
    $n_{ij}$ on each edge. 

  \item Let $v_i\in R$ be a vertex and $(V_i ,D_i)$ be an
    anticanonical pair such that $\mathfrak{F}(V_i ,D_i)$ models $v_i$
    and contains all edges of $v_{ij}$ coming into $v_i$. We require
    the data of a line bundle $L_i \in \textrm{Pic}(V_i)$ such that
    $\deg{L_i}\big{|}_{D_{ij}}=n_{ij}$ for the components $D_{ij}$ of
    $D_i$ corresponding to edges $v_{ij}$ and $L_i$ has degree zero on
    all other components of $D_i$.
  \end{enumerate}
\end{definition} 

\begin{definition}\label{def:polarizing-ia-div}
  A divisor $R_{\ia}\subset B$ is {\it polarizing} if each line bundle
  $L_i$ is nef and at least one $L_i$ is big. The {\it
    self-intersection} is
  $R_\ia^2:=\sum_i L_i^2\in \Z_{>0}.$
\end{definition}

\begin{definition}
  Given an nef model $L\rightarrow X$,
  we get an integral-affine divisor
  $R_\ia\subset B=\Gamma(X_0)$ by simply restricting
  $L$ to each component. Since $L$ is nef,
  the divisor $R_\ia$ is effective i.e. $n_{ij}\geq 0$.
\end{definition}

\begin{remark}\label{balancing1}
  When $v_i\in R_\ia$ is non-singular, the pair $(V_i,D_i)$ is toric,
  and the labels $n_{ij}$ uniquely determine $L_i$. They must satisfy
  a balancing condition. If $\vec{e}_{ij}$ are the primitive integral
  vectors in the directions $v_{ij}$ then one must have
  $\sum n_{ij}\vec{e}_{ij}=0$ for such a line bundle $L_i$ to exist.

  Similarly, if $I_1=\mathfrak{F}(V_i,D_i)=I(\vec{e})$
  i.e. $(V_i,D_i)$ is a single internal blow-up of a
  toric pair, the $n_{ij}$ determine a unique line bundle $L_i$ so
  long as $\sum n_{ij}\vec{e}_{ij} \in \Z\vec{e}$. This condition is
  well-defined: the $\vec{e}_{ij}$ are well-defined up to shears in the
  $\vec{e}$ direction.
\end{remark}

Let $B$ be a lattice triangulated $\ias$ or equivalently,
$B=\Gamma(X_0)$ is the dual complex of a Type III 
degeneration. When $B$ is generic, an integral-affine divisor
$R_\ia\subset B$ is uniquely specified by a weighted graph satisfying
the balancing conditions of Remark \ref{balancing1}, so the extra
data (2) of Definition \ref{IAdivisor} is unnecessary.

\begin{definition}
  An integral-affine divisor $R_\ia\subset B$ is {\it compatible} with
  a triangulation if every edge of $R_\ia$ is formed from edges of the
  triangulation.
\end{definition}

If $B$ comes with a triangulation, we assume that an integral-affine
divisor is compatible with it.

\section{Compactification for the ramification divisor}
\label{sec:ramification-divisor}

\begin{theorem}\label{thm:ramif-comp}
  The normalization of the stable pair compactification $\slcram$ is
  the toroidal compactification $\torram$. 
\end{theorem}
\begin{proof}
  Let $\II_{3,19}\ni h$ be the K3 lattice and a vector with $h^2=2$.
  Denote by $N_2$ the lattice
 $h^\perp = \II_{1,1} \oplus \II_{1,17} \oplus \la -2\ra$.  of signature
  $(2,19)$ and by $\bD_2$ be the corresponding Type IV domain. It is
  well known that the moduli space of polarized K3 surfaces $(X,L)$ of
  degree $2$ with $ADE$ singularities is the arithmetic quotient
  $F_2 = O^*(N_2)\backslash\bD_2$ for a finite subgroup
  $O^*(N_2)\subset O(N_2)$.

  There are two $O^*(N_2)$-orbits of vectors $v\in N_2$ with $v^2=-2$,
  with representatives $v_1$ and $v_2$ of divisibility $1$,
  resp. $2$ in $N_2^*$. They define two hyperplanes $v_k^\perp$ in $\bD_2$ and
  two Heegner divisors in $F_2$, for the nodal and unigonal K3
  surfaces. The second hyperplane $\bD=v_2^\perp$ is the Type IV
  domain for the lattice $N_\el = \II_{2,18} \subset N_2$,
  and its arithmetic quotient is our space $F=F_\el$.

  There are single orbits of primitive square $0$ vectors in $N_2$ and in
  $N_\el$. Let us a choose a representative $e\in N_\el\subset N_2$.
  Baily-Borel compactifications $\oF_2^\bb$ and $\oF_\el^\bb$ both have a
  single $0$-cusp.  A toroidal compactification of $F_2$,
  resp. $F_\el$, is described by a single fan supported on the light
  cone $\oC_\bQ$ for the lattice $e^\perp/e$, where $e^\perp$ is taken
  in $N_2$, resp. $N_\el$. One has
  $e^\perp/e = \II_{1,17}\oplus\la -2\ra$, resp.
  $e^\perp/e = \II_{1,17}$.  In particular, the Coxeter fans
  $\cF^2_\cox$, resp. $\cF_\cox = \cF^\el_\cox$, is defined by intersecting
  $\oC_\bQ$ by the hyperplanes $\alpha^\perp$ orthogonal to the roots
  in $e^\perp/e$. It follows that the toroidal compactification
  $\oF_\el^{\cF^\el_\cox}$ is the closure of $F_\el$ in the toroidal
  compactification $\oF_2^{\cF^2_\cox}$. 

  The fundamental domain in $\cF^2_\cox$ is described by the Coxeter
  diagram with 24 vertices (fundamental roots) $\alpha_i$ pictured in
  \cite[Fig. 4.1]{alexeev2019stable-pair}. The roots of divisibility
  $2$ are $\alpha_{21}, \alpha_{22}, \alpha_{23}$. Let us take
  $v_2=\alpha_{23}$. Then the hyperplane $\alpha_j^\perp$ intersects
  $v_2^\perp$ iff $|\alpha_{23} \cdot \alpha_i| \le 2$. Thus, the
  Coxeter diagram for $N_\el$ is obtained from that for $N_2$ by
  removing the nodes $\alpha_{21}, \alpha_{22}, \alpha_3$, and the
  result is precisely the Coxeter diagram of Fig.~\ref{fig:coxeter}
  for lattice $\II_{1,17}$. 

  It is shown in \cite{alexeev2019stable-pair} that the normalization
  of the stable pair compactification $\oF_2^\ram$ for the
  ramification divisor is a semitoric compactification for the semifan
  $\cF^2_\ram$ that is the coarsening of the Coxeter fan $\cF^2_\cox$
  obtained by reflecting the fundamental domain by the Weyl group
  $W_2$ generated by reflections in the six ``irrelevant'' roots
  $\alpha_{18}, \dotsc, \alpha_{23}$. This group is infinite, and so
  $\cF^2_\ram$ is a semifan and not a fan; the maximal-dimensional
  cones are not finitely generated.

  It follows that $\oF_\el^\ram$ is the closure of $F_\el$ in
  $\oF_2^\ram$ and its normalization is the semitoric compactification
  for the fan $\cF^\el_\ram = \cF^2_\ram \cap v_2^\perp$. Thus, it is
  the semifan obtained by reflecting the fundamental domain of
  $\cF^\el_\cox$ by the Weyl group $W_\el$ generated by reflections in
  ``irrelevant'' roots $\alpha_{18}, \dotsc, \alpha_{23}$ that are
  not $\alpha_{21}, \alpha_{22}, \alpha_3$ and
  $\alpha_{23}=v_2$ itself. In Fig.~\ref{fig:coxeter} these are the two
  roots denoted $\alpha_1$ and $\alpha_{19}$. Since the two vertices 1
  and 19 are disjoint, one has $W_\el = \bZ_2\oplus\bZ_2$, the semifan
  $\cF^\el_\ram$ is in fact a fan, and the semitoric compactification
  $\oF_\el^{\cF_\ram}$ is toroidal.
\end{proof}

\begin{figure}[htp]
  \centering
  \includegraphics[width=.6\linewidth]{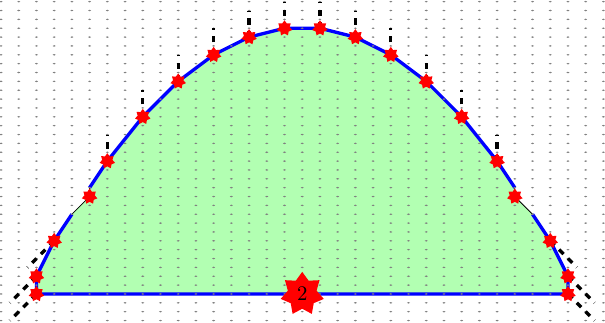}
  \caption{$(\ias, R_\ia)$ for the ramification polarization divisor}
  \label{fig:ramif-ias} 
\end{figure}

\begin{remark}
In \cite{alexeev2019stable-pair} the degenerations of degree~2 K3
pairs $(X,\epsilon R)$ are described by the integral-affine pairs
$(\ias, R_\ia)$ of \cite[Fig.9.1]{alexeev2019stable-pair}.
Following the proof of the above theorem, the pairs for $\oF^\ram$ 
are obtained by setting $a_{23}=0$, i.e. closing
the gap in the second presentation of \emph{loc. cit.}  We give the
result in Fig.~\ref{fig:ramif-ias}. The picture
shows the upper hemisphere, and the entire sphere is glued from two
copies like a taco or a pelmeni (a dumpling).  The polarizing divisor
is the equator; it is drawn in blue.

The divisor models and stable models can be read off from the pair
$(\ias, R_\ia)$: The divisor $R$ is the fixed locus of an involution
on the Kulikov model which acts on the dual complex by
switching the two hemispheres.
Irreducible components of the stable model correspond to the vertices
of $R_\ia$. Fig.~\ref{fig:ramif-ias} gives a stable model
with the maximal possible number 18 of irreducible components.
\end{remark}

\section{Compactification for the rational curve divisor}
\label{sec:rat-curves-divisor}

\subsection{Kulikov models of type III degenerations}
\label{sec:K-models}

Let $L,R\in\{1,2,3\}$.
Consider the following $19$ vectors in $(\frac{1}{2}\Z)^2$
\begin{align*} \vec{v}_1 &= \twopartdef{(0,1)}{L = 2,3}{(1,\tfrac{9}{2})}{L=1} \\
\vec{v}_i &= (1,\tfrac{10-i}{2})\textrm{ if }i=2,\dots,18 \\ 
\vec{v}_{19} &= \twopartdef{(0,-1)}{R = 2,3}{(1,-\tfrac{9}{2})}{R=1.} \end{align*}
Let $\ell=(\ell_1,\dots,\ell_{19})\in \Z_{\geq 0}^{19}$ be non-negative integers,
satisfying
the condition that $\sum \ell_i\vec{v}_i$ is a horizontal vector.
 
Form a polygon $P_{LR}(\ell)$ whose edges are the vectors
$\ell_i\vec{v}_i$ put end-to-end in the plane, together with a segment
on the $x$-axis. For instance $P_{1,2}(2,\dots,2,9)$ is shown in 
Fig.~\ref{fig:rc-ias}. Let $Q_{LR}(\ell)$ be the
lattice polygon which results from taking the union of $P_{LR}(\ell)$
with its reflection across the $x$-axis.

\begin{figure}[htp]
  \centering
  \includegraphics[width=.6\linewidth]{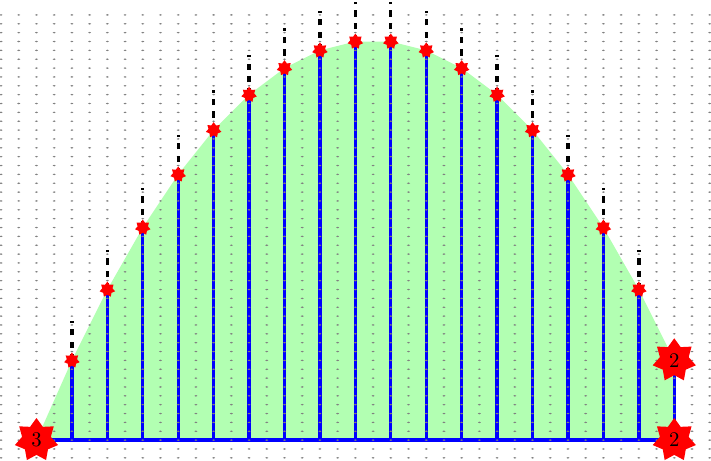}
  \caption[$(\ias, R_\ia)$ for the rational curve divisor]
  {$(\ias, R_\ia)$ for the rational curve polarization divisor.
    End behaviors: $L=1$, $R=2$ or $3$.}
  \label{fig:rc-ias}
\end{figure}

\begin{definition}\label{def:rc-ias} Define $B_{LR}(\ell)$, a naive singular $\ias$, as follows:
Glue each edge $\ell_i\vec{v}_i$ of $Q_{LR}(\ell)$ to its reflected edge
by an element of $\SL_2(\Z)\ltimes \R^2$ which preserves vertical lines.
This uniquely specifies the gluings, except when $\ell_1, \ell_{19}>0$
and $L,R\in\{2,3\}$ respectively. For these edges, we must specify the gluing
to be $-A^4$ where $A(x,y) = (x+y,y)$ is a unit vertical shear.
 \end{definition}
 
 \begin{remark} As naive $\ias$, we have that $B_{LR}(\ell)$ are isomorphic
 when we interchange the end behaviors $2\leftrightarrow 3$.
 It is only when we impose the extra data as in Definition \ref{def:true-sing}
 that we can distinguish them.\end{remark}
 
From Definition \ref{def:rc-ias}, we determine the $\SL_2(\Z)$-monodromy
of the naive $\ias$. Assume for convenience that all $\ell_i>0$.
Let $g_i\in \pi_1(B_{LR}(\ell) \setminus \{p_i\},\,*)$ for $i=1,\dots,20$ be simple 
counterclockwise loops 
based at a point $*$ in the interior of $Q_{LR}(\ell)$,
which successively enclose the singularities of $B_{LR}(\ell)$ from left to
right. Then the $\SL_2(\Z)$-monodromies
are: 
\begin{align*} 
  \rho(g_1) &= A^9 \textrm{ if }L=1,\hspace{10pt}
              \rho(g_1)=\rho(g_2) = -A^4\textrm{ if }L=2,3 \\
  \rho(g_{20}) 
            &= A^9\textrm{ if }R=1, \hspace{10pt}
              \rho(g_{19})=\rho(g_{20}) = -A^4\textrm{ if }R=2,3 \\
  \rho(g_i)&=A^{-1}\textrm{ for all remaining }i.
\end{align*}
When some $\ell_i=0$,
the monodromy of the resulting cone singularity is the product.

\begin{remark} 
  The image of the $\SL_2(\Z)$-monodromy representation of
  $B_{LR}(\ell)$ lands in the abelian group $\pm A^{\Z}$. This is
  related to the existence of a broken elliptic fibration on the
  corresponding Kulikov models.  When all $24$ singularities are
  distinct, the monodromy of an $\ias$ is never abelian, because the
  sphere would then admit a non-vanishing vector field. Here, we
  always have some singularity of charge $\ge2$. 
\end{remark}

Next, we enhance $B_{LR}(\ell)$ from a naive $\ias$ to 
an $\ias$:

\begin{definition}\label{enhanced-ias} 
The multisets of rays (cf. Definition \ref{def:true-sing})
giving toric models of the anticanonical pairs
whose pseudo-fans
model each singularity are listed
in Table~\ref{tab:sings-on-ias}. The rays
are chosen with respect to 
the open chart $Q_{LR}(\ell)$ on $B_{LR}(\ell)$.
The marked rays for right end $R$ are analogous, 
but reflected across the $y$-axis.
\begin{table}[h!]
\caption{Pseudofans modeling each singularity, for the left end type $L$}
\label{tab:sings-on-ias}
\begin{tabular}{c | l | l | c }
$L$ & Singularity & Marked rays & Notation \\ 
\hline
$1$ & $\ell_1\neq 0$, end singularity & $(1,-3),(1,0),(1,3)$ & $X_3$ \\
$1,2$ & $\ell_1=0$, $\ell_2\neq 0$ & $(1,-2),(1,0),(1,1),(1,3)$ & $X_4$ \\
$3$ & $\ell_1=0$, $\ell_2\neq 0$ & $(1,-2), (1,0),(1,2),(1,4)$ & $\iX_4$ \\
$1,2, 3$ & $\ell_i=0$ for $i\leq k$, $k\geq 2$ & All choices equivalent & $X_{k+3}$ \\
$2,3$ & $\ell_1\neq 0$, inner singularity & $(1,0),(1,2)$ & $Y_2$ \\
$2$ & $\ell_1,\ell_2\neq 0$, outer singularity & $(1,1), (1,3)$ & $Y_2$ \\
$3$ & $\ell_1,\ell_2\neq 0$, outer singularity & $(1,2),(1,4)$ & $\iY_2$ \\
$2,3$ & $\ell_1\neq 0,\,\ell_i=0$ for $2\leq i \leq k$ & All choices equivalent & $Y_{k+1}$ \\
 & $\ell_{i+j}=0$, $1\leq j\leq k$ in interior & $(0,-1)$, multiplicity $k$ & $I_k$ \\
\end{tabular}
\bigskip
\caption{All allowable combinatorial types of $\ias$}
\label{tab:all-types}
  \begin{tabular}[htp]{ c | c |  c | c | c }
& $L$ Symbol(s) &   Intermediate Symbols & $R$ Symbol(s) &   \\
    \hline
    $L=1$ & $X_3$ & $I_{1+n_1}\cdots I_{1+n_k}$, $n_i\geq 0$ &     $X_3$  & $R=1$  \\
    $L=2$ & $Y_2Y_2$ & &     $Y_2Y_2$ & $R=2$ \\
    $L=3$ & $Y_2Y_2'$ or $X_4'$ & &    $Y_2'Y_2$ or $X_4'$ & $R=3$ \\
    $L=1,2$ & $X_4$ & &     $X_4$ &  $R=1,2$ \\
    $L=2,3$ & $Y_2Y_{2+n}$, $n\geq 1$ & &    $Y_2Y_{2+n}$, $n\geq 1$ &  $R=2,3$ \\
    $L=1,2,3$ & $X_{3+n}$, $n\geq 2$ & &    $X_{3+n}$, $n\geq 2$ &   $R=1,2,3$
    \end{tabular}
\end{table} 
\end{definition}

When an end is an isolated point, the symbol $X$ is used.
When the left end is a vertical segment
the symbols $Y$ are used for the so-called
{\it inner} and {\it outer} singularities at the points $p_1$ and $p_2$, respectively.
The same applies to $p_{20}$ and $p_{19}$ at the right end.
For instance, in Fig.~\ref{fig:rc-ias}, there 
is one left-most singularity, labeled $X_3$. There are two right-most singularities.
Both are labeled $Y_2$ and the upper right-most singularity in the figure
is the ``outer singularity." The lower right-most singularity is the ``inner singularity."
Intermediate singularities are labeled $I_k$ and in
Fig.~\ref{fig:rc-ias}, specifically $I_1$.

The singularities notated $Y_2$ and $\iY_2$ are abstractly isomorphic,
but the prime is necessary to
distinguish how the marked rays sit on the sphere $B_{LR}(\ell)$
at the outer singularity. This is distinguishes Ends $2$ and $3$, respectively.


\begin{notation}\label{d-notation} Table \ref{tab:sings-on-ias} allows for very succinct notation
for the types of $\ias$ that appear in our construction.
For instance, if $(L,R)=(3,2)$ and $\ell_i\neq 0$
for exactly $i=2,5,6,16,19$ then we say that $B_{LR}(\ell)$
is of {\it combinatorial type} $\iX_4I_3I_1I_{10}Y_4 Y_2$
indicating the sequence of singularities one sees traveling along the vectors $\ell_i\vec{v}_i$.
The subscripts denote the charges, so they always add to $24$. As another example,
Fig.~\ref{fig:rc-ias} has an $\ias$ of combinatorial type $X_3I_1^{17}Y_2Y_2$
assuming that $R=2$. If $R=3$, the combinatorial type is instead
$X_3I_1^{17}Y_2'Y_2$. Generally, all allowable combinatorial types can be formed
by concatenating symbols as in Table \ref{tab:all-types} in an arbitrary manner, choosing
one symbol out of each column, in such a way that the sum of all indices is $24$,
and ensuring that no $X$-symbol has an index of $12$ or more.
%
\end{notation}

\begin{lemma}
  The types of the $\ias$ defined above are in a bijection with the
  types III cones in the fan $\cF_\rc$ of Lemma~\ref{lem:rc-cones} via
  the correspondence of symbols 
  $E_n=X_{n+3}$, $E'_1=X'_4$, $D_n=Y_2Y_{2+n}$, $D'_0=Y_2Y'_2$,
and $A_n=I_{n+1}$.
\end{lemma}

\begin{proof} 
  We have defined 9 maximal dimensional cones in $\mathcal{F}_\rc$ modulo
  $W$ and 9 types of $\ias$, with the Dynkin labels
  $(E_0|D_0|D'_0)A_0^{18|17|16}(E_0|D_0|D'_0)$ and with
  combinatorial types $(X_3|Y_2Y_2|Y_2Y_2')I_I^{18|17|16}(X_3|Y_2Y_2|Y_2'Y_2)$,
  respectively. For each type, an
  $\ias$ is defined by the collection of positive numbers $\ell_i$
  satisfying a single linear relation, that the height difference from
  the left end to the right end is zero. This linear relation between the $\ell_i$
  has 9 positive coefficients, 1 zero coefficient, and 9 negative coefficients.

  On the other hand, a point $\lambda$ in a maximal cone is defined by
  a collection of 19 nonnegative numbers, the intersection numbers
  between $\lambda$ and the 19 vectors among $\alpha_i$, $\beta_L$,
  $\gamma_L$, $\beta_R$, $\gamma_R$ that give the facets of this
  cone. These intersection numbers satisfy the relations given in
  equation~\eqref{eq:lin-relations} with the same sign pattern.  In fact,
  one checks that the formulas given in Cor.~\ref{final-monodromy}
  give an explicit bijection between lattice points in the interiors of the 9
  maximal cones of $\mathcal{F}_\rc$ and $\ias$
  of the corresponding combinatorial type with all $\ell_i>0$. This
   bijection extends to the faces the maximal cones, by allowing some
   $\ell_i=0$, and giving the symbol substitution rules described in the lemma.
\end{proof}



\medskip

We now decompose $B_{LR}(\ell)$ into unit width vertical strips
(in fact these are integral-affine cylinders). Cut these cylinders by the horizontal line
along the base of $P_{LR}(\ell)$
joining the left to the right end, to form a collection of unit width trapezoids, and triangulate
each trapezoid completely into unit lattice triangles.

\begin{remark}\label{rem:odd-issue}  If $\ell_i$ is odd for some odd $i$,
the singularities of $B_{LR}(\ell)$ may not
lie at integral points. In these cases, we can adjust the location of the singularity by moving
it vertically half a unit. This destroys the involution symmetry of $B_{LR}(\ell)$, but
the singularities of $B_{LR}(\ell)$ will be vertices of the triangulation.
Alternatively, we could just triangulate $B_{LR}(2\ell)$ in the same manner,
but our current approach allows
for a wider range of valid $\ell$ values.
\end{remark}
 
\begin{definition}\label{def:kulikov-from-sphere}
Define $X_{0,LR}(\ell)$ to be the unique deformation type of Type III Kulikov surface
associated to the triangulated $B_{LR}(\ell)$ by Theorem \ref{thm:ias-to-kulikov}.\end{definition}

Shifting singularities and replacing $\ell\mapsto 2\ell$ 
as in Remark \ref{rem:odd-issue} has the effect  \cite[Sec.~4]{engel2021smoothings}
of birational modifications and an order $2$ base change
to the Kulikov model in Definition \ref{def:kulikov-from-sphere},
neither of which ultimately affect the stable model.

\begin{example} The deformation type of an anticanonical pair $(V,D)$ forming a component
of $X_{0,LR}(\ell)$ can be quickly read off from Table \ref{tab:sings-on-ias}. For instance,
the singularity $\iX_4$ is the result of gluing the circular sector $\R_{\geq 0}\{(1,-4),(1,4)\}$ by
$A^8(x,y)=(x,8x+y)$ and has the rays $(1,-2),(1,0),(1,2),(1,4)$ marked. To realize this singularity
as a pseudo-fan we should further decompose the circular sector into standard affine cones
so that the one-dimensional rays
are $\vec{e}_n=(1,n)$ for $n=-4,\dots,4$. By the formula $\vec{e}_{i-1}+\vec{e}_{i+1} = -D_i^2 \vec{e}_i$
we have that the anticanonical cycle of $(Y,D)$
consists of eight $(-2)$-curves---computing
$-D_4^2$ requires taking indices mod $8$ and performing the gluing.

The marked rays indicate that four disjoint exceptional curves
meet $D_{-2}$, $D_0$, $D_2$, $D_4$.
Blowing these down gives the unique toric surface
whose anticanonical cycle has self-intersections
$(-1,-2,-1,-2,-1,-2,-1,-2)$, which is itself the blow-up
of $\mathbb{P}^1\times \mathbb{P}^1$ at the four corners of an anticanonical
square. \end{example} 

\subsection{Nef and divisor models of degenerations}
\label{sec:D-models}

We assume henceforth that our polarizing divisor is $R=s+\sum f_i$.
The case $R=s+m\sum f_i$ is treated similarly, by simply 
adding factors of $m$ to anything vertical.

Define a polarizing divisor $R_{\ia}$ on every $\ias$
of the form $B_{LR}(\ell)$ as follows:
The underlying weighted graph of $R_{\ia}$ is the 
union of the following straight lines:
\begin{enumerate}
\item the horizontal line joining the two ends, with label $n_{ij}=1$, and
\item the vertical line through any singularity, with label $n_{ij}=Q$, where
$Q$ is the total charge of the singularities on the vertical line. 
\end{enumerate}
See Figure \ref{fig:rc-ias}, where the graph
is shown in blue 
(note that a copy is reflected across the $x$-axis). In the example,
the label of the right-hand
vertical blue segment is $4$.

To give a complete definition of $R_{\ia}$ as in
Definition \ref{IAdivisor} requires choosing various line bundles. It is simpler
to directly specify the divisor model by giving a
divisor $R_i$ on each component of
$V_i\subset \mathcal{X}_{0,LR}(\ell)$ with appropriate intersection
numbers with the double curves, i.e. $R_i\cdot D_{ij}=n_{ij}$. 
These are listed in Table \ref{tab:divs-on-ias}
and require some explanation.

\medskip
${X_{k+3}\,(k\geq 0)}$, ${\iX_4:}$
The end component $(V,D)$ is an anticanonical pair
with $D$ a cycle of $(-2)$-curves of length $9-k$. Thus, $(V,D)$ is in the deformation
type of an elliptic rational surface with $D$ a fiber of 
Kodaira type $I_{9-k}$. We assume that $(V,D)$ is in fact elliptic.
The $f_i$ in Table \ref{tab:divs-on-ias} are the $Q(V,D)=k+3$
singular elliptic fibers not equal to $D$ and $s$ is a section.
When $Q=4$, the two cases $X_4$ and $\iX_4$
are the two different deformation types of pairs $(V,D)$
with a cycle of eight $(-2)$-curves. In the $\iX_4$ case,
$\oplus \Z D_i$ is an imprimitive
sublattice of $H^2(Y,\Z)$; in the $X_4$ case it is a primitive sublattice.

\smallskip
{Inner ${ Y_2}$:} Taking $(1,0),(0,1)$ to be the rays of the pseudo-fan
with polarization degrees $1$ and $Q$ respectively,
we get a pair $(\mathbb{F}_1 , D_1+D_2)$ with
$D_1^2=0$ and $D_2^2=4$.
Note $D_2$ is a bisection of the ruling on $\mathbb{F}_1$ with fiber
class $D_1$. Then $s$ is the $(-1)$-section
and $f_1$ and $f_2$ are the two fibers in the class of $D_1$
tangent to the bisection $D_2$. 
The fibers $f_i'$ are $Q-4$ other fibers in the same class as,
but not equal to $D_1$.
Here $Q$ is the total charge at the end.

\smallskip
{Outer ${Y_2}$ and ${\iY_2}$:} Taking $(0,-1),(1,4)$ to be the rays of the pseudo-fan
with polarization degrees $4$ and $0$ respectively,
we get $Y_2 = \mathfrak{F}(\mathbb{F}_1 , D_1+D_2)$ and
$\iY_2 = \mathfrak{F}(\mathbb{F}_0 , D_1+D_2)$
with $D_1^2=4$ and 
$D_2^2=0$ in both cases. Then $f_1$ and $f_2$
are the two fibers in the class of $D_2$ tangent to the bisection $D_1$.
Our notation with the prime indicates that $Y_2$ represents the
``primitive'' case, and $\iY_2$ the ``imprimitive'' case. 

\smallskip
${Y_{k+2}\,(k\geq 0):}$ Take $(0,-1),(1,4-k)$ to be the rays of the pseudo-fan.
This anticanonical pair $(V,D_1+D_2)$ has self-intersections $D_1^2=4-k$ and
$D_2^2=0$ respectively. It is the result of blowing up either of the previous
two cases at $k$ points on $D_1$. These cases coincide once $k>0$. Then $f_1$ and
$f_2$ are the pullbacks of the original two fibers tangent to the bisection, and the $f_i'$
are pullbacks of fibers which go through the
points blown up on $D_1$.

\smallskip
${I_k:}$ Take $(0,-1),(0,1)$ and two rays pointing left and right to be the 
rays of the pseudo-fan. Then $(V,D)$ is the blow-up of some
Hirzebruch surface $\mathbb{F}$ at $k$ points on a section. The $f_i$
are the pullbacks of fibers going through blown up points.

\smallskip
{Non-singular surfaces:} All non-singular surfaces $V_i$ are toric and ruled
over either of the double curves corresponding to the
vertical direction. The $f_i$ are fibers of this ruling. The total count of
fibers is $Q$ where $Q$ is the total charge on the vertical
line through the vertex $v_i\in B_{LR}(\ell)$.
At intersection points where the horizontal and
vertical lines of $R_{IA}$ meet, we include a section of the
vertical fibration. At an end of type $2$ or $3$, two of the fibers 
$f_1$ and $f_2$ are quadrupled.

\begin{table}
\caption{Divisors on each anticanonical pair}
\label{tab:divs-on-ias}
 \begin{tabular}{ l l }
Singularity & $R_i\subset V_i\subset X_{0,LR}(\ell)$ \\ 
\hline
$X_{k+3}$, $\iX_4$ & $s+\sum_{i=1}^{k+3} f_i$ \\
inner $Y_2$ & $s+2f_1+2f_2+\sum_{i=1}^{Q-4} f_i'$  \\
outer $Y_2$, $\iY_2$ & $2f_1+2f_2$  \\
 $Y_{k+2}$, $k>0$  & $2f_1+2f_2$+$\sum_{i=1}^k f_i'$   \\
$I_k$ & $\sum_{i=1}^k f_i$  \\
non-singular point at End $2,3$ & $4f_1+4f_2+\sum_{i=1}^{Q-4} (f_i'+f_i'')$  \\
non-singular intersection point of $R_{\ia}$ & $s+ \sum_{i=1}^Q f_i$ \\
non-singular point on vertical line of $R_{\ia}$ & $\sum_{i=1}^Q f_i$  \\
non-singular point not on $R_{\ia}$ & empty  \\
\end{tabular}
\vspace{5pt}
\end{table} 

\begin{definition}\label{def:sing-fibered} We say that $X_{0,LR}(\ell)$ is {\it fibered} if
\begin{enumerate} 
\item The end surfaces (for $X$-type ends) are elliptically fibered, and
\item A connected chain of fibers of the vertical rulings
glue to a closed cycle.
\end{enumerate}
Then $X_{0,LR}(\ell)$ admits a fibration of arithmetic genus $1$ curves over
a chain of rational curves.
We say it is furthermore {\it elliptically fibered} if sections $s$
on the components connecting the left and right ends glue to
a section of this fibration.
\end{definition}

\begin{remark} We henceforth assume that $X_{0,LR}(\ell)$ is glued in such a way
as to be elliptically fibered.
\end{remark}

\begin{remark} When the left end $L\in \{2,3\}$ and $\ell_1>0$,
the chain of fibers in
Definition \ref{def:sing-fibered} consists of one fiber on the
components corresponding to the inner and outer singularity,
and a sum of two fibers on
the intermediate surfaces. Thus, the genus $1$ curve loops through each
intermediate component twice: On its way up, and on its way down. \end{remark}

The number of nodes of the chain over which $X_{0,LR}(\ell)$ is fibered
is the $x$-component of $\ell_1\vec{v}_1+\cdots + \ell_{19}\vec{v}_{19}$
or alternatively the lattice length of
the base of $P_{LR}(\ell)$. The induced map of dual complexes is the projection of
$B_{LR}(\ell)$ onto the base of $P_{LR}(\ell)$, decomposed into unit intervals.

\begin{definition}\label{divisordef} To define the divisor model of $X_{0,LR}(\ell)$:
Assume that $X_{0,LR}(\ell)$ is elliptically fibered.
Choose divisors $R_i\subset V_i$ as prescribed by Table \ref{tab:divs-on-ias}
which glue to a Cartier divisor $R$ on
$X_{0,LR}(\ell)$ and so that the vertical components of $R$
are elliptic fibers. \end{definition}

\begin{definition}\label{def:very-sing} Let $X_{0,LR}(\ell)$ be elliptically fibered. We call the
vertical components of $R$ the {\it very singular fibers}.  \end{definition}

\begin{example} Consider $B_{21}(\ell)$ with $\ell_1=2$, $\ell_8=\ell_{16}=1$, and all other
$\ell_i=~0$. In Notation \ref{d-notation}, the combinatorial type is $Y_2Y_8I_8X_6$.
The polygon $Q_{21}(\ell)$ is shown in Figure \ref{ex:kul} and is decomposed into lattice
triangles with black edges. The decomposition refines the vertical unit strips.
The black circles indicate non-singular vertices and the red triangles are
the four (once glued) singular vertices $Y_2$, $Y_8$, $I_8$,~$X_6$.

The intersection complex of $X_{0,21}(\ell)$
is overlaid on the dual complex, with orange edges for double curves $D_{ij}$ and blue
vertices for triple points. The self-intersections $D_{ij}\big{|}_{V_i}^2$ are written in dark green and satisfy
the triple point formula
$D_{ij}\big{|}_{V_i}^2+D_{ij}\big{|}_{V_j}^2=-2$ which is necessary for being a Kulikov model.
The neon green indicates the section $s$ and the hot pink indicates the very singular fibers, with $\times N$
indicating that there are $N$ such vertical components of $R$ and $2(\times 2)$ indicating
that there are two such vertical components, each doubled.
\end{example}

\begin{figure}[htp]
\centering
  \includegraphics[width=5in]{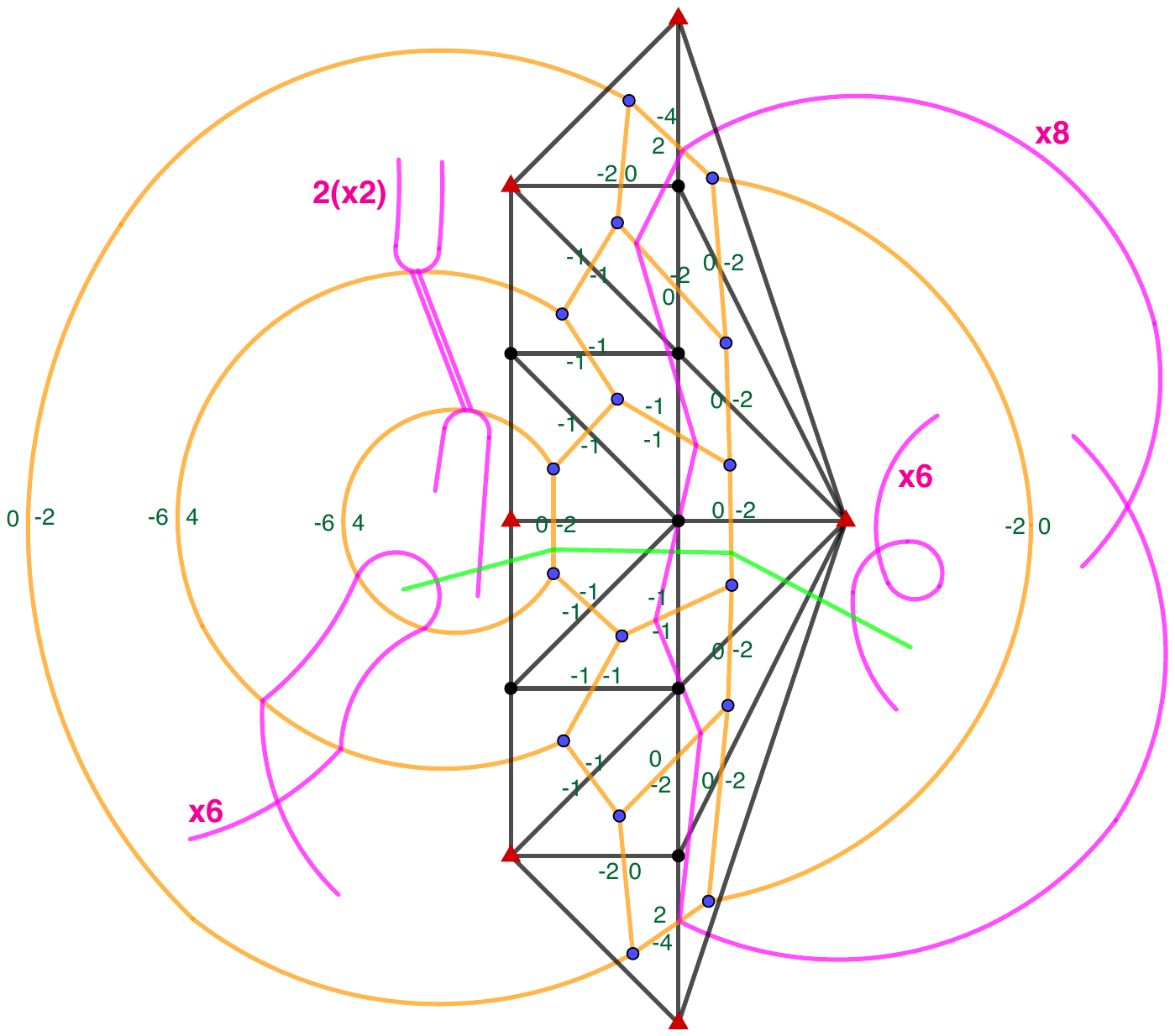}
  \caption{A divisor model of type $Y_2Y_8I_8X_6$.} 
  \label{ex:kul}
\end{figure}

\subsection{Moduli of $d$-semistable divisor models}
\label{sec:dss}

In this section we understand the condition of $d$-semistability
on our elliptically fibered surfaces $X_{0,LR}(\ell)$. 
Let $X_0$ denote a {\it Kulikov surface}, that is, a topologically
trivial deformation of the central fiber of a Kulikov model $X\to (C,0)$.
For example, $X_{0,LR}(\ell)$ is a Kulikov surface.

\begin{definition}\label{def:dss} We say $X_0$ is {\it $d$-semistable} if
$\mathcal{E}xt^1(\Omega^1_{X_0},\mathcal{O}_{X_0})=\cO_{(X_0)_{\rm sing}}.$
\end{definition}

By \cite{friedman1983global-smoothings}, $X_0$ is the central fiber of a Kulikov model
if and only if it is $d$-semistable. We recall some basic statements about $d$-semistable Kulikov
surfaces from \cite{friedman1986type-III, laza2008triangulations,
  gross2015moduli-of-surfaces}.  Let $X_0$ be a Type III Kulikov surface
with irreducible components $V_i$ and double curves
$D_{ij}=V_i\cap V_j$.
One defines the lattice of ``numerical Cartier divisors''
\begin{displaymath}
  L = \ker\big( \oplus_i \Pic V_i \to \oplus_{i<j} \Pic D_{ij} )
\end{displaymath}
with the homomorphism given by restricting line bundles and applying
$\pm1$ signs.
The map is surjective over $\bQ$ by \cite[Prop. 7.2]{friedman1986type-III}.
The set of
isomorphism classes of not necessarily $d$-semistable Type III
Kulikov surfaces of the combinatorial type $X_0$ is isogenous to
$\Hom(L, \bC^*)$.

The {\it period point} \cite[Sec.~3]{friedman1986type-III}
associated to $X_0$ is an element $\psi \in \Hom(L,\bC^*)$.
It inputs a collection of line bundles $L_i\in \Pic V_i$ whose
degrees agree on double curves $L_i\cdot D_{ij} = L_j\cdot D_{ji}$ and measures
an obstruction in $\C^*$ to their gluing together to form a line bundle on $X_0$.
In particular, the Picard group of the surface $X_0$ is $\ker(\psi)$. The surface is $d$-semistable
iff the following divisors are Cartier:
$\xi_i = \sum_j D_{ij}-D_{ji} \in L$. Note that $\sum_i
\xi_i=0$. Thus, the $d$-semistable surfaces correspond to the points
of multiplicative group $\Hom(\bL,\bC^*)$, where 
\begin{displaymath}
  \Xi = \frac{ \oplus_i \bZ\xi_i }{ (\sum_i \xi_i)},
  \qquad
  \bL = \coker(\Xi\to L).
\end{displaymath}
There is a symmetric bilinear form on
$L$ defined by $(R_i)^2:=\sum R_i^2$ which descends to $\mathbb{L}$
because $\Xi$ is null (in fact it generates
 the null space over $\Q$). Define
 $\overline{\mathbb{L}}:=\mathbb{L}/(\rm tors)$.

\begin{definition} Call a surface $X_0$ with $\psi=1\in \Hom(L,\C^*)$
a {\it standard surface}. \end{definition}

\begin{proposition}\label{prop:dss-fiber} Let $X_{0,LR}(\ell)$ be an elliptically fibered divisor model
as in Definition \ref{divisordef}. The classes of the fibers of the fibration
$$X_{0,LR}(\ell)\to \mathbb{P}^1\cup \dots \cup \mathbb{P}^1$$ reduce to the same
class in $\mathbb{L}$.  \end{proposition}

\begin{proof} Let $f_i$ be a fiber of the fibration over a
non-nodal point on the $i$th $\mathbb{P}^1$.
Define $\sigma_i :=\sum_{j\in S_i } \xi_j $ where $S_i$ denotes
the set of components which fiber over a $\mathbb{P}^1$
with index less than $i$. Then $[f_i]-[f_1] = \sigma_i$. Hence $[f_i]$ and $[f_1]$
define the same class in $\mathbb{L}$ for all $i$, which we denote by $f$. \end{proof}

\begin{lemma}\label{dlemma} A standard surface $X_{0,LR}(\ell)$ is elliptically fibered. \end{lemma} 

\begin{proof} Consider a vertical chain of rational curves as in
Definition \ref{def:sing-fibered} on $X_{0,LR}(\ell)$,
which is not, a priori, elliptically fibered. This vertical chain defines a class $f_i\in L$
and it is easy to check that $\psi(f_i)$ is the element of $\C^*$ which makes the two ends
of the chain match on the appropriate double curve. Since $\psi(f_i)=1$,
the chain $f_i$ closes into a cycle.
Since the standard surface is $d$-semistable,
Proposition \ref{prop:dss-fiber} implies all vertical strips of $X_{0,LR}(\ell)$ are fibered.

Similarly, there is a unique way to successively glue the components of the section $s$
into a chain from left to right, except possibly that the section at the right end doesn't match up.
The mismatch is an element of $\C^*$ equal to $\psi(s)$. Hence $s$ glues to a section
on the standard surface. \end{proof} 

\begin{proposition}\label{dssinput} The moduli space of $d$-semistable elliptically fibered
surfaces $X_{0,LR}(\ell)$ is isogenous to the torus
$\Hom(\overline{\mathbb{L}}/\Z f\oplus \Z s,\C^*)\cong (\C^*)^{17}$. In particular,
all deformations which keep $f$ and $s$ Cartier are elliptically fibered. \end{proposition} 

\begin{proof} By Proposition \ref{dlemma}, a $d$-semistable elliptically fibered surface exists.
Given one, the $d$-semistable 
topologically trivial deformations
are locally parameterized by the $19$-dimensional torus
$\Hom(\overline{\mathbb{L}},\C^*).$
Those that keep $s$ and $f$ Cartier are thus identified with the $17$-dimensional
subtorus for which $\psi(f)=\psi(s)=1$. Starting with the elliptically fibered
standard surface $X_{0,LR}(\ell)$, the arguments in Lemma \ref{dlemma} imply that
keeping $s$ and $f$ Cartier preserves the condition of being elliptically fibered. The
converse is also true, so the proposition follows.
\end{proof}

The space of $d$-semistable deformations of $X_{0,LR}(\ell)$ 
which keep $f$ and $s$ Cartier is $18$-dimensional and smooth and the
$17$-dimensional subspace of topologically trivial deformations
is a smooth divisor. 

\begin{definition}\label{lambdalatt}
Let $X_0$ be any Kulikov model. Define for any component $V_i$
the lattice $\widetilde{\Lambda}_i:=\{D_{ij}\}^\perp\subset H^2(V_i,\Z)$. Then
there is an inclusion $\iota_i\colon \widetilde{\Lambda}_i\hookrightarrow L$
sending $\lambda\in \widetilde{\Lambda}_i $ to
the numerically Cartier divisor
which is $\lambda$ on $V_i$ and $0$ on all other components.
Now suppose that $X_0 = X_{0,LR}(\ell)$ is elliptically fibered.
Define $\Lambda_i$ to be the image of $\widetilde{\Lambda}_i$ in
$\overline{\mathbb{L}}/\Z f\oplus \Z s$ and let $\Lambda := \oplus \Lambda_i$.
\end{definition}

Concretely, $\widetilde{\Lambda}_i$ is zero unless $Q(V_i)> 0$ and it maps
isomorphically to $\Lambda_i$ unless $V_i$ is an $X$-type end surface, in which
case the map to $\Lambda_i$ quotients by $\Z f$.

\begin{remark}\label{surjperiods}
By Proposition \ref{dssinput}, it is possible to realize any homomorphism
$\Hom(\Lambda,\C^*)$ as the restriction of the period map $\psi$
of some $d$-semistable elliptically fibered surface. Following 
\cite{gross2015moduli-of-surfaces}, \cite{friedman2015on-the-geometry}
the {\it period point} of the anticanonical pair $(V_i, \sum_j D_{ij})$ is the 
restriction homomorphism
$$\psi_i\colon \widetilde{\Lambda}_i \to {\rm Pic}^0(\textstyle\sum_j D_{ij})\cong \C^*$$
and this period map is compatible with the inclusion of $\widetilde{\Lambda}_i$
into $L$ in the sense that $\psi\circ \iota_i = \psi_i$. Thus, any period
point of any component $V_i$ can be realized by
some $d$-semistable elliptically fibered surface, {\it except} for the case
when $V_i$ is an $X$-type end, where the extra condition $\psi_i(f)=1$ ensures
either of the equivalent conditions that (1) $\psi_i$ descends to $\Lambda_i$
or (2) $V_i$ is elliptically fibered in class $f$.

\end{remark}

\subsection{Limits of elliptic fibrations}

We prove in this section that $X_{0,LR}(\ell)$ is a limit of elliptically fibered
K3 surfaces and that the very singular fibers (cf. Definition \ref{def:very-sing}) are the limits
of the correct number of singular fibers. 

\begin{proposition}\label{limit} Let $X_{LR}(\ell)\rightarrow C$
be a smoothing of an
elliptically fibered $X_{0,LR}(\ell)$ which keeps $f$
and $s$ Cartier. Then the general fiber is an elliptic K3 surface,
the very singular fibers are the limits of the singular fibers,
and the section $s$ is the limit of the section. \end{proposition}

\begin{proof} Let $f$ be some fiber.
Since we keep $s$ and $f$ Cartier, there are line bundles $L_s$ and $L_f$ on
$X_{LR}(\ell)$ which when restricted to the central
fiber are $\mathcal{O}(s)$ and $\mathcal{O}(f)$ respectively.
By constancy of the Euler characteristic, $\chi(\mathcal{O}(s))=1$ and
$\chi(\mathcal{O}(f))= 2$. 
Since 
$h^0(\cO(s))=1$, $h^0(\cO(f))=2$ and
$h^0(\cO(-s))=h^0(\cO(-f))=0$
on every fiber,
it follows from Serre duality that
$h^1(\mathcal{O}(s))=h^1(\mathcal{O}(f))=0$ on every fiber.
By Cohomology and Base Change
\cite[III.12.11]{hartshorne1977algebraic-geometry} we conclude that 
$H^0(L_s)$ and $H^0(L_f)$ surject onto the corresponding
spaces of sections on the central fiber. Thus, we can ensure that $s$ and
$f$ are flat limits of curves.
Note that for any choice of $f$, the line bundle $L_f$ is the
same on the general fiber, and so any $f$ is the limit of a section
from the same linear system.

A local analytic model of the smoothing shows that
any simple node of a fiber of
$X_{0,LR}(\ell)\rightarrow \mathbb{P}^1\cup\cdots\cup\mathbb{P}^1$
lying on a double curve gets smoothed. So any representative of
 $f$ which is not very singular is the limit of a smooth genus $1$
curve: Each node lies on the double locus.
Similarly, the nodes of $s$ are 
necessarily smoothed to give a smooth genus $0$ curve.
So the general fiber of $X_{LR}(\ell)$ is an elliptic K3 surface with
fiber and section classes $f$ and $s$.

Thus,
the only fibers which can be limits of singular fibers of the elliptic fibration
are the very singular fibers. If the ends are $L,R=1$, the generic choice
of $X_{0,LR}(\ell)$ has $24$ distinct very singular fibers with only one
node not lying on a double curve. Hence they must be limits of at worst $I_1$
Kodaira fibers on a smoothing. By counting,
each very singular fiber is the flat limit of an $I_1$ fiber.

It remains to show that the when $\ell_1>0$ for end type $L$ or $R=2,3$
the two non-reduced vertical components of $R$ are each limits of two singular fibers.
This again follows from counting, 
along with a monodromy argument which shows these two components
of $R$ must be limits of an equal number of singular fibers.

Finally when $X_{0,LR}(\ell)$ is not generically chosen,
is it a limit of such. This allows us to determine the
multiplicities in all cases. \end{proof}

\begin{remark} A consequence of Proposition \ref{limit} is that
on any degeneration of elliptic K3 surfaces, the limit of any individual fiber or the section
in the divisor or stable model is Cartier
(though a priori, only the limit of $s+m\sum f_i$ need be Cartier).  \end{remark}
 
 \subsection{The monodromy theorem}
 \label{sec:monodromy-thm}
 
We begin with a well-known result on the monodromy of Kulikov/nef models:

\begin{theorem}[\cite{friedman1986type-III}]\label{thm:fs} Let $X\to C$ be a Type II or III degeneration
of $\mathbb{M}$-lattice polarized K3 surfaces.
Then the logarithm of monodromy 
on $H^2(X_t)$ of a simple loop enclosing $0\in C$ has the form 
$\gamma\mapsto (\gamma \cdot \delta)\lambda -(\gamma\cdot \lambda)\delta$
for $\delta$ isotropic, $\delta\cdot \lambda=0$, and
$\lambda^2=\#\{\textrm{triple points of }X_0\}.$ Furthermore
$\lambda, \delta\in \mathbb{M}^\perp$.
There is a homomorphism $\overline{\mathbb{L}}\rightarrow \{\delta,\lambda\}^\perp/\delta$
which is an isometry and respects $\mathbb{M}$. \end{theorem}

To compute the monodromy invariant $\lambda$ of the degeneration $X_{LR}(\ell)$
requires constructing an explicit basis of the lattice $\delta^\perp/\delta$, to coordinatize
the cohomology.

\begin{definition} Let $B$ be a generic $\ias$. A {\it visible surface} is a $1$-cycle valued in
the integral cotangent sheaf $T^*_\Z B$. Concretely, it is a collection of paths $\gamma_i$
with constant covector fields $\alpha_i$ along $\gamma_i$ such that at the boundaries of the
paths, the vectors $\alpha_i$ add to zero in $T^*_\Z B$. When the paths $\gamma_i$ are incident
to an $I_1$ singularity, the covectors $\alpha_i$ must sum to a covector vanishing on the
monodromy-invariant direction. Such a visible surface is notated $\gamma=\{(\gamma_i, \alpha_i)\}$.
 \end{definition}
 
 \begin{example}\label{ex:simple-vis} The simplest example of a
 visible surface is a path connecting two $I_1$ singularities
 with parallel monodromy-invariant lines (under parallel transport along the path). Another example
 is an integral-affine divisor $R_{\ia}$: It is the special case where the paths are straight lines $e_{ij}$
 and the cotangent vector field is $n_{ij}$ times the primitive integral
 covector vanishing along the corresponding edge. \end{example}
 
Following \cite{symington2003four-dimensions}, if $B$ is a generic $\ias$, there is a symplectic
four-manifold $(S,\omega)$ diffeomorphic to a K3 surface, together with
$\mu\colon (S,\omega)\rightarrow B$
a Lagrangian torus fibration over $B$ that has $24$ singular fibers
over the $I_1$ singularities. From a visible surface
$\gamma$ one can build from cylinders a surface $\Sigma_\gamma\subset S$ 
fibering over $\gamma$. Its fundamental class is well-defined in $F^\perp/F$,
where $F=[\mu^{-1}({\rm pt})]$ is the Lagrangian fiber class. Its symplectic area can be computed as
$$[\omega]\cdot [\Sigma_\gamma]=\sum_i \int_{\gamma_i} \alpha_i(\gamma_i'(t))\,dt$$
and so in particular, for any integral-affine divisors $R_{\ia}$ we have $[\omega] \cdot [\Sigma_{R_{\ia}}]=0.$
 Furthermore, the symmetric bilinear form
 $$\gamma\cdot \nu=\{(\gamma_i,\alpha_i)\}\cdot \{(\nu_j,\beta_j)\}
 := \sum_{p\in \gamma\,\cap\, \nu } (\gamma_i\cdot \nu_j)_p \det(\alpha_i,\beta_j)_p$$
 agrees with the intersection number $[\Sigma_\gamma]\cdot [\Sigma_\nu]$ in $F^\perp/F$.
 The relevance of the symplectic geometry
lies in the following theorem:

 \begin{theorem}[Monodromy Theorem]\cite[Prop.3.14]{engel2021smoothings}, 
\cite[Thm.8.38]{alexeev2019stable-pair}
\label{thm:mono}  Suppose that $B=\Gamma(X_0)$ is
 generic and the dual complex of a Type III Kulikov model. There is a symplectic K3
 manifold $S$ with a Lagrangian torus fibration over $B$, and
 a diffeomorphism $\phi\,:S\rightarrow X_t$ to a nearby smooth fiber such that
 \begin{enumerate}
 \item $\phi_*[F] = \delta$
 \item $\phi_*[\omega] = \lambda$
 \end{enumerate}
Furthermore, if $R$ is an integral-affine divisor, then $R$ determines both
an element $[R]\in \mathbb{L}$ and a visible surface $\Sigma_R\subset S$. The image of
$[R]$ under the map $\overline{\mathbb{L}}\rightarrow \{\delta,\lambda\}^\perp/\delta$
from Theorem \ref{thm:fs}
is the same as $\phi_*[\Sigma_R]$.
 \end{theorem}
  
 By choosing a collection of visible surfaces $\gamma$, we may produce
 coordinates on the lattice $\delta^\perp/\delta$ which allow us to determine
 how the classes $\lambda$ sit relative to various classes. But, to employ
 this technique for general $X_0$
 we must first factor all singularities with charge $Q>1$ into $I_1$
 singularities, and only then apply the Monodromy Theorem.
 We describe this process when all $\ell_i>0$ but the general case
 follows from a limit argument.
 
Consider $B_{LR}(\ell)$. Let $f_{\ia}$ and $s_{\ia}$ be the
 integral-affine divisors corresponding
 to the fiber $f$ and section $s$ of $X_{0,LR}(\ell)$, respectively. We have described
 in Table \ref{tab:sings-on-ias} toric models for the $Q=2$ and $Q=3$ singularities.
 We may flop all the exceptional $(-1)$-curves in these toric models in the
 smooth threefold $X_{LR}(\ell)$.
 This has the effect of blowing down these $(-1)$-curves
 and blowing up the intersection point with the double curve
 on the adjacent component. 
 In particular, the left and
 right ends of the section $s$ are $(-1)$-curves which get flopped.

By first making a base change of $X_{LR}(\ell)\rightarrow C$ and resolving to a
new Kulikov model, we may ensure that the $(-1)$-curves get flopped onto toric components.
 This gives a new Kulikov model $X_{0,LR}'(\ell)$ with $24$ distinct $I_1$ singularities.
  The effect of these modifications on the dual complex is to first refine
  the triangulation (the base change), then factor each singularity
 into $I_1$ singularities, moving each one one unit of lattice length in its
 monodromy-invariant direction (the flops). These {\it $I_1$-factorization directions}
 are listed for the various end singularities in Table \ref{tab:sings-on-ias}.
 
 \begin{definition} We define $19$ visible surfaces $\gamma_i\in \{s_{\ia},f_{\ia}\}^\perp$
 in the dual complex $\Gamma(X_{0,LR}'(\ell))$
 as follows: If $\ell_i\vec{v}_i$ connects two $I_1$ singularities,
 then $\gamma_i$ is the path along the vector $\ell_i\vec{v}_i$
 connecting them as in Example \ref{ex:simple-vis}. For $i=1,2,3$ and all end behaviors,
 the visible surfaces $\gamma_i$ are uniquely defined by the following properties:
 \begin{enumerate}
 \item $\gamma_i$ is supported on the edge $\ell_i\vec{v}_i$ and the segments along
 which the $I_1$-factorization occurs of the singularities at the two ends of $\ell_i\vec{v}_i$.
\item The support of $\gamma_1$ does not contain the $I_1$-factorization
 direction corresponding to the section $s$.
 \item $\gamma_i$ is integral, primitive, and
 $[\omega]\cdot \Sigma_{\gamma_i}$ is a positive integer multiple of $\ell_i$.
 \end{enumerate}

 \end{definition}
 
\begin{example} The visible surface $\gamma_1$ has
 weights $-1,0,1$ along
 the $I_1$ factorization directions $(1,-3),(1,0),(1,3)$ respectively
of $X_3$ and is balanced by a unique choice of covector along the edge
$\ell_1\vec{v}_1$. Here the ``weight" is the multiplicity of the primitive covector
vanishing on the monodromy-invariant direction of the $I_1$ singularity
at the end of the segment. The covector that $\ell_1\vec{v}_1$ carries ends up
being three times the primitive covector vanishing on the monodromy-invariant
direction at the endpoint
of $\ell_1\vec{v}_1$. \end{example}

As we are henceforth concerned only with intersection numbers,
we lighten the notation by simply writing $\gamma$ for $\phi_*[\Sigma_\gamma]$.
 
 \begin{proposition}\label{vis-int} The classes $\lambda=\phi_*[\omega]$ and $\gamma_i$
 lie in $\{s,f\}^\perp$ and their
 intersection matrices for the three end behaviors are:
 
 \begin{table}[H]
  \begin{tabular}[htp]{c|ccc}
   $L=1$   &$\gamma_1$ &$\gamma_2$ & $\gamma_3$ \\
    \hline
  $\lambda$ 	 & $3\ell_1$ & $\ell_2$ & $\ell_3$ \\
  $\gamma_1$   &  $-8$ & $3$ & $0$ \\
  $\gamma_2$   & $3$ & $-2$ & $1$  \\
  $\gamma_3$   & $0$ & $1$ & $-2$ \\
  \end{tabular}\hspace{5pt}
   \begin{tabular}[htp]{c|ccc}
     $L=2$   &$\gamma_1$ &$\gamma_2$ & $\gamma_3$ \\
    \hline
  $\lambda$ 	& $2\ell_1$ & $2\ell_2$ & $\ell_3$ \\
  $\gamma_1$   & $-8$ & $2$ & $0$ \\
  $\gamma_2$   & $2$ & $-4$ & $2$ \\
  $\gamma_3$   & $0$ & $2$ & $-2$ \\
  \end{tabular}
  \hspace{5pt}
     \begin{tabular}[htp]{c|ccc}
     $L=3$   &$\gamma_1$ &$\gamma_2$ & $\gamma_3$ \\
    \hline
  $\lambda$ & $\ell_1$ & $2\ell_2$ & $\ell_3$ \\
  $\gamma_1$   & $-2$ & $1$ & $0$  \\
  $\gamma_2$   & $1$ & $-4$ & $2$ \\
  $\gamma_3$   & $0$ & $2$ & $-2$ \\
  \end{tabular}
  \end{table}
\noindent We also have $\gamma_i\cdot \gamma_{i-1}=1$, $\gamma_i^2=-2$,
  $\lambda\cdot \gamma_i = \ell_i$ for $i\geq 4$ until the right end.
  \end{proposition} 
  
  \begin{proof} Because the weight of the visible surface $\gamma_1$
  along the edge corresponding to $s_{\ia}$ is always zero, so
  we have $\Sigma_{\gamma_1} \cdot \Sigma_{s_{\ia}}=0$. The other $\gamma_i$ 
  are also disjoint from $s_{\ia}$. Furthermore, all $\gamma_i$ are disjoint
  from some fiber $f_{\ia}$ and hence $\Sigma_{\gamma_1} \cdot \Sigma_{f_{\ia}}=0$.
  Because $s_{\ia}$ and $f_{\ia}$ are integral-affine divisors, we have
  $[\omega] \cdot \Sigma_{f_{\ia}}  = [\omega]\cdot \Sigma_{s_{\ia}}  =0$.
  More generally, the formula $\int_{\Sigma_\gamma} \omega =\sum \int \alpha_i(\gamma_i'(t))\,dt$
  allows us to compute $[\omega]\cdot \Sigma_{\gamma_i}$ for all $i$.
  The other intersection numbers $\Sigma_\gamma \cdot \Sigma_\nu$ 
  can be computed via the defined intersection form $\gamma\cdot \nu$ on visible surfaces.
  Applying $\phi_*$ to the aforementioned classes
  preserves their intersection numbers, giving the tables above.
    \end{proof}

\begin{corollary}\label{identifygamma}
  After an isometry in $\Gamma$,
  the classes $\gamma_i \in \{s,f\}^\perp$ are:
\begin{align*}
&L=1 && \gamma_1 = -\beta_L,&& && && \gamma_i=\alpha_i \textrm{ for }i\geq 2 \\
& L=2 && \gamma_1 = \beta_L, && \gamma_2 = -\gamma_L, && && \gamma_i=\alpha_i \textrm{ for } i\geq 3 \\
& L=3 && \gamma_1 = \alpha_2, && \gamma_2 = \gamma_L, && \gamma_3=\alpha_1, && \gamma_i=\alpha_i \textrm{ for } i\geq 4.
\end{align*}
\end{corollary}
 
 \begin{proof} This follows directly from Proposition \ref{vis-int}. When $L=1,3$ the
 $\gamma_i$ span a lattice isomorphic to $I\!I_{1,17}$ and hence their intersection matrix
 determines them uniquely up to isometry in $\Gamma$. When $L=2$,
 the lattice spanned by $\gamma_i$ is imprimitive
 but after adding the {\it integral} visible surface
 $\frac{1}{2}(\gamma_1+\gamma_2)$ it becomes all of $I\!I_{1,17}$ and the same logic
 applies. Note $\frac{1}{2}(\beta_L-\gamma_L)$ is also integral.\end{proof}

\begin{corollary}\label{final-monodromy} The monodromy invariant of $X_{LR}(\ell)$ is the unique
  lattice point $\lambda\in \sigma_{LR}$ whose coordinates
  $a_i=\alpha_i\cdot \lambda$, $b_L=\beta_L\cdot \lambda$, $c_L=\gamma\cdot \lambda$,
  $b_R=\beta_R\cdot \lambda$, $c_R=\gamma_R\cdot \lambda$ 
(cf. Section \ref{sec:rc-fan}) take the values

\begin{centering}

  \begin{tabular}[htp]{c | c c c | c c c | c c c | c }
    $L$ End &$\ell_1$ &$\ell_2$ &$\ell_3$ & $\cdots$ & $\ell_i$ & $\cdots$ & $\ell_{17}$ & $\ell_{18}$ & $\ell_{19}$ & $R$ End  \\
    \hline
  $1$  &$-b_L/3$ &$a_2$ &$a_3$ & & & & $a_{17}$ & $a_{18}$ & $-b_R/3$ & $1$ \\
  $2$  &$b_L/2$ &$-c_L/2$ &$a_3$ & $\cdots$ & $a_i$ & $\cdots$ & $a_{17}$ & $-c_R/2$ & $b_R/2$ & $2$ \\
  $3$ &$a_2$ &$c_L/2$ & $a_1$ & & & & $a_{19}$ & $c_R/2$ & $a_{18}$ & $3$
    \end{tabular}
  
  \end{centering}
    
\end{corollary}

 \begin{proof} The monodromy invariant $\lambda = \phi_*[\omega]$ is uniquely determined
 by the tabulated values of $\lambda\cdot \gamma_i$ in Proposition \ref{vis-int}. 
 The result follows from Corollary \ref{identifygamma}.
 \end{proof}
 
 \begin{definition}\label{monodivisor} Let $X(\lambda)\to (C,0)$ be a divisor model
   of a degeneration of elliptic K3 surfaces whose monodromy invariant
 is $\lambda\in \sigma_{LR}$. That is, $X(\lambda)=X_{LR}(\ell)$ for an appropriate choice of $\ell$ by
 Proposition \ref{limit}.
From Corollary \ref{final-monodromy} such a model exists
whenever \begin{align*} b_L(\lambda)& \equiv b_R(\lambda)\equiv 0\pmod 6, \\
c_L(\lambda)&\equiv c_R(\lambda)\equiv 0\pmod 2.\end{align*}
Let $X_0(\lambda)$ be the central fiber
and $B(\lambda):=\Gamma(X_0(\lambda))$
be the dual complex.
 \end{definition}
  
  \begin{remark} The divisor model $X(\lambda)$ is not combinatorially unique---various
  choices were made in its construction, such as how to triangulate $B(\lambda)$. But
  these choices play no role, since the function of $X(\lambda)$ in the paper is to
  apply Theorem \ref{threeinputs}. It verifies
  input {\bf (div)} and serves an example on which input {\bf (d-ss)} can be checked.
  \end{remark}

\subsection{Type II models}
\label{sec:type-ii}

We now describe Type II divisor models. These correspond to when
the $\ias$ on the dual complex degenerates to a segment. It can do
so in two ways.

If $\{L,R\}\in \{2,3\}$ and
$\ell_2=\cdots = \ell_{18}=0$, the sphere degenerates to a vertical segment.
Define a Type II Kulikov model, of combinatorial type $\wY_4\wY_{20}$, 
associated to the Type II ray $\wD_{16}$ of $\cF_\rc$
as follows:

It is a vertical chain of surfaces. The bottom ($\wY_4$)
of the chain is $\mathbb{F}_2$.
It is glued to the next component up along a genus $1$ curve in the
anticanonical class $2(s+2f)$ with $s$ the $(-2)$-section.
Next, a sequence of elliptic ruled surfaces glued successively along
elliptic sections of the ruling, of self-intersections $-8$ and $8$. 
At the top of the chain ($\wY_{20}$)
is the blow-up $Bl_{16}\mathbb{F}_2$ at $16$ points on a
genus $1$ curve in the class $2(s+2f)$, glued along the strict transform of the
curve.

We now give the structure of a divisor model.
On the top of the chain, the divisor $R$
is the sum of the $16$ reducible fibers of the ruling and
four doubled fibers tangential to the double curve. On the bottom
it is four doubled fibers tangential to the double curve, plus $16$ fibers
of the ruling, plus the $(-2)$-section. On the intermediate surfaces,
it is the sum of $16$ fibers and $4$ double fibers.
The union of fibers of the rulings on all components 
form the very singular fibers. 

If $\{L,R\}\in\{1,2,3\}$ and $\ell_1=\cdots = \ell_9=\ell_{11}=\cdots \ell_{19}=0$,
the sphere degenerates to a
horizontal segment. Define a Type II Kulikov model, of combinatorial type
$\wX_{12}\wX_{12}$, associated to the Type II ray $\wE_8\wE_8$ of $\cF_\rc$
as follows:

The left end ($\wX_{12}$) is a rational elliptic surface. It is glued along a smooth elliptic fiber
to a chain of surfaces isomorphic to $E\times \mathbb{P}^1$ until the right end ($\wX_{12}$)
is reached, which also rational elliptic. The divisor model is defined as follows:
The section is an exceptional curve at each end, and a section $\{e\}\times \mathbb{P}^1$ on the
intermediate components. The very singular fibers are the singular fibers of the 
elliptic fibrations of the left and right ends.

\subsection{Stable models and their irreducible components}
\label{sec:irr-comps}

It remains to describe the stable model resulting from the divisor model $X(\lambda)$.
We describe here the components which will appear in the stable model,
and prove that in Type III their moduli spaces are affine. 

\begin{definition}\label{stabletype} 
  The {\it stable type} of a cone in $\cF_\rc$ is gotten by the
  following transformations on ADE type:
  Bold the symbols $A_n$, $E_n$, $E'_1$, replace $D_n$ by ${\bf C_n}$ and
  $D_0,D'_0$ by ${\bf C_0}$. Thus the stable type only fails to distinguish
  between $D_0$ and $D'_0$; both of them have the stable type ${\bf C_0}$. 
   The conversions between the three notations of the paper
  are summarized in Table \ref{tab:notation-conversions}.
  \end{definition}

\begin{table}

  \begin{tabular}[htp]{ c | c |  c  }
Cones of $\cF_\rc$ & Singularities of $\ias$ & Stable types   \\
    \hline
$E_k\,(0\leq k\leq 8)$ and $E_1'$ & $X_{k+3}$ and $X_4'$ & ${\bf E_k}$ and ${\bf E_1'}$  \\
$D_k\,(0\leq k\leq 17)$ and $D_0'$  & $Y_2Y_{2+k}$ and $Y_2Y_2'$ & ${\bf C_k}$ and ${\bf C_0}$ \\
$A_k\,(0\leq k\leq 17)$  & $I_{k+1}$ & ${\bf A_k}$ \\
$\wE_8\wE_8$ and $\wD_{16}$  & $\wX_{12}\wX_{12}$ and $\wY_{4}\wY_{20}$ & ${\bf \wE_8\wE_8}$ and ${\bf \wC_{16}}$
    \end{tabular}
    \bigskip
    \caption{Conversions between the three notations}
    \label{tab:notation-conversions}
\end{table} 



\begin{definition}\label{def:stab}
For each possible symbol in the stable type, we define an
irreducible stable pair $(X, \Delta + \epsilon R)$ as follows:

\smallskip
${\bf E_n} \,(n\geq 0)$, ${\bf E_1'}$:
$X$ is the contraction of an elliptically fibered rational surface with 
an $I_{9-n}$ fiber along all components of fibers not meeting a section $s$.
In particular an $A_{8-n}$ is contracted in the $I_{9-n}$ fiber to give the nodal
curve $\Delta$. The divisor $R$ is $s$ plus the images of the
singular fibers not equal to $\Delta$. There is an induced lattice embedding
$A_{8-n} \subset E_8$. For $k=1$, the inclusion $A_7\subset E_8$ can be
either primitive (for the surface ${\bf E_1}$) or imprimitive (for the
surface ${\bf E_1'}$).

\smallskip
${\bf A_n}\,(n\ge 0)$: Let $(X^\nu,\Delta^\nu)$ be the toric anticanonical pair
$(\mathbb{F}_0,s_1+f_1+s_2+f_2)$. Then $X$ is the result of
gluing along the two sections $s_1$ and $s_2$ via the fibration $|f_1|$. The
boundary $\Delta$ is the sum of two glued fibers $f_1$ and $f_2$ and
$R$ is another section $s$ plus $n+1$ other nodal fibers. 

Restricting to either $s_1$ or $s_2$ gives a weighted stable curve
$(\bP^1, q_1+q_2 +\epsilon \sum p_i)$ with two boundary points
$q_1,q_2$ and $n+1$ other points $p_i$. Stable degenerations of ${\bf A_n}$
surfaces are in a bijection with degenerations of such curve
pairs and are describe by the well known Losev-Manin space $\oL_{n+1}$
\cite{losev2000new-moduli-spaces}. Thus, the moduli of ${\bf A_n}$ surfaces
are reduced to the moduli of curves. 

\smallskip
${\bf C_n}\,(n\ge0)$: Let $X^\nu=(\mathbb{F}_1, \Delta_1+\Delta_2)$ be an
anticanonical pair with $\Delta_1^2=0$ and $\Delta_2^2=4$. Then $X$ is the result of
gluing $X^\nu$ along the bisection $\Delta_2$ by the involution
switching the intersection points with the fibration of class $\Delta_1$. Here $\Delta$ is
$\Delta_1$ plus the gluing of $\Delta_2$ and $R$ is the $(-1)$-section $s$, plus
the sum of $k$ nodal glued fibers not equal to $\Delta$, plus twice the fibers
tangent to $\Delta_2$. These fibers become cuspidal upon gluing.

Restricting to $\Delta_2$ gives a weighted stable
curve $(\bP^1, q^+ + q^- + \epsilon \sum_{i=1}^k (p_i^+ + p_i^-))$
together with an involution $\iota$ exchanging $q^\pm$ and
$p_i^\pm$. The stable degenerations of such curve pairs are the
$C_n$ curves of Batyrev-Blume \cite{batyrev2011on-generalizations}, and
they are in a bijection with stable degenerations of our surface
pairs. Thus, as in the ${\bf A_n}$ case, the moduli of ${\bf C_n}$ surfaces are
reduced to the moduli of curves.
One should compare this to Lemma~\ref{lem:C-fan}.

\smallskip
${\bf \widetilde{C}_{16}}$: The Hirzebruch surface $\mathbb{F}_2$ glued to itself
along a smooth genus $1$ bisection of the ruling, in class $2(s+2f)$. The divisor
is the section, plus double the fibers tangent to the bisection
which get glued to cuspidal curves, plus $16$ nodal fibers. There is
no boundary.

\smallskip
${\bf \widetilde{E}_8}$: A rational elliptic surface contracted along
 all components of fibers not meeting a section $s$. The boundary $\Delta$ is an
$I_0$ fiber, i.e. a smooth elliptic curve and the divisor $R$ is
$s$ plus the sum of the singular fibers. 

\smallskip
Given a stable type ${\bf S_1}\cdots {\bf S_N}$
we define a stable surface as follows: For each symbol
${\bf S_i}$ take the corresponding irreducible stable pair listed
above, and glue the ${\bf S_i}$ together along $\Delta$ such that
the sections $s$ glue.
\end{definition}

\begin{remark}
  The maximal number of irreducible components of a stable pair
  is $20$, achieved for the $\cF_\rc$ cone $E_0A_0^{18}E_0$ or
  $\ias$ combinatorial type
  $X_3I_1^{18}X_3$, whose stable type is
  ${\bf E_0A_0^{18}E_0}$. 
\end{remark}

\begin{warning}
  All of the above stable pairs are Weierstrass fibrations, normal or
  non-normal. Thus, they have an elliptic involution $\iota$, and their moduli
  can be analyzed from the perspective of their $\iota$ quotients, in
  a manner similar to \cite{alexeev17ade-surfaces}.
  But the ${\bf ACE}$ surfaces defined above for the rc
  polarizing divisor are \emph{different} from the $ADE$ surfaces of
  \cite{alexeev17ade-surfaces}; the latter are adapted to the ramification
  polarizing divisor.
\end{warning} 

Recall the definitions of the root lattices $C_n$ ($n\ge1$) and $E_n$
($n=6,7,8$).  The $C_n$ lattice is the same as the $D_n$ lattice: the
sublattice of $\bZ^n(-1)$ of vectors with even sum of coordinates. The
Weyl group $W(C_n)$ is the group $\bZ_2^n\rtimes S_n$ of signed permutations,
and $W(D_n)$ is the index
2 subgroup $\bZ_2^{n-1}\rtimes S_n$ of signed permutations
with an even number of sign changes.

$E_n$ is the lattice $K_V^\perp\subset\Pic V$
for a smooth del Pezzo surface $V$ of Picard rank $\rho=n+1$.  Their
Weyl groups are defined to be generated by reflections in the
$(-2)$-vectors.
For some small $n$ these definitions give root lattices
$E_5=D_5$, $E_4=A_4$, $E_3=A_2A_1$. For lower
$n$ the definitions still make sense but may produce non-root lattices. In addition,
for $\rho=2$ there are two del Pezzo surfaces surfaces $\bF_1$ and
$\bF_0$, giving $E_1$ and $\iE_1$ respectively. We list the lattices
and their Weyl groups for these special cases in
Table~\ref{tab:small-DE}.

\begin{table}[htp!]
  \caption{Lattices $E_1,E'_1,E_2$ and their Weyl groups}
  \label{tab:small-DE}
  \centering
  \begin{tabular}{cccccccc}
    Symbol &Lattice &Group &&&Symbol &Lattice &Group\\
           $E_1$ &$\la -8\ra$ &$1$
               &&&$E_2$ &$\begin{pmatrix}-2&1\\1&-4\end{pmatrix}$ &$W(A_1)$\\
           $\iE_1$ &$\la -2\ra$ &$W(A_1)$\\

  \end{tabular}
\end{table}

\begin{definition}
  For a Dynkin type $A_n,C_n,E_n,\iE_1$, we denote by $\Lambda$ the
  corresponding root lattice, $T = \Hom(\Lambda, \bC^*)$ the torus with
  this character group, and by $W$ the Weyl group.
\end{definition}

\begin{theorem}\label{thm:moduli-slc-pair}
  The coarse moduli of stable pairs of type ${\bf A_n}$, ${\bf C_n}$, 
  ${\bf E_n}$, ${\bf E_1'}$ is $T/W$.
  The moduli space
  of stable pairs of a fixed stable type
  ${\bf S_1}\cdots {\bf S_N}$ is the product of the moduli
  spaces for the pairs of types ${\bf S_i}$, divided by the LR involution
  if the type is left-right symmetric.
\end{theorem}

\begin{proof} The easiest cases are ${\bf A_n}$ and ${\bf C_n}$ since they are
  reduced to the curve case.
  For ${\bf A_n}$ the
  moduli of such choices is simply a choice of $n+1$ fibers not equal
  to either component of $\Delta$, up to the $\C^*$-action on the
  base. This gives $\bC^*\backslash (\bC^*)^{n+1} / S_{n+1}= \Hom(A_n,\bC^*) /
  W(A_n)$. 

  For ${\bf C_n}$ surfaces the moduli is given by choosing
  $n$ fibers $y_1,\dotsc, y_n\in\bC$ not equal to the irreducible
  fiber $\Delta$ at $\infty$, modulo $S_n$ and the involution
  $(y_i)\to (-y_i)$. Using the maps $y_i = x_i + \frac1{x_i}$, this is
  the same as choosing
  $(x_1, \dotsc, x_n) \in \Hom(\bZ^n,\bC^*) / (\pm) = \Hom(C_n,
  \bC^*)$ modulo $S_n\ltimes\bZ_2^n = W(C_n)$.

  The minimal resolution of an ${\bf E_n}$ or ${\bf E_1'}$ surface is a
  rational elliptic surface $Y$ with a section $s$ and anticanonical
  $I_{9-n}$ fiber $D= D_1+\cdots + D_{9-n}$.  One has
  \begin{displaymath}
  E_n=\{D_1,\dots,D_{9-n}\}^\perp/f = \{s,D_1,\dots,D_{9-n}\}^\perp.
  \end{displaymath}
  Contracting $s$ then
  successively contracting all but one component of $D$, we see that
  $E_n\cong (K_V)^\perp$ on a del Pezzo surface $V$, so this is the
  same definition of $E_n$ as above.  
  The \emph{period torus} for the anticanonical pairs preserving the
  elliptic fibration is $\Hom(E_n, \bC^*)$. 
  Deformations of such pairs always preserve the $(-1)$-section
  $s$. 

  The {\it period point} $\varphi_Y\in \Hom(E_n,\,\C^*)$ is given by
  the restriction map on line bundles $E_n \to \Pic^0(D)= \C^*$.  In
  the current setting, the Torelli theorem for anticanonical pairs
  \cite[Thm.1.8]{gross2015moduli-of-surfaces},
  \cite[Thm.8.7]{friedman2015on-the-geometry} implies that two such
  surfaces $Y$ with marked section $s$ and fiber $D$ are isomorphic if
  and only if there is an element of the finite reflection group
  $W(E_n)$ relating their period points $\varphi_Y$. Thus the moduli
  space is $\Hom(E_n ,\C^*)/ W(E_n)$.

  For a stable surfaces of type ${\bf S_1\cdots S_N}$, the
  gluings of the components are unique up to an isomorphism, since the
  components form a chain. So the moduli space is the product of the
  moduli spaces for the irreducible components, modulo the LR
  involution if the type is symmetric.
\end{proof}
 
 \begin{corollary}\label{cor:affine}
   A type III stratum in $\slcrc$ of a fixed stable type 
   is affine.
 \end{corollary}

 \begin{remark}\label{DC-remark}
   For an $\ias$ with a $Y_2Y_{2+n}$ or $Y_2Y_{2+n}'$
   end as in Table~\ref{tab:sings-on-ias} and Notation~\ref{d-notation}, 
  there is an irreducible component~$V$ of a divisor model defining
   singularity $Y_{n+2}$ or $Y'_{2+n}$ of the integral-affine
   structure. For $Y_{n+2}$ (resp. $Y'_{2+n}$), one begins with an anticanonical pair
   $(\bF_1,\Delta_1+\Delta_2)$ (resp. $(\bF_0,\Delta_1+\Delta_2)$), 
   $\Delta_1^2=0$, $\Delta_2^2=4$ and blows up $n$ points on $\Delta_2$
   plus some corner blowups.  For $n=0$ these two deformation types 
   are distinct; but once $n>0$ they coincide (which is why we do not require
   the notation $Y_{2+n}'$ for $n>0$).

   For $n>0$,
   the orthogonal complement $\{\Delta_1, \Delta_2\}^\perp \subset \Pic V$ to the
   boundary is the $D_n$-lattice, and the group of admissible
   monodromies is $W(D_n)$, so the moduli space of
   anticanonical pairs is
   $T(D_n)/W(D_n)$. Exchanging a \emph{pair} of points $p_1,p_2\in
   \Delta_2$ which are blown up to their involution partners $\iota
   p_1,\iota p_2$ gives an anticanonical pair with the same period
   modulo $W(D_n)$
   but changing only one point $p$ to $\iota p$ changes the
   period and the isomorphism type of the surface.

   However, on the stable model this information is lost: exchanging
   any point $p$ to $\iota p$ gives the same fiber. So the moduli
   space of stable pairs is $T(C_n)/W(C_n)$, the quotient of
   $T(D_n)/W(D_n)$ by an involution. For $n=0$ there are two types of
   anticanonical pairs but only one type of stable pairs, so the map
   to the stable moduli is again $2:1$. 
 \end{remark}

 \begin{remark}
   Since the stable degenerations of the ${\bf A_n}$ and ${\bf C_n}$
   surfaces are compatible with the degenerations of the ${\bf A_n}$
   and ${\bf C_n}$ curves, their the moduli spaces come with
   compactifications of the form $\overline{T}/W$, where
   $\overline{T}$ is a toric variety for the Coxeter fan of type $A$,
   resp. $C$.  These are moduli spaces of Losev-Manin curves
   \cite{losev2000new-moduli-spaces} and Batyrev-Blume curves
   \cite{batyrev2011on-generalizations}.

   For the moduli of ${\bf E_n}$ surfaces, taking the star of the
   corresponding cone in $\cF_\rc$ and the closure of $T/W$ in
   $\oF^\rc$ provides a stable slc pair compactification $\oT/W$,
   where $\oT$ is a toric variety for the fan obtained from the $E_n$
   Coxeter fan by subdiving a Weyl chamber by the hyperplanes
   $\beta^\perp$, $\gamma^\perp$ as in Fig.~\ref{fig:cones}. This is a
   very interesting fan indeed which we investigate further 
   in a forthcoming paper.
 \end{remark}

 \begin{notation}\label{not:R-quotient}
   We now study the moduli stack of pairs of type ${\bf S}$. To do so,
 we introduce the following notations: Let $G$ be a discrete group acting properly
 discontinuously on an analytic space $X$. We notate the coarse space of the quotient by $X/G$
 and we notate the stack (orbifold) quotient by $[X:G]$.
 
 For our purposes, we require a more refined notion. Suppose we are
 given a reflection subgroup $W\subset G$ corresponding to a root
 system $R$ and for every root $\alpha\in R$ a divisor
 $\Delta(\alpha)\subset X$ contained in the fixed locus of the
 reflection $s_\alpha$. For any $x\in X$ consider the set of roots
 $R_x = \{\alpha \mid x\in\Delta(\alpha) \}$ and the subgroup
 $W(R_x)\subset G_x$ in the stabilizer of $x$ generated by the
 reflections $s_\alpha$ with $\alpha\in R_x$. Assume that $W(R_x)$ is
 a normal subgroup of $G_x$ for all $x$.  Under these conditions, we
 define $[X:_R G]$ as follows.

 For each point $x\in X$, there is an open $G_x$-invariant
 neighborhood $U_x\ni x$ where the $G_x$-action is approximated by the
 linear action of $G_x$ on the tangent space $\bT_x$, and 
 which satisfies following condition: for any
 $y\in U_x$ one has $G_y\subset G_x$, $R_y\subset R_x$, and
 $W(R_x)\cap G_y = W(R_y)$. This neighborhood is obtained by
 applying Luna's slice theorem and by 
 successfully removing the closed subsets where the above conditions
 fail. Now define
 \begin{displaymath}
   [U_x:_RG] =[U_x/W(R_x):G_x/W(R_x)].   
 \end{displaymath}
In words: we take the coarse quotient by $W(R_x)$ first, and then the
stack quotient by the remaining group $G_x/W(R_x)$.

At this point, we recall the theorem of Chevalley, Shephard and Todd
\cite{chevalley1955invariants-of-finite, shephard1954finite-unitary}:
if $G$ is a finite group acting linearly on a complex vector space
then $V/G$ is smooth iff $G$ is generated by pseudoreflections, i.e. linear
transformations fixing a hyperplane pointwise. For a Weyl group $W$
pseudoreflections are reflections $s_\alpha$. In particular, 
if $U_x$ is smooth then so is $U_x/W(R_x)$.
 
 For any $y\in U_x$ we can choose an open neighborhood $U_y\ni y$ in
 the same way as above.
 Since $W(R_x)\cap G_y = W(R_y)$, the
 map $U_y/W(R_y) \to U_x/W(R_x)$ is unramified. Thus, the map
 $[U_y:_RG] \to [U_x:_RG]$ is an open embedding, and the stacks
 $[U_x:_RG]$ patch together to define the stack $[X:_RG]$.

 The coarse moduli space of $[X:_R G]$ is $X/G$, same as for $[X:G]$, but the stack
 structure is different: the local reflection groups $W(R_x)$ are not
 part of the inertia.
 
 The reason for introducing this notation is that it concisely describes the types of moduli stacks
 which occur in the presence of a Torelli type theorem, for a more detailed discussion, see
 \cite[Rem.~2.36, Thm.~8.12]{alexeev2021compact}. For instance, the moduli stack of lattice-polarized ADE
 K3 surfaces is $[{D}:_RG]$ where ${D}$ is the period domain, $G$ is the appropriate
 arithmetic subgroup of $O(2,19)$, and the root system $R$ consists 
 of the vectors $\alpha$ with $\alpha^2=-2$.
 \end{notation}

 We also prove the following Lemma.  Using the notations of
 \cite{bourbaki2002lie-groups}, let $R$ be a root system with a root
 lattice $Q$, weight lattice $P$ and Weyl group $W$.
 Denote by $\mu_R$
 the finite abelian group $\Hom(P/Q, \bC^*)$. Then one has the
 following commutative diagram

 \begin{equation}
   \label{eq:PQ}
 \begin{tikzcd}
   \Hom(P,\bC^*) \ar[r] \ar[d, "/\mu_R"] & 
   \Hom(P,\bC^*)/W \ar[d, "/\mu_R"] \\
   \Hom(Q,\bC^*) \ar[r] &
   \Hom(Q,\bC^*)/W 
 \end{tikzcd}
 \end{equation}

 It is a basic result of the invariant theory of multiplicative type that
 one has $\Hom(P,\bC^*)/W = \bA^n$, with the coordinates on $\bA^n$ equal to
 the characters of the fundamental weights, see e.g. \cite[Ch.8, \S7,
 Thm.2]{bourbaki2005lie-groups}.
 
 For each root $\alpha\in Q$, let $\Delta_Q(\alpha)$ be the kernel of
 the homomorphism $\Hom(Q,\bC^*)\to\Hom(\bZ\alpha,\bC^*)=\bC^*$. Let
 $\Delta_P(\alpha)$ be its preimage in $\Hom(P,\bC^*)$. We use these
 divisors $\Delta_Q(\alpha)$, $\Delta_P(\alpha)$ to define the
 $:_R$ quotients as in Notation~\ref{not:R-quotient}. 
 
 \begin{lemma}\label{lem:R-quotient}
   One has $[\Hom(Q,\bC^*):_R W] = [\Hom(P,\bC^*)/W : \mu_R] = [\bA^n:\mu_R]$.
 \end{lemma}
 \begin{proof}
   The ramification divisor of $\Hom(P,\bC^*) \to \bA^n$ is
   $\cup_{\alpha\in R}\Delta_P(\alpha)$, see
   e.g. \cite[6.4, 6.8]{steinberg1965regular-elements}.
   An easy direct computation shows
   that the fixed locus of the reflection $s_\alpha$ on the weight lattice
   torus $\Hom(P,\bC^*)$ is
   $\Delta_P(\alpha)$. 

   For $x\in \Hom(P,\bC^*)$, let $\bT_x$ be the tangent space at $x$.
   Since the quotient $\Hom(P,\bC^*)/W=\bA^n$ is smooth, the quotient
   $\bT_x/W_x$ is smooth as well (this follows from a baby version of
   Luna's slice theorem).  By the theorem of Chevalley, Shephard and
   Todd,
   $\bT_x/W_x$ is smooth if and only if $W_x$ is generated by
   pseudo-reflections.
   Alternatively, we can cite \cite{bourbaki2002lie-groups}, Exercise
   7 to Ch.V \S5.
   Pseudoreflections in $W$ are reflections. So $W_x$ is
   generated by the reflections that it contains.  (We thank Michel
   Brion for this argument.)

   Now it follows that $[\Hom(P, \bC^*) :_R W] = \Hom(P, \bC^*)/ W =
   \bA^n$ and that 
   \begin{displaymath}
     [\Hom(Q, \bC^*) :_R W ] = 
     [\Hom(P, \bC^*) :_R (W \times \mu_R) ] =
     [\bA^n : \mu_R]. 
   \end{displaymath}
 \end{proof}

 \begin{remark}
   For the action of $W$ on the root lattice torus $\Hom(Q,\bC^*)$ it is in general not true that
   the stabilizer $W_x$ coincides with $W(R_x)$ for all $x$. For
   example, for 
   $R=A_2$ and $W=S_3$ there are two points with stabilizer
   $W_x=\bZ_3$ and trivial $W(R_x)$.
   Also, an explicit computation shows that the fixed locus of a reflection
   $s_\alpha$ on $\Hom(Q,\bC^*)$ is $\Delta_Q(\alpha)$ if and only if $\alpha$ is
   primitive in the weight lattice $P$. This holds for all irreducible
   root $ADE$ lattices except for $A_1$, in which case
   $\Delta_Q(\alpha)$ is a single point $\{1\}$ while the fixed
   locus of the involution $z\to z\inv$ is $\{\pm1\}$.
 \end{remark}

 For a simply laced $ADE$ root system one has $Q=\Hom(P,\bZ)$. For the
 $C_n$ root system one has $P(C_n)=\bZ^n$ and $Q(C_n)$ the sublattice
 of vectors with even sums, so that $P/Q=\bZ_2$. To simplify notation,
 we frequently denote the root lattice by the same symbol as the root
 system, and write $A_n$ etc. instead of $Q(A_n)$ etc.
 
 \begin{theorem}\label{thm:irr-comp-stacks}
   The moduli stack of irreducible pairs of type ${\bf S}$ is a
   $\mu_2$-gerbe over:
   \smallskip

   \begin{tabular}{lll}
     ${\bf S}$ & Stack & Group action\\ \hline
     ${\bf A_n}$, $n\ge0$ & $[\bA^n : \mu_{n+1}]$
                   & acting as $\mu_R$: $(c_i)\to (\xi^ic_i)$, $\xi^{n+1}=1$\\
     ${\bf C_n}$, $n\ge 0$ & $[\bA^n : \mu_2]$ & $(c_i) \to (-c_i)$\\
     ${\bf E_n}$, $n\ge3$ & $[\bA^n : \mu_{9-n}]$ & acting as $\mu_R=\Hom(P/Q,\bG_m)$\\
     ${\bf E_2}$ & $\bG_m \times \bA^1$\\
     ${\bf E_1}$ & $\bG_m$ \\
     ${\bf E'_1}$ & $[\bA^1 : \mu_4]$ & for $\mu_4=\la g\ra$, $g(c) = -c$\\
     ${\bf E_0}$ & $[\pt : \mu_3]$
   \end{tabular}

   \smallskip
   Here, for the ${\bf A_n}$ pairs we fix the left-right orientation.
 \end{theorem}
 \begin{proof}
   For ${\bf A_n}, {\bf C_n}$ and for ${\bf E_n}$ with $n\ge3$ the result is the direct
   application of Lemma~\ref{lem:R-quotient}. For smaller ${\bf E_n}$ we use
   the explicit normal forms of the surfaces.

   \smallskip   
   For ${\bf A_n}$, $\Hom(P,\bC^*)$ is the same as the choice of $n+1$
   points $p_i\in\bC^*$ with $\prod p_i=1$, with a choice of the
   origin, and $\Hom(P,\bC^*)/W=\bA^n$ is the set of coefficients
   $(c_i)$ in
   the equation
   \begin{math}
     \prod (x+p_i) = 1+c_1x + \dotsc + c_nx^n+x^{n+1}
   \end{math}
   which are well defined up to rescaling to $(c_i)\to (\xi^i c_i)$.

   \smallskip
   
   The data for the surface ${\bf C_n}$, $\Hom(P,\bC^*)=(\bC^*)^n$ is the
   data for the $n$ points $p_i^+$ on the bisection $B_2\setminus
   \{0,\infty\}$. These points are well defined up to switching
   $p_i^+\to p_i^-=\iota p_i^+$ and switching $p_i^+\to p_j^{\pm}$ for
   $i\ne j$. The quotient is $\bA^n$ with the coefficients $(c_i)$
   giving the equation $x^n+c_1x^{n-1} + \dotsb +c_n$ of the fibers on
   $\bP^1\setminus \{0,\infty\}/\iota = \bP^1\setminus\infty$. 

   Alternatively, $\Hom(Q,\bC^*)$ is the choice of $n$ points $p_i^+\in\bC^*$
   defined up to $p\mapsto -p$, and the moduli stack is
   $[\Hom(Q,\bC^*):_R W]$, giving the same result by
   Lemma~\ref{lem:R-quotient}. 

   \smallskip
   
   The normal forms for ${\bf E_6}$, ${\bf E_7}$, ${\bf E_8}$ were given in
   \cite{alexeev17ade-surfaces}. Here, we extend them to ${\bf E_n}$ with
   $0\le k\le 5$ and ${\bf E'_1}$. The quotient by the elliptic involution
   is $X/\mu_2=\bF_2$,
   the double cover is branched in the $(-2)$-section and a
   trisection.  After contracting the $(-2)$-section we get
   $\bP(1,1,2)$ and the equation of the trisection is a polynomial
   $f(x,y)$ of degree~6, where $\deg x=1$ and $\deg y=2$ so that
   $f(x,y)$ is a cubic in $y$. In affine
   coordinates $X$ has the equation $z^2+f(x,y)=0$. 

   For a Weierstrass surface $V\to \bP^1$, its minimal resolution
   $\wV$ has an $I_n$ Kodaira fiber with $n\ge1$ over $x_0\in\bP^1$
   iff the equation $f(x_0,y)$ has a double and a single roots in
   $y$, see e.g. \cite[IV.2.2]{miranda1989the-basic-theory}.  Putting
   the double root at $y=0$ and the single root at $y=\frac14$, we can
   assume that $f(x_0,y) = y^3 - y^2/4$. If the nodal fiber is at
   $x_0=\infty$, this means that the degree~6 part of $f(x,y)$ is
   $y^3-(xy)^2/4$.  By making substitutions $x\to x+a$ and
   $y\to y+bx +c$ and completing the square, $f(x,y)$ can
   be put in the following form, unique up to rescaling $x,y$ 
   (cf. \cite[Sec. 5]{alexeev17ade-surfaces}):
   \begin{displaymath}
     f =
     y^3 + c'_2 y^2 + c'_1 y 
     - \tfrac14(xy-c'')^2 +
     c_0 + c_1x + c_2x^2 +c_3x^3 + c_4x^4 + c_5x^5.
   \end{displaymath}
   This surface has an $I_n$ fiber iff its discriminant satisfies
   $\mult_{x_0}\Delta(x) = n$. For $x_0=\infty$ this means that
   $\deg\Delta(x) = 12-n$.  Putting $f(x,y)$ in the Weierstrass form
   and computing the normalized discriminant
   $\Delta_n(x)=-2^4(4A^3+27B^2)$, we find the following.
   One has
   $\deg\Delta_n(x)\le 11$ and the coefficient of $x^{11}$ in
   $\Delta_n$ is $c_5$. Thus, the surface $V$ is of type ${\bf E_8}$,
   (i.e. with $I_1$ fiber at $x_0=\infty$) iff $c_5\ne0$, in
   which case we can set $c_5=1$. If $c_5=0$ then
   $\coeff(x^{10},\Delta_n) = c_4$. Thus $V$ has type ${\bf E_7}$
   (i.e. with $I_2$ fiber at $x_0=\infty$) iff $c_5=0$ and $c_4\ne0$, and we can
   set $c_4=1$. This argument continues for ${\bf E_6},\dotsc, {\bf E_3}$.

   For ${\bf E_2}$, one must have $c_5=\dotsb = c_0=0$ and then
   $\coeff({x^5},\Delta_n)=c_1'c''$. We can normalize by setting
   $c''=1$ and take any $c_1'\ne0$.

   For $k=1$ one must have $c_1'c''=0$. Choosing $c'_1=0$ gives
   $\coeff(x^4,\Delta_n)=c_2'(c'')^2$. We normalize by setting $c''=1$
   and $c'_1\ne0$ and call this ${\bf E_1}$. Choosing $c''=0$ gives
   $\coeff(x^4,\Delta_n)=(c_1')^2$. Normalizing $c'_1=1$ gives ${\bf E'_1}$.

   Finally, for $k=0$ one must have
   $c_1'= c_2'=0$, and then we normalize $c''=1$. We call this case ${\bf E_0}$.
   When $c_1'=c''=0$, one has
   $\Delta_n(x)\equiv0$, so all fibers are singular. This is a nonnormal
   surface of type ${\bf C_0}$; one may call the fiber at infinity $I_\infty$.
   
   \begin{table}[ht]
     \caption{Normal forms of rational elliptic surfaces
       with $I_n$ fiber}
     \label{tab:E-normal-forms}
     \centering
     \begin{tabular}[h]{ll|cccccccccll}
       ${\bf S}$&$I_n$ &$c'_2$&$c'_1$&$c''$&$c_0$&$c_1$&$c_2$&$c_3$&$c_4$
       &$c_5$&$x^ny^m$&$G$\\ \hline
       ${\bf E_8}$&$I_1$& $\star$&$\star$&$\star$&$\star$&$\star$&$\star$&$\star$&$\star$&$1$&$x^5$&$\mu_1$\\
       ${\bf E_7}$&$I_2$& $\star$&$\star$&$\star$&$\star$&$\star$&$\star$&$\star$&$1$&&$x^4$&$\mu_2$\\
       ${\bf E_6}$&$I_3$ &$\star$&$\star$&$\star$&$\star$&$\star$&$\star$&$1$&&&$x^3$&$\mu_3$\\
       ${\bf E_5}$&$I_4$ &$\star$&$\star$&$\star$&$\star$&$\star$&$1$&&&&$x^2$&$\mu_4$\\
       ${\bf E_4}$&$I_5$ &$\star$&$\star$&$\star$&$\star$&$1$&&&&&$x$&$\mu_5$\\
       ${\bf E_3}$&$I_6$ &$\star$&$\star$&$\star$&$1$&&&&&&$1$&$\mu_6$\\
       \hline
       ${\bf E_2}$&$I_7$ &$\star$&$\ne0$&$1$&&&&&&&$xy$&$\mu_3$\\
       ${\bf E_1}$&$I_8$ &$\ne0$&&$1$&&&&&&&$xy$&$\mu_3$\\
       ${\bf E'_1}$&$I_8$ &$\star$&$1$&&&&&&&&$y$&$\mu_4$\\
       ${\bf E_0}$&$I_9$ &&&$1$&&&&&&&$xy$&$\mu_3$\\
       ${\bf C_0}$&$I_\infty$&$1$&&&&&&&&&$y^2$&$\mu_2$\\
     \end{tabular}
   \end{table}

   We summarize the results in Table~\ref{tab:E-normal-forms} and
   Fig.~\ref{fig:forms}.  The
   star means the coefficient is arbitrary and we don't write zeros.
   The normal forms of this table are unique up to the subgroup $G$ of
   $(\bC^*)^2$ acting on $x,y$ for which $y^3 + x^2y^2 + x^ny^m$ is
   semi-invariant. The monomial $x^ny^m$ and the group $G$ are given
   the last two columns. Taking the quotient of $\bA^n$, resp. of
   $\bG_m\times\bA^{k-1}$ by $G$ gives the stacks in the statement of
   the theorem. For ${\bf E_2}$ and ${\bf E_1}$, when a $\bG_m$ summand is
   present, the $\mu_3$-action is free.

   \begin{figure}[htp!]
     \centering
     \includegraphics{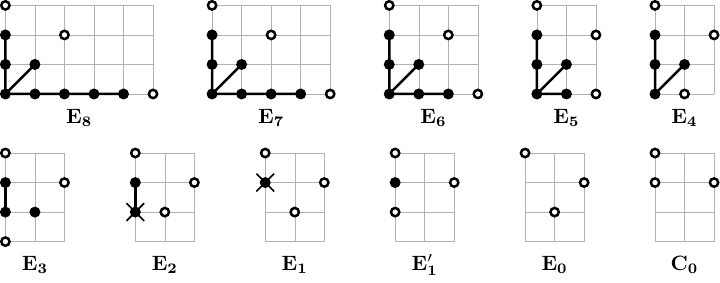}
     \caption{Normal forms of rational elliptic surfaces with $I_n$
       fiber}
     \label{fig:forms}
   \end{figure}
 \end{proof}

 For the Type II strata in $\slcrc$ we have the following:


\begin{theorem}\label{thm:moduli-slc-pair-type2}
  For the irreducible pairs the 
  moduli stack of $\bZ_2$-quotients of the stable pairs by the
  elliptic involution are
  \begin{displaymath}
    {\bf \wE_8}: \  [\Hom(E_8, \cE) :_R W(E_8)], 
    \qquad
   {\bf \wC_{16}}: \ [\Hom(C_{16},\cE) :_R W(C_{16})],
  \end{displaymath}
  where $\cE$ is the universal family of elliptic curves over their moduli
  $j$-stack. For the stable pairs of these types the moduli stack is a
  $\bZ_2$-gerbe over these.

  For the surfaces of type ${\bf \wE_8\wE_8}$ the moduli stack 
  is $[\Hom(E_8^2, \cE) :_R W(E_8^2)\rtimes\bZ_2]$.
\end{theorem}
\begin{proof}
  By Torelli theorem for anticanonical pairs
  \cite{gross2015moduli-of-surfaces, friedman2015on-the-geometry}, for a
  fixed elliptic curve $E$, the moduli of ${\bf \wE_8}$ surfaces is the $:_R$-quotient of
  $\Hom(D^\perp / f, E)$ by the group of admissible monodromies, where $D\sim f$
  is the boundary, a smooth elliptic curve. We have an identification
  $D^\perp/f = \{D, s\}^\perp = E_8$, and the group is
  $W(E_8)$. A surface of type ${\bf \wE_8\wE_8}$ is glued from two such
  surfaces along the boundary $D\simeq E$, so we get the product
  above. The additional $\bZ_2$ is the left-right symmetry. Varying
  the elliptic curve, for the stack we get the same formulas with $E$
  replaced by the universal elliptic curve over the moduli stack of elliptic curves.

  A pair $(X,\epsilon R)$ of type ${\bf \wC_{16}}$ is $\bF_2$ with a
  smooth bisection $D\sim 2s+4f$, an elliptic curve $E$, plus 16
  fibers. The data of the 16 fibers gives a point $(x_1,\dotsc,
  x_{16}) \in E^{16}$ but defined only up to a 2-torsion (an element
  of $E[2]$), permuting the points, and dividing by the elliptic
  involution. One has the exact sequences
  \begin{displaymath}
    0 \to C_{16} \to \bZ^{16}\to \bZ_2\to 0 
    \qquad
    0 \to E[2] \to E^{16} \to \Hom(C_{16}, E) \to 0.
  \end{displaymath}
  Therefore a point $(x_i)$ mod $E[2]$ is an element of $\Hom(C_{16},
  E)$, and we take the $:_R$-quotient of this space by $\bZ_2^{16}\rtimes S_{16} =
  W(C_{16})$. Varying the elliptic curve $E$ gives the same
  formulas with $E$ replaced by the universal family $\cE$.
\end{proof}

\begin{remark}
  The commutative diagram~\eqref{eq:PQ}
  holds with
  $\bC^*$ replaced by an elliptic curve $E$. However, it is no longer
  true in general that $W_x = W(R_x)$, see e.g. the example in
  \cite[3.6]{looijenga1976root}. The Chevalley-Shephard-Todd's theorem
  implies that one has $W_x=W(R_x)$ for all $x$ iff the quotient
  $\Hom(P,E)/W$ is smooth.  By Looijenga
  \cite{looijenga1976root}, $\Hom(P,E)/W$ is a weighted projective
  space with the weights equal to the coefficients of highest root,
  and 1. It is frequently singular, e.g. for $R=D_n$ ($n\ge4$) and
  $E_n$ ($n=6,7,8$). Then the
  coarse quotient $\Hom(P,E)/W$ is singular but the stack
  $[\Hom(P,E):_R W]$ is smooth since it has an \'etale cover by the
  local quotients $U_x/W(R_x)$ as in \eqref{not:R-quotient}, which are
  smooth. Thus, in this case the weighted projective spaces may be
  considered as smooth orbifolds instead of as singular varieties.

  For $E_8$, $\Hom(P,E)/W = \bP(1,2^2,3^2,4^2,5,6)$ and for $C_n$ it
  is $\bP(1^2,2^{n-1})$.  Thus, the alternative description of the
  boundary divisors is as a family of stacky weighted projective
  spaces over the $j$-stack of elliptic curves.
 \end{remark}

\subsection{Proof of main theorem}
\label{sec:stable-models}

In this section we assemble the inputs necessary to apply Theorem \ref{threeinputs}.
First, we must show:

\begin{proposition}\label{takestable}
Let $X(\lambda)\rightarrow (C,0)$ be a divisor model with monodromy
invariant~$\lambda$. The stable model $\oX(\lambda)$
(cf. Definition \ref{stabletype})
has stable type gotten from
the combinatorial type (cf. Notation \ref{d-notation}) of the cone containing $\lambda$.

Furthermore, it is possible to vary $X(\lambda)\rightarrow (C,0)$ so that
any stable surface of the given combinatorial type
is realized as the stable model $\oX(\lambda)$.
  \end{proposition}
  
  \begin{proof} The first statement
  follows from seeing which curves are contracted by
  the linear system of $L_n:= n(s+m\sum f_i)$ for $n\geq 4$
  on $X_0(\lambda)$. A curve $Z\subset X_0(\lambda)$ is contracted iff
  $L_n\cdot Z=0$. Thus the stable model $\oX_0(\lambda)$ is the result of:
  (1) contracting the vertical ruling on all components $V_i$ not containing the section,
  then (2) contracting the components $V_i$ containing the section but no marked fibers
  along the horizontal ruling. The resulting surface $\oX_0(\lambda)$ has the
  stable type ${\bf S_1\cdots S_N}$ associated to the cone containing $\lambda$.
  
 We now prove the second statement. First observe that the lattice $\Lambda$
  of Definition \ref{lambdalatt} is exactly given by the direct sum
 $$\Lambda = \oplus_i \,(A\textrm{ or }D\textrm{ or }E)_{n_i}$$ corresponding to the components along
 the top edge of $P_{LR}(\ell)$, i.e. the summands $\Lambda_i$
 of $\Lambda$ are in fact the character lattices associated to the corresponding
 symbol ${\bf S_i}$ of the stable type, except for switching ${\bf C}$ with $D$,
 see Remark \ref{DC-remark}.
    
   By Remark \ref{surjperiods}, there is an elliptically
 fibered $d$-semistable surface $X_0(\lambda)$
 with period map $\psi\,:\,\overline{\mathbb{L}}/ \Z f\oplus \Z s \to \C^*$ realizing
 any element $\psi\big{|}_{\Lambda}\in \Hom(\Lambda,\C^*)$ and hence any period
 point of the corresponding anticanonical pair $(V_i,\sum_j D_{ij})$,
subject to the condition that if $V_i$ is an $X$-type end, it is elliptically fibered.
 
 The element $\psi\big{|}_\Lambda$ determines uniquely the locations of the very
 singular fibers of $X_0(\lambda)$ in exactly the same manner that a point in the
 torus $\Hom(\Lambda_i,\C^*)$ determines the modulus of a stable surface:
 For the singularity $v_i= I_{n+1}$
 the relative location of two
 very singular fibers of $X_0(\lambda)$ containing the exceptional curves
 $E_1$ and $E_2$
 on the component $V_i$ is
$\psi(E_1-E_2)\in \C^*$ and hence $\psi\big{|}_{\Lambda_i}\in \Hom(A_n,\C^*)$
determines the relative locations of the very singular fibers intersecting $V_i$.
A similar computation
holds for type $Y_2Y_{2+k}$ and $Y_2Y_2'$ and $\Hom(D_n,\C^*)$.
By definition, the period point of an
elliptically fibered anticanonical pair of type 
$X_{k+3}$ lies in $\Hom(E_n,\C^*)$. It (inexplicitly) determines the location
of the singular fibers.

Finally, by Proposition \ref{limit}, the very singular fibers on $X_0(\lambda)$
are the limits of singular fibers
of the elliptic fibration on the general fiber. These curves contract to
the limits of the singular fibers on the stable model. So
 the restricted period point $\psi\big{|}_\Lambda\in \Hom(\Lambda,\C^*)$
 is compatible with the computation of stable moduli
made in Theorem \ref{thm:moduli-slc-pair}.\end{proof}

\begin{lemma}\label{diminput} 
  The dimensions of a stratum of $\torrc$ and the dimension of
  the corresponding moduli space of stable surfaces of fixed type
  are equal. For Type III strata, the former is equal to
  $20-\text{(length of its combinatorial type)}$, and the latter to
  the sum of its Dynkin indices.
\end{lemma}


\begin{proof}
  The dimension of a stratum of the toroidal compactification is the
  codimension of the corresponding cone. The dimensions of strata in
  $\slcrc$ are computed in Theorem~\ref{thm:moduli-slc-pair}.
  For type III cones the codimension of the cone and the dimension of
  the corresponding stable stratum both equal to the sum of the indices $n_i$
  in the label $(E_{n_0}|E'_1|D_{n_0}|D_0') A_{n_1}\dotsc
  A_{n_k}(E_{n_{k+1}}|E'_1|D_{n_{k+1}}|D_0')$, resp. with $D$ replaced
  by ${\bf C}$ and all letters bolded.


  For the Type II cones $\wE_8\wE_8$ and
  $\wD_{16}$ the strata in
  $\torrc$ are divisors, and the dimensions of their image strata
  ${\bf \wE_8\wE_8}$, ${\bf \wC_{16}}$ in $\slcrc$ are $8+8+1=16+1=17$.
\end{proof}

\begin{theorem}\label{thm:rc-comp}
  The normalization of the stable pair compactification $\slcrc$ is
  the toroidal compactification $\torrc$. 
\end{theorem}

\begin{proof} We apply Theorem \ref{threeinputs} to the case at hand.
Taking the divisor model
$X(\lambda)$ of Definition \ref{monodivisor} gives input {\bf (div)} for the integer $n=6$.
Proposition \ref{dssinput} implies input {\bf (d-ss)}. Next, the first part of Proposition 
\ref{takestable} gives input {\bf (fan)}. By {\bf (div)} and {\bf (d-ss)},
all strata of stable surfaces are been enumerated.
Thus, input {\bf (qaff)} reduces to
Corollary~\ref{cor:affine}. 
We conclude that there is a morphism $\torrc\to \slcrc$.

Furthermore, this morphism sends toroidal strata to the strata of the corresponding
stable types. Thus, the additional condition {\bf (dim)} follows from Lemma~\ref{diminput}
if we can prove that the morphisms on strata surject onto the
moduli spaces of stable pairs. This follows from the second part of
Proposition \ref{takestable}, as the restricted period $\psi\big{|}_\Lambda$
encodes the image of $0$ under the period map $(C,0)\to \torrc$.
Corollary \ref{normalization} implies the theorem.
\end{proof}
 
\begin{question} Having described the normalization of the stable pair compactifications for $R^\ram$ and $R^\rc$ it is natural to ask: Is the normalization of the compactification for $tR^\ram+(1-t)R^\rc$ toroidal for all $t\in [0,1]$? At what values of $t$ does the compactification change, and how? \end{question}

\subsection{The normalization map}
\label{sec:normalization}

Let ${\bf S_1\cdots S_N}$ be a Type III stable type. By
Thm.~\ref{thm:moduli-slc-pair} the stratum in $\slcrc$ of stable pairs
$(X,\epsilon R)$ of this type is
\begin{displaymath}
  (T / W^{ACE}) / G_{\LR}
\end{displaymath}
where $\Lambda^{ACE} = \oplus_{i=1}^n\Lambda_i$ is the sum of the $ACE$
lattices of ${\bf S_i}$-type, $T = \Hom(\Lambda, \bC^*)$ is the
corresponding torus, $W^{ACE} = \oplus W(\Lambda_i)$ is the Weyl group,
and $G_\LR = \bZ_2$ if
the type is left-right symmetric and trivial otherwise.

Recall once again that the $C_n=D_n$ as lattices but
$W(C_n)/W(D_n)=\bZ_2$.

\begin{definition}
  For a stable type ${\bf S_1\dots S_N}$ we have an embedding of
  the corresponding $ADE$  lattice
  $\Lambda \subset \IIell$: the lattices $\Lambda_i$ are generated by
  the explicit elements of $\IIell$, the roots $\alpha_i$ and the
  vectors $\beta_L$, $\beta_R$, $\gamma_L$, $\gamma_R$. The generators
  of the $E_1$ and $D_1$ lattices are $\beta$ and $\gamma$ respectively.
  We denote by $\oLambda$ the saturation of $\Lambda$ in $\IIell$, and
  by $\Tsat = \Hom(\oLambda,\bC^*)$ the corresponding torus.
\end{definition}

\begin{theorem}\label{thm:map-strata}
  For the type III strata in $\torrc$ and $\slcrc$ the following holds:
  \begin{enumerate}
  \item The only strata of $\torrc$ glued by the normalization
    morphism $\torrc\to\slcrc$ are the strata $D_0\cdots$,
    $D'_0\cdots$ (on either left and right ends) both mapping to
    the ${\bf C_0}\cdots$ stratum of $\slcrc$.
  \item For a cone $\sigma$ of the fan $\cF_\rc$ with stable type
    ${\bf S_1\cdots S_N}$, the corresponding stratum in $\torrc$ is
    \begin{math}
      (\Tsat / W^{ADE}) / G^\cF_{\LR},
    \end{math}
    where $G^\cF_\LR = \bZ_2$ or $1$ depending on whether the cone
    $\sigma$ is left-right symmetric or not, i.e. $\sigma$ and
    $\iota\sigma$ are in the same $W(\IIell)$-orbit for the involution
    $\iota$, $O^+(\II_{1,17}) / W(\IIell) = \la\iota\ra$.
  \item The map of strata
    \begin{displaymath}
      (\Tsat / W^{ADE}) / G^\cF_{\LR} \longrightarrow
      (T / W^{ACE}) / G_{\LR}
    \end{displaymath}
    is defined by the homomorphism of tori $\Tsat\to T$, dual to the
    lattice embedding $\Lambda\to\oLambda$ and by the $2:1$ map for
    each $D_n\to {\bf C_n}$ type with $n>0$. 
  \end{enumerate}
\end{theorem}
\begin{proof}
  (1) follows from Def.~\ref{stabletype}.

  (2) The stratum in $\torrc$ is the the torus orbit corresponding to
  $\sigma$, which is $\Tsat$ as defined, modulo the stabilizer of
  $\sigma$ in $O^+(\IIell)$, equal to the stabilizer of $\sigma$ in
  $W(\IIell)$ plus the involution $\iota$ if the cone $\sigma$ is
  symmetric.  $\Stab_{W(\IIell)}(\sigma)$ is the stabilizer of the
  minimal Coxeter cone containing it.

  We observe that for each of the
  cones with the end behavior $E_1$, $E'_1$, $E_2$ the
  stabilizer is the same as the Weyl group of the lattice for its
  stable type ${\bf E_1}$, ${\bf E_1'}$, ${\bf E_2}$, as given in
  Table~\ref{tab:small-DE}.
  For the cones $E_0$, $D_0$, $D'_0$
  with stable types ${\bf E_0}$, ${\bf C_0}$ the stabilizers are trivial.  The
  other cones of $\cF_\rc$ are already Coxeter cones and for them the
  stabilizer is obviously the corresponding Weyl group.

  (3) As in the proof of Theorem~\ref{threeinputs}, we pick a
  monodromy invariant $\lambda\in\sigma^0$ in the interior of the cone
  and consider a family of divisor models over the partial toroidal
  compactification $\oF^\lambda$ with a boundary divisor
  $\Delta$.  The space $\oF^\lambda$ is an open subset of the blowup
  of $\torrc$ at the stratum corresponding to $\lambda$. In terms of
  the character groups this gives embeddings
  \begin{math}
    \sigma^\perp \to \lambda^\perp \to \IIell.
  \end{math}

  On the other hand, as in Section~\ref{sec:dss} there is a period map
  $\Delta \to \Hom(\bL, \bC^*)$, where $L=\ker\big( \oplus_i \Pic V_i
  \to \oplus_{i<j} \Pic D_{ij} )$ and
  $\bL = \coker(\Xi \to L)$. In terms of the character lattices it
  corresponds to the homomorphism
  \begin{math}
    \bL \to \lambda^\perp.
  \end{math}

  As in the proof of Prop.~\ref{takestable}, the composition of this
  period map and the projection to the periods of the irreducible
  components of $(X_0,R_0)$ is given by the embedding of the character
  lattices
  \begin{math}
    \Lambda = \oplus\Lambda_i \to \bL.
  \end{math}
  Putting this together, we have homomorphisms
  \begin{displaymath}
    \sigma^\perp \to \lambda^\perp \to \IIell
    \quad\text{and}\quad
    \Lambda \to \bL \to \lambda^\perp.
  \end{displaymath}
  
  For a one-parameter degeneration $(X,R)\to (S,0)$ of K3 surfaces the
  period point of the central fiber  $X_0$ over $\Delta\subset\oF^\lambda$
  is determined by the limit mixed Hodge structure.  By
  \cite[Prop. 3.4]{friedman1986type-III} the map
  $\Delta\to\Hom(\bL,\bC^*)$ is defined by the mixed Hodge structure
  of $X_0$. It follows that the map of strata is given by the map of
  tori with the character groups $\Lambda\to\sigma^\perp\cap\IIell$. 

  By comparing the dimensions of the spaces, it follows that the image
  of $\Lambda\otimes\bR$ in $\lambda^\perp \subset \IIell\otimes\bR$
  is $\sigma^\perp$ and so $(\im\Lambda)^{\rm sat} =
  \sigma^\perp\cap\IIell = \oLambda$. 
\end{proof}

It remains to find the saturation $\oLambda$. This is enough to do for
the cones with end behavior 1 and 3, since the strata for the end
behaviors 2 and 3 are identified by the map $\torrc\to\slcrc$. For
these cones, the description is given by the next lemma (with a
trivial proof), which we apply to the vectors
$-\beta_L,\alpha_2, \alpha_3\dotsc$,
resp. $\alpha_2,\gamma_L,\alpha_1,\dotsc$ that satisfy the linear
relations~\eqref{eq:lin-relations}.

\begin{lemma}\label{lem:saturation}
  Suppose that vectors $v_1,\dotsc,v_{19}$ generate $\IIell$ with a
  single linear relation $\sum_{i=1}^{19} n_iv_i = 0$, $n_i\in\bZ$,
  $\gcd(n_1,\dotsc,n_{19})=1$. For a subset $I\subset\{1,\dotsc,19\}$
  let $\Lambda = \la v_i,\ i\in I\ra$. Then $\oLambda / \Lambda =
  \bZ/d\bZ$, where $d = \gcd(n_j,\ j\not\in I)$. 
We use the convention that $\gcd(0,\dotsc,0) = 1$.
\end{lemma}

Finally, we give a description of the normalization map for the Type
II strata. 

\begin{theorem}\label{thm:map-strata-type2}
  The $\wE_8\wE_8$ stratum of $\torrc$ maps to the
  ${\bf \wE_8\wE_8}$ stratum of $\slcrc$ isomorphically. For
  $\wD_{16}\to{\bf \wC_{16}}$, the map of the strata has degree 8 and it is
  \begin{displaymath}
    [\Hom(D_{16}^+, \cE) :_R W(D_{16})] \to 
    [\Hom(C_{16}, \cE) :_R W(C_{16})]
  \end{displaymath}
  where $\cE$ is the universal elliptic curve over the $j$-stack and
  $D_{16}^+ = \II_{0,16}$.
\end{theorem}
\begin{proof}
  The $1$-cusps of the Baily-Borel compactification $\oF^{\rm BB}$
  correspond to the primitive isotropic planes
  $J\subset\II_{2,18}$. One has
  $\II_{2,18}\simeq J\oplus\Lambda\oplus \hJ$ for the unimodular
  lattice $\Lambda=J^\perp/J$, and the respective stratum in $\torrc$
  is (the coarse moduli space of) $[\Hom(\Lambda,\cE) : O(\Lambda)]$,
  cf. \cite{ash1975smooth-compactifications, clingher07on-k3-surfaces}. 
  
  For $\wE_8^2$ one has $\Lambda=E_8^2$ and
  $O(E_8^2)=W(E_8^2)\rtimes\bZ_2$, so we get the same strata in
  $\torrc$ and $\slcrc$ by Theorem~\ref{thm:moduli-slc-pair-type2}.
  For $\wD_{16}$ one has $\Lambda = D_{16}^+ = \II_{0,16}$, a $2:1$
  extension of $D_{16}$, and $O(D^+_{16}) = W(D_{16})$, an index 2
  subgroup of $W(C_{16})$. So the map of strata is a composition of
  quotients by $\cE[2]$ and $\bZ_2$ and it has degree $4\cdot 2=8$.
\end{proof}


\bibliographystyle{amsalpha}

\def\cprime{$'$}
\providecommand{\bysame}{\leavevmode\hbox to3em{\hrulefill}\thinspace}
\providecommand{\MR}{\relax\ifhmode\unskip\space\fi MR }
\providecommand{\MRhref}[2]{%
  \href{http://www.ams.org/mathscinet-getitem?mr=#1}{#2}
}
\providecommand{\href}[2]{#2}

\end{document}